\documentclass{amsart}

\usepackage[letterpaper]{geometry}

\usepackage{amsmath, amsthm, amsfonts, amsbsy, thmtools, amssymb,tikz,hyperref,cleveref}
\usepackage[mathscr]{euscript}
\usetikzlibrary{arrows}

\usepackage{graphicx}

\usetikzlibrary{graphs,patterns,decorations.markings,arrows,matrix}
\usetikzlibrary{calc,decorations.pathmorphing,decorations.pathreplacing,shapes}

\renewcommand{\tikz}[2]{
	\begin{tikzpicture}[scale=#1,baseline=(current bounding box.center),>=stealth]
		#2
\end{tikzpicture}}

\colorlet{lgray}{white!85!black}

\numberwithin{equation}{section}

\newtheorem{thm}{Theorem}[section]
\newtheorem{prop}[thm]{Proposition}
\newtheorem{lem}[thm]{Lemma}
\newtheorem{cor}[thm]{Corollary}

\theoremstyle{remark}
\newtheorem{rem}[thm]{Remark}
\theoremstyle{definition}
\newtheorem{definition}[thm]{Definition}

\newtheorem{example}[thm]{Example}

\DeclareMathOperator{\FV}{FV}
\DeclareMathOperator{\CV}{CV}
\DeclareMathOperator{\MM}{MM}

\DeclareMathOperator{\PV}{PV}
\DeclareMathOperator{\HL}{HL}
\DeclareMathOperator{\SM}{SM}
\DeclareMathOperator{\TW}{TW}

\begin{document}
	
	\title{Deformed polynuclear growth in (1+1) dimensions}
	\author{Amol Aggarwal, Alexei Borodin, and Michael Wheeler}

	\begin{abstract} 
		
		We introduce and study a one parameter deformation of the polynuclear growth (PNG) in (1+1)-dimensions, which we call the $t$-PNG model. It is defined by requiring that, when two expanding islands merge, with probability $t$ they sprout another island on top of the merging location. At $t=0$, this becomes the standard (non-deformed) PNG model that, in the droplet geometry, can be reformulated through longest increasing subsequences of uniformly random permutations or through an algorithm known as patience sorting. In terms of the latter, the $t$-PNG model allows errors to occur in the sorting algorithm with probability $t$.
		
		We prove that the $t$-PNG model exhibits one-point Tracy--Widom GUE asymptotics at large times for any fixed $t\in [0,1)$, and one-point convergence to the narrow wedge solution of the Kardar--Parisi--Zhang (KPZ) equation as $t$ tends to $1$. We further construct distributions for an external source that are likely to induce Baik--Ben Arous--P\'{e}ch\'{e} type phase transitions. The proofs are based on solvable stochastic vertex models and their connection to the determinantal point processes arising from Schur measures on partitions. 
	\end{abstract}

	\maketitle 
	
    \setcounter{tocdepth}{1}
    \tableofcontents
	
	\section{Introduction}
	\label{Intro} 
	
	The process of \emph{polynuclear growth} (PNG, for short) is a mathematical model for randomly growing interfaces. If the space is one-dimensional, it can be described as follows; see the book of Meakin \cite{Meakin} for a broader context. The interface is represented by a continuous broken line in a plane that consists of horizontal linear segments and height 1 up or down steps between them. As time progresses, the up and down steps move with speed 1 to the left and to the right, respectively; this is interpreted as lateral growth of islands that form on the interface. When a left-moving up step and a right-moving down step meet, they disappear, which corresponds to merging of neighboring islands. 
	In addition to that, new islands are randomly created by adding an up step and a down step separated by infinitesimal distance (that immediately starts growing). The creation, or \emph{nucleation} events are space-time uncorrelated, which is represented by their space-time locations forming the two-dimensional Poisson process with intensity 1. 
	
	We will be interested in the so-called \emph{droplet} PNG, where initially the interface is perfectly flat, and all nucleation events take place in the light cone of the origin $\big\{ (x,\tau)\in \mathbb{R}\times \mathbb{R}_+: |x|<\tau \big\}$. Three successive snapshots of this process are depicted in Figure \ref{fig1}, where the growing interface is pictured on top in red, and the cone below is the space-time locus of the nucleation events symbolized by dots. The horizontal line that runs through the cones indicates the value of time $\tau$ at which the interface is drawn, with the corner corresponding to $\tau=x=0$. The full animation, created by Patrik Ferrari, can be found on his webpage \texttt{https://wt.iam.uni-bonn.de/ferrari/research/animationpng}.
	
	\begin{figure}
		\label{fig1}
		\includegraphics[width=\textwidth]{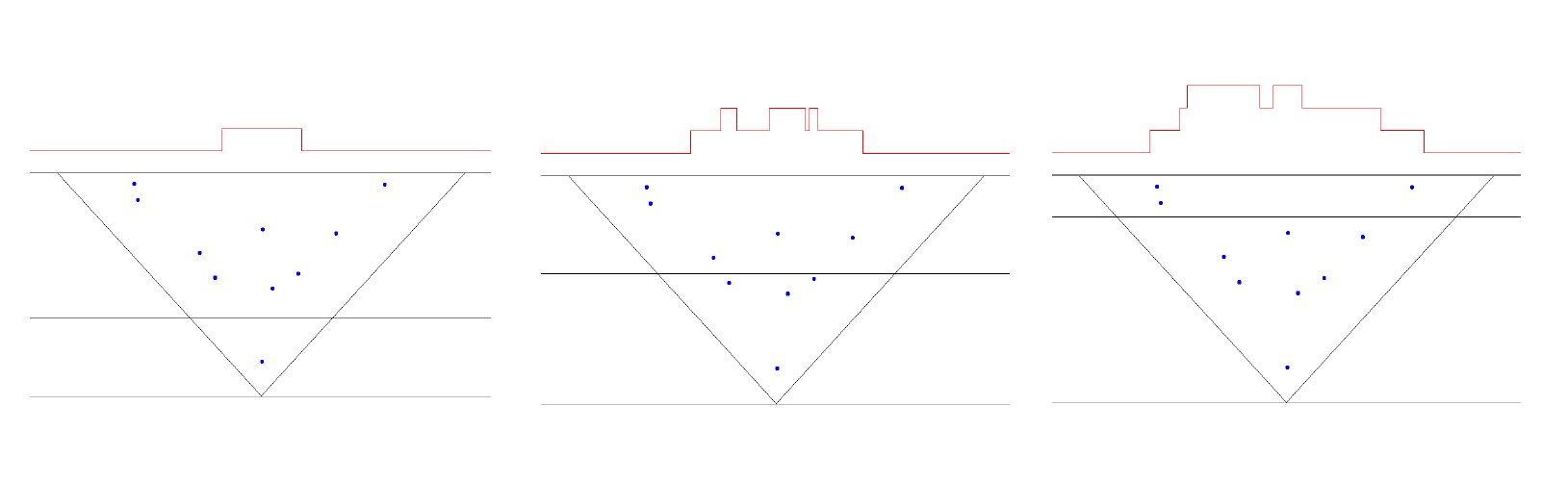}
		\caption{Snapshots of the droplet PNG in (1+1) dimensions.}
	\end{figure}
	
	We will view the interface of the droplet PNG as the graph of a function called the \emph{height function}, and denote it by $\mathfrak H(x,\tau)$; we assume that initially $\mathfrak H(x,0)\equiv 0$. It is not difficult to show, cf. Pr\"ahofer--Spohn \cite{PS-PRL}, that for $|x|<\tau$, $\mathfrak H(x,\tau)$ is equidistributed with the length of the longest increasing subsequence of the uniformly random permutation of size $n$, where $n$ is itself an independent Poisson-distributed random variable with parameter $\frac{1}{2}(\tau^2-x^2)$.\footnote{This random variable counts the number of nucleation events that affect $\mathfrak H(x,\tau)$, and its expectation is the area of the rectangle with opposite corners at $(0,0)$ and $(x,\tau)$ and sides parallel to the cone walls.} Two decades ago, breakthrough results by Baik--Deift--Johansson \cite{DLLSRP} on asymptotic fluctuations of the length of the longest increasing subsequences of random permutations,  and by Johansson \cite{Joh-shape} on asymptotic fluctuations of the totally asymmetric simple exclusion process (TASEP), opened the gates towards understanding a close relationship between such (1+1)d random growth models and random matrix type ensembles, cf. the survey of Ferrari--Spohn \cite{FS-survey}. The height function $\mathfrak H(x,\tau)$ also admits an interpretation through an algorithm called \emph{patience sorting}; see the survey \cite{LSPST} of Aldous--Diaconis (and Appendix \ref{ModelSort} below).

	Both PNG and TASEP belong to the so-called (conjectural) Kardar--Parisi--Zhang (KPZ) universality class of random growth models, named after the authors of seminal work \cite{KPZ}. Another member of this class is the KPZ stochastic partial differential equation introduced in the same paper. It is more difficult to analyze, and arguably the best known way to understanding large time asymptotics of this equation is through two one-parameter deformations of the TASEP, namely, the (partially) asymmetric exclusion process (ASEP) dating back to the work of Spitzer \cite{Spitzer} and Macdonald--Gibbs--Pipkin \cite{MGP}, and the more recent $q$-TASEP introduced by Borodin--Corwin in \cite{P}. Remarkable analysis of the ASEP by Tracy-Widom \cite{TW} led to finding the form and asymptotics of certain solutions of the KPZ equation by Amir--Corwin--Quastel \cite{ACQ} and Sasamoto--Spohn \cite{SS1}. An alternative and non-rigorous approach to such solutions via 1d delta-interaction Bose gas and replica by Dotsenko \cite{Dotsenko} and Calabrese--Le Doussal--Rosso \cite{CDR}, was regularized by means of the $q$-TASEP in \cite{P}. 
	
	Despite the substantial progress in this area that ensued, no analogously simple deformation of the PNG process has been described so far, to the best of our knowledge. The goal of this work is to present one. It would be fitting to use the term ``$q$-PNG'' for such a deformation. However, in what follows we choose a different letter $t$ to denote the deformation parameter, because of its tight connection to a similarly named parameter in the theory of symmetric functions. Correspondingly, we will speak of a $t$-PNG below. The value of $t=0$ corresponds to the standard (non-deformed) PNG process that was described above.    
	
	The definition of this (droplet) $t$-PNG model is very similar to the non-deformed one. The only difference is in what happens when a right-moving up step and a left-moving down step meet. We now stipulate that, with probability $1-t$, they disappear as before (in which case, the corresponding islands simply merge). With the complementary probability $t$, simultaneously with the merging, another island of infinitesimal size is created on top of the merging place. In other words, with probability $t$ another nucleation is added at the space-time location of the merging event. In yet another interpretation, if we speak in the language of rays in space-time formed by the moving up/down steps (this is the language we use in the text below), merging corresponds to annihilation of two rays at their intersection, while merging with nucleation corresponds to those rays moving through each other despite their intersection. See Figure \ref{model1} below for an illustration of the behavior of these rays. There is also a concise description of the corresponding deformation of the patience sorting -- one needs to introduce independent errors occuring with probability $t$ in choosing which pile to place a card onto; see Appendix \ref{ModelSort} below. 
	
	We prove two limiting statements about large time asymptotic behavior of the height function of the $t$-PNG model at a single point. 
	In order to state them, it is more convenient to set the intensity of the Poisson process of the nucleation events to be $1-t$. Let us denote the corresponding height function by $\mathfrak{H}_t$. 
	
	First, we prove, in Theorem \ref{hxnynlimit} below, that
	\begin{flalign*}
		\displaystyle\lim_{\tau^2-x^2 \rightarrow \infty} \mathbb{P} \bigg[ \displaystyle\frac{\mathfrak{H_t} (x, \tau) - (\tau^2-x^2)^{1/2}}{2^{-1/3}(\tau^2-x^2)^{1/6}} \le s \bigg] = F_{\TW} (s),
	\end{flalign*} 
	\noindent where $F_{\TW}$ denotes the Tracy--Widom Gaussian Unitary Ensemble (GUE) distribution. At $t=0$ this coincides with the (Poissonized version of the) Baik--Deift--Johansson theorem, and the only $t$-dependence in the statement is in the definition of $\mathfrak{H}_t$.
	 
    The second claim is convergence to a solution of the KPZ equation. To that end, we choose a small parameter $\varepsilon>0$ and set $t=\exp(-\varepsilon)$. Then, assuming $|x|<\tau$, the normalized and centered height function 
   \begin{flalign*} 
   	\varepsilon \big( \mathfrak{H}_t (\varepsilon^{-3}x, \varepsilon^{-3}\tau) - \varepsilon^{-3} (\tau^2-x^2) \big) - \log \varepsilon
   \end{flalign*}
   weakly converges, as $\varepsilon \to 0$, to $\frac{T}{24} - \mathcal{H}_T(0)$, where $T=\tau^2-x^2$, and $\mathcal{H}_T(X)$ denotes the (Cole--Hopf) solution of the KPZ equation with narrow wedge initial data at time $T$ and position $X$. This is the subject of Theorem \ref{equationlimit} below. 
   
   We also describe a family of distributions for additional nucleation events along one of the boundaries of the cone 
   $|x|<\tau$ that are likely to induce what is known as the Baik-Ben Arous-P\'{e}ch\'{e} type phase transition \cite{PTLECSCM} for $\mathfrak H_t$. Such phase transitions for the $q$-TASEP, ASEP, and KPZ equation were described by Barraquand \cite{Barraquand}, Aggarwal--Borodin \cite{PTASSVM}, and Borodin--Corwin--Ferrari \cite{FEF},  respectively. Additional nucleations at the two walls of the cone are often referred to as \emph{external sources}, cf. Baik--Rains \cite{LDPGMES} and Imamura--Sasamoto \cite{FPGMES}. 
   
   We expect that the $t$-PNG model should admit further results, such as multi-point convergence to the Airy$_2$ process (following the ideas of Vir\'{a}g \cite{Virag}); multi-point convergence to the narrow wedge solution of the KPZ equation (see, e.g., Corwin--Ghosal--Shen--Tsai \cite{CGST} and references therein); introduction of the second external source and description of the stationary growth (cf. Aggarwal \cite{Agg-ASEP}); adding colors to the model in such a way that it remains integrable (cf. Borodin--Wheeler \cite{SVMST}); and extending the model to multiple layers via RSK-type algorithms (cf. Pr\"ahofer--Spohn \cite{SIDP} in the $t=0$ case). However, we will not pursue these directions in the present text.   

   Our proofs are based on relatively recent techniques of solvable stochastic lattice models and their relation to the theory of symmetric functions. The $t$-PNG model arises as a certain limit of a fully fused $U_t(\widehat{\mathfrak{sl}}_2)$ stochastic vertex model in a quadrant. Such models are known to be related to \emph{Macdonald measures} on partitions \cite{P} in two different ways: (a) the two have equal averages of certain observables, and (b) the height function of the vertex models is equidistributed with the length of the corresponding random partitions for the Hall--Littlewood measures; see Borodin \cite{SHSVMM} and Borodin--Bufetov--Wheeler \cite{SSVMP}. Applying (a) to connect to the Schur measures, we are able to deduce our limit results from the Airy asymptotics of the determinantal point processes related to the non-deformed PNG process (and to the Plancherel measure on partitions), which has been well understood since 
   the works of Baik--Deift--Johansson \cite{DLLSRP}, Borodin--Okounkov--Olshanski \cite{AMSG}, and Johansson \cite{DOPM}. Applying (b) tells us that, in the language of symmetric functions, our construction of the $t$-PNG model corresponds to passing from Hall--Littlewood measures to those related to \emph{modified} Hall--Littlewood polynomials and further considering Plancherel specializations of those. This should be compared to the role of the $q$-Whittaker measures for the $q$-TASEP, see \cite{P}, and Hall--Littlewood measures for the ASEP, see Bufetov--Matveev \cite{BM}. In fact, it is the focus on the modified Hall--Littlewood polynomials, which played an important role in our recent work \cite{ABW}, that led us to the deformed PNG.   
	
	The remainder of this text is organized as follows. In \Cref{Vertex} we recall the definitions and various properties of the $U_t ( \widehat{\mathfrak{sl}}_2)$ fused stochastic higher spin vertex models and their relation to Macdonald measures. In \Cref{WeightsFused} we analyze certain limits of these fused weights, which will give rise to the $t$-PNG model (possibly with boundary conditions) in \Cref{Modelq}. In \Cref{Limit} we use a matching result between the $t$-PNG model and the Poissonized Plancherel measure to derive asymptotic results concerning the $t$-PNG model, including its large scale fluctuations and its scaling limit to the Kardar--Parisi--Zhang (KPZ) equation. In \Cref{ModelSort} we provide an alternative interpretation for the $t$-PNG model through patience sorting; in \Cref{ConvergeModel} we provide a careful proof of how the $t$-PNG model appears as a limit of a stochastic fused vertex model; and in \Cref{EquationProof} we provide an alternative proof of an expectation matching result (\Cref{pvm} below) that we use. 
	
	In what follows, we denote the \emph{q-Pochhammer symbol} $(a; q)_k = \prod_{j = 0}^{k - 1} (1 - q^j a)$, for any complex numbers $a, q \in \mathbb{C}$ and integer $k \ge 0$. 
	
	\subsection*{Acknowledgments}
	
	Amol Aggarwal was partially supported by a Clay Research Fellowship. Alexei Borodin was partially supported by the NSF grants DMS-1664619, DMS-1853981 and the Simons Investigator program. Michael Wheeler was supported by an Australian Research Council Future Fellowship, grant FT200100981.  
	
	\section{Miscellaneous Preliminaries}
	
	\label{Vertex}
	
	In this section we collect various miscellaneous results concerning vertex models and Macdonald measures. In \Cref{PathEnsembles} we recall the definition of the fused stochastic higher spin vertex models from \cite{SHSVML,HSVMSRF} and the notion of fusion. In \Cref{MeasureModel} we recall matching results from \cite{SHSVMM} between these stochastic vertex models and certain Macdonald measures.
	
	\subsection{Fused Stochastic Higher Spin Vertex Model}
	
	\label{PathEnsembles}

	The vertex models we consider will be probability measures on ensembles of directed up-right paths\footnote{We will later ``complement'' these paths in a way that changes their orientations from up-right to up-left.} on the positive quadrant $\mathbb{Z}_{> 0}^2$ that emanate from the $x$ and $y$ axes; see the right side of \Cref{arrows} for an example. The specific forms of these probability measures are expressed through weights associated with each vertex $v \in \mathbb{Z}_{> 0}^2$. These weights will depend on the \emph{arrow configuration} of $v$, which is a quadruple $(i_1, j_1; i_2, j_2) = (i_1, j_1; i_2, j_2)_v$ of non-negative integers. Here, $i_1$ counts the number of paths vertically entering through $v$. In the same way $j_1$, $i_2$, and $j_2$ count paths horizontally entering, vertically exiting, and horizontally exiting through $v$, respectively.  An example of an arrow configuration is depicted on the left side of \Cref{arrows}. 
	
	Assigning values $j_1$ to vertices on the line $(1, y)$ and values $i_1$ to vertices on the line $(x, 1)$ can be viewed as imposing boundary conditions on the vertex model. If for some sequence $\boldsymbol{J} = (J_1, J_2, \ldots )$ of nonnegative integers we have $j_1 = J_y$ at $(1, y)$ and $i_1 = 0$ at $(x, 1)$ for each $x, y > 0$, then $J_y$ paths enter through each site $(0, y)$ of the $y$-axis, and no paths enter through any site of the $x$-axis. We will refer to this particular assignment as \emph{$\boldsymbol{J}$-step boundary data}; in the special case when $\boldsymbol{J} = (1, 1, \ldots )$, it will be abbreviated \emph{step boundary data}. See the right side of \Cref{arrows} for an example when $\boldsymbol{J} = (2, 2, \ldots )$. In general, we will refer to any assignment of $i_1$ to $\mathbb{Z}_{> 0} \times \{ 1 \}$ and $j_1$ to $\{ 1 \} \times \mathbb{Z}_{> 0}$ as \emph{boundary data}, which can be deterministic (like $\boldsymbol{J}$-step) or random. 
	
	In addition to depending on the arrow configuration $(i_1, j_1; i_2, j_2)$, the vertex weight at $v \in \mathbb{Z}^2$ will also be governed by several complex parameters. The first among them consist in two pairs of \emph{rapidity parameters} $(u; r)$ and $(\xi; s)$, which are associated with the row and column intersecting to form $v$, respectively; these rapidities $(u; r)$ and $(\xi; s)$ may vary across the domain but remain constant along rows or columns, respectively. The last is a \emph{quantization parameter}\footnote{In the framework of vertex models, this parameter is typically denoted by $q$. However, it will eventually match with the parameter denoted by $t$ in the context of Macdonald polynomials, and so we use the notation $t$ here.} $t$, which cannot vary and is fixed throughout the model. This produces five governing parameters, but the vertex weight will in fact only depend on $u$ and $\xi$ through their quotient $z = \frac{u}{\xi}$, which is sometimes referred to as a \emph{spectral parameter}. Given this notation, we can define the following vertex weights.
	
	\begin{definition}
		
		\label{lzsr}
		
		Fix an arrow configuration $(i_1, j_1, i_2, j_2) \in \mathbb{Z}_{\ge 0}^4$ and complex numbers $z, r, s \in \mathbb{C}$. Assume there exists an integer $J \ge 0$ such that $r^2 = t^{-J}$. Then, define the \emph{vertex weight} $L_z (i_1, j_1; i_2, j_2 \boldsymbol{\mid} r, s)$ by setting
		\begin{flalign}
			\label{lz} 
			\begin{aligned} 
			L_z (i_1, j_1; i_2, j_2 \boldsymbol{\mid} r, s) & =  \textbf{1}_{i_1 + j_1 = i_2 + j_2} \textbf{1}_{j_1 \le J} \textbf{1}_{j_2 \le J} \cdot (-1)^{i_1} t^{\binom{i_1}{2} + i_1 j_1} z^{i_1} s^{j_1 + j_2 - i_2} \\
			& \qquad \times \displaystyle\frac{(s^{-1} z; t)_{j_2 - i_1} (s^2; t)_{i_2} (t^{j_2 - i_1 + 1}; t)_{i_2} (r^{-2} t^{1 - i_2 - j_2}; t)_{i_2}}{(t; t)_{i_2} (sz; t)_{i_2 + j_2} (r^{-2} t^{1 - j_1}; t)_{j_1 - j_2}}  \\
			& \qquad \times \displaystyle\sum_{k = 0}^{i_2} t^k \displaystyle\frac{(t^{-i_2}; t)_k (t^{-i_1}; t)_k (r^{-2} sz; t)_k (tsz^{-1}; t)_k}{(t; t)_k (s^2; t)_k (t^{j_2 - i_1 + 1}; t)_k (r^{-2} t^{1 - i_2 - j_2}; t)_k}.
			\end{aligned} 
		\end{flalign}
	
	\end{definition}

	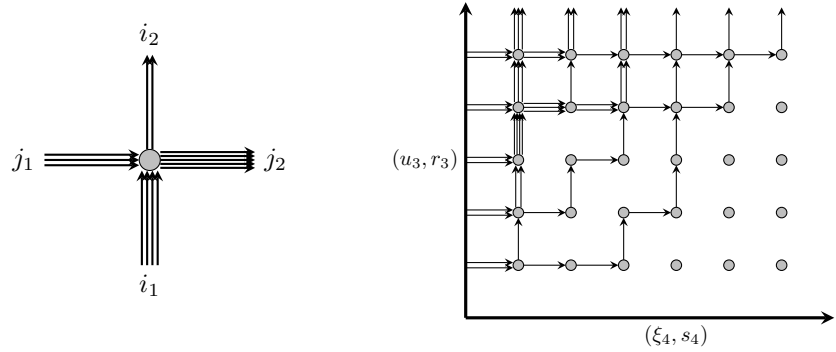
\begin{figure}
		
		\begin{center} 
			
			\begin{tikzpicture}[
				>=stealth,
				scale = .7
				]
				
				\draw[->,black, thick] (-.05, .2) -- (-.05, 2);
				\draw[->,black, thick] (.05, .2) -- (.05, 2);
				\draw[->,black, thick] (-.15,-2) -- (-.15, -.2);
				\draw[->,black, thick] (-.05,-2) -- (-.05, -.2);
				\draw[->,black, thick] (.05,-2) -- (.05, -.2);
				\draw[->,black, thick] (.15,-2) -- (.15, -.2);
				\draw[->,black, thick] (-2, -.1) -- (-.2, -.1);
				\draw[->,black, thick] (-2, 0) -- (-.2, 0);
				\draw[->,black, thick] (-2,.1) -- (-.2, .1);
				\draw[->,black, thick] (.2, -.15) -- (2, -.15); 
				\draw[->,black, thick] (.2, -.075) -- (2, -.075); 
				\draw[->,black, thick] (.2, 0) -- (2, 0); 
				\draw[->,black, thick] (.2, .075) -- (2, .075); 
				\draw[->,black, thick] (.2, .15) -- (2, .15);
				
				\draw[] (0, 2) circle [radius = 0] node[above]{$i_2$};
				\draw[] (0, -2) circle [radius = 0] node[below]{$i_1$};
				\draw[] (-2, 0) circle [radius = 0] node[left]{$j_1$};
				\draw[] (2, 0) circle [radius = 0] node[right]{$j_2$};
				
				\filldraw[fill=gray!50!white, draw=black] (0, 0) circle [radius=.2];
				
				\draw[->, very thick] (6, -3) -- (13, -3);
				\draw[->, very thick] (6, -3) -- (6, 3);
				
				\draw[->] (6, -1.95) -- (6.9, -1.95);
				\draw[->] (6, -2.05) -- (6.9, -2.05);
				
				\draw[->] (6, -.95) -- (6.9, -.95);
				\draw[->] (6, -1.05) -- (6.9, -1.05);
				
				\draw[->] (6, -.05) -- (6.9, -.05);
				\draw[->] (6, .05) -- (6.9, .05);
				
				\draw[->] (6, .95) -- (6.9, .95);
				\draw[->] (6, 1.05) -- (6.9, 1.05);
				
				\draw[->] (6, 1.95) -- (6.9, 1.95);
				\draw[->] (6, 2.05) -- (6.9, 2.05);
				
				\draw[->] (7, -1.9) -- (7, -1.1);
				\draw[->] (7.1, -2) -- (7.9, -2);
				\draw[->] (8.1, -2) -- (8.9, -2);
				\draw[->] (9, -1.9) -- (9, -1.1);
				
				\draw[->] (6.95, -.9) -- (6.95, -.1);
				\draw[->] (7.05, -.9) -- (7.05, -.1);
				\draw[->] (7.1, -1) -- (7.9, -1);
				\draw[->] (8, -.9) -- (8, -.1);
				\draw[->] (9.1, -1) -- (9.9, -1);
				\draw[->] (10, -.9) -- (10, -.1);
				
				\draw[->] (6.925, .1) -- (6.925, .9);
				\draw[->] (6.975, .1) -- (6.975, .9);
				\draw[->] (7.025, .1) -- (7.025, .9);
				\draw[->] (7.075, .1) -- (7.075, .9);
				\draw[->] (8.1, 0) -- (8.9, 0);
				\draw[->] (9, .1) -- (9, .9);
				\draw[->] (10, .1) -- (10, .9);
				
				\draw[->] (6.925, 1.1) -- (6.925, 1.9);
				\draw[->] (7, 1.1) -- (7, 1.9);
				\draw[->] (7.075, 1.1) -- (7.075, 1.9);
				\draw[->] (7.1, .925) -- (7.9, .925);
				\draw[->] (7.1, 1) -- (7.9, 1);
				\draw[->] (7.1, 1.075) -- (7.9, 1.075);
				\draw[->] (8, 1.1) -- (8, 1.9);
				\draw[->] (8.1, .95) -- (8.9, .95);
				\draw[->] (8.1, 1.05) -- (8.9, 1.05);
				\draw[->] (8.95, 1.1) -- (8.95, 1.9);
				\draw[->] (9.05, 1.1) -- (9.05, 1.9);
				\draw[->] (9.1, 1) -- (9.9, 1);
				\draw[->] (10, 1.1) -- (10, 1.9);
				\draw[->] (10.1, 1) -- (10.9, 1);
				\draw[->] (11, 1.1) -- (11, 1.9);
				
				\draw[->] (6.925, 2.1) -- (6.925, 2.9);
				\draw[->] (7, 2.1) -- (7, 2.9);
				\draw[->] (7.075, 2.1) -- (7.075, 2.9);
				\draw[->] (7.1, 1.95) -- (7.9, 1.95);
				\draw[->] (7.1, 2.05) -- (7.9, 2.05);
				\draw[->] (7.95, 2.1) -- (7.95, 2.9);
				\draw[->] (8.05, 2.1) -- (8.05, 2.9);
				\draw[->] (8.1, 2) -- (8.9, 2);
				\draw[->] (8.95, 2.1) -- (8.95, 2.9);
				\draw[->] (9.05, 2.1) -- (9.05, 2.9);
				\draw[->] (9.1, 2) -- (9.9, 2);
				\draw[->] (10, 2.1) -- (10, 2.9);
				\draw[->] (10.1, 2) -- (10.9, 2);
				\draw[->] (11, 2.1) -- (11, 2.9);
				\draw[->] (11.1, 2) -- (11.9, 2);
				\draw[->] (12, 2.1) -- (12, 2.9);
				
				\filldraw[fill=gray!50!white, draw=black] (7, -2) circle [radius=.1];
				\filldraw[fill=gray!50!white, draw=black] (8, -2) circle [radius=.1];
				\filldraw[fill=gray!50!white, draw=black] (9, -2) circle [radius=.1];
				\filldraw[fill=gray!50!white, draw=black] (10, -2) circle [radius=.1];
				\filldraw[fill=gray!50!white, draw=black] (11, -2) circle [radius=.1];
				\filldraw[fill=gray!50!white, draw=black] (12, -2) circle [radius=.1];
				
				\filldraw[fill=gray!50!white, draw=black] (7, -1) circle [radius=.1];
				\filldraw[fill=gray!50!white, draw=black] (8, -1) circle [radius=.1];
				\filldraw[fill=gray!50!white, draw=black] (9, -1) circle [radius=.1];
				\filldraw[fill=gray!50!white, draw=black] (10, -1) circle [radius=.1];
				\filldraw[fill=gray!50!white, draw=black] (11, -1) circle [radius=.1];
				\filldraw[fill=gray!50!white, draw=black] (12, -1) circle [radius=.1];
				
				\filldraw[fill=gray!50!white, draw=black] (7, 0) circle [radius=.1];
				\filldraw[fill=gray!50!white, draw=black] (8, 0) circle [radius=.1];
				\filldraw[fill=gray!50!white, draw=black] (9, 0) circle [radius=.1];
				\filldraw[fill=gray!50!white, draw=black] (10, 0) circle [radius=.1];
				\filldraw[fill=gray!50!white, draw=black] (11, 0) circle [radius=.1];
				\filldraw[fill=gray!50!white, draw=black] (12, 0) circle [radius=.1];
				
				\filldraw[fill=gray!50!white, draw=black] (7, 1) circle [radius=.1];
				\filldraw[fill=gray!50!white, draw=black] (8, 1) circle [radius=.1];
				\filldraw[fill=gray!50!white, draw=black] (9, 1) circle [radius=.1];
				\filldraw[fill=gray!50!white, draw=black] (10, 1) circle [radius=.1];
				\filldraw[fill=gray!50!white, draw=black] (11, 1) circle [radius=.1];
				\filldraw[fill=gray!50!white, draw=black] (12, 1) circle [radius=.1];
				
				\filldraw[fill=gray!50!white, draw=black] (7, 2) circle [radius=.1];
				\filldraw[fill=gray!50!white, draw=black] (8, 2) circle [radius=.1];
				\filldraw[fill=gray!50!white, draw=black] (9, 2) circle [radius=.1];
				\filldraw[fill=gray!50!white, draw=black] (10, 2) circle [radius=.1];
				\filldraw[fill=gray!50!white, draw=black] (11, 2) circle [radius=.1];
				\filldraw[fill=gray!50!white, draw=black] (12, 2) circle [radius=.1];

				\draw[] (6, 0) circle [radius = 0] node[left, scale = .8]{$(u_3, r_3)$};
				\draw[] (10, -3) circle [radius = 0] node[below, scale = .8]{$(\xi_4, s_4)$};
			\end{tikzpicture}
			
		\end{center}
		
		\caption{\label{arrows} Shown to the left is a vertex with arrow configuration $(i_1, j_1; i_2, j_2) = (4, 3; 2, 5)$. Shown to the right is a vertex model with $(2, 2, \ldots )$-step boundary data.} 
	\end{figure}

	The weights \eqref{lz} were originally found as equation (5.8) of \cite{ESVM} as entries for the higher spin $R$-matrix associated with the affine quantum algebra $U_t (\widehat{\mathfrak{sl}}_2)$. They were later interpreted as weights for stochastic vertex models through Theorem 3.15 of \cite{SHSVML} and equation (5.6) of \cite{HSVMSRF}; in particular, \eqref{lz} matches with the latter upon equating the $(q, q^J)$ there with $(t, r^{-2})$ here. As indicated by Theorem 3.15 of \cite{SHSVML}, the weights $L_z$ are \emph{stochastic}, meaning 
	\begin{flalign}
		\label{lzsum1} 
		\displaystyle\sum_{i_2, j_2 \ge 0} L_z (i_1, j_1; i_2, j_2 \boldsymbol{\mid} r, s) = 1.
	\end{flalign}

	\noindent Throughout, the parameters $(z, r, s)$ will be selected so that the summands in \eqref{lzsum1} are nonnegative.
	
	Now let us describe how to sample a random path ensemble using the $L_z$ weights from \eqref{lz}. We will first define probability measures $\mathbb{P}_n$ on the set of path ensembles whose vertices are all contained in triangles of the form $\mathbb{T}_n = \{ (x, y) \in \mathbb{Z}_{\ge 0}^2: x + y \le n \}$, and then we will take a limit as $n$ tends to infinity to obtain the vertex models in infinite volume. The first two measures $\mathbb{P}_0$ and $\mathbb{P}_1$ are both supported by the empty ensembles (that have no paths). 
	
	For each positive integer $n \ge 1$, we will define $\mathbb{P}_{n + 1}$ from $\mathbb{P}_n$ through the following Markovian update rules. Use $\mathbb{P}_n$ to sample a directed path ensemble $\mathcal{E}_n$ on $\mathbb{T}_n$. This yields arrow configurations for all vertices in the triangle $\mathbb{T}_{n - 1}$. To extend this to a path ensemble on $\mathbb{T}_{n + 1}$, we must prescribe arrow configurations to all vertices $(x, y)$ on the complement $\mathbb{T}_n \setminus\mathbb{T}_{n - 1}$, which is the diagonal $\mathbb{D}_n = \{ (x, y) \in \mathbb{Z}_{> 0}^2: x + y = n \}$. Since any incoming arrow to $\mathbb{D}_n$ is an outgoing arrow from $\mathbb{D}_{n - 1}$, $\mathcal{E}_n$ and the initial data prescribe the first two coordinates, $i_1$ and $j_1$, of the arrow configuration to each $(x, y) \in \mathbb{D}_n$. Thus, it remains to explain how to assign the second two coordinates ($i_2$ and $j_2$) to any vertex on $\mathbb{D}_n$, given the first two coordinates. 
	
	This is done by producing $(i_2, j_2)_{(x, y)}$ from $(i_1, j_1)_{(x, y)}$ according to the transition probability 
	\begin{flalign}
			\label{configurationprobabilities} 
			& \mathbb{P}_n \big[ (i_2, j_2)_{(x, y)} \big| (i_1, j_1)_{(x, y)} \big] = L_{u_x \xi_y} (i_1, j_1; i_2, j_2 \boldsymbol{\mid} r_y, s_x),   
	\end{flalign}
	
	\noindent where $t \in \mathbb{C}$ is a complex number and $\boldsymbol{u} = (u_1, u_2, \ldots ) \subset \mathbb{C}$, $\boldsymbol{\xi} = (\xi_1, \xi_2, \ldots ) \subset \mathbb{C}$, $\boldsymbol{r} = (r_1, r_2, \ldots ) \subset \mathbb{C}$, and $\boldsymbol{s} = (s_1, s_2, \ldots ) \subset \mathbb{C}$ are infinite sequences of complex numbers, so that $r_y^2 = t^{-J_y}$ for each $y \ge 1$, for some sequence of nonnegative integers $\boldsymbol{J} = (J_1, J_2, \ldots )$ (as in \Cref{lzsr}). We assume that these parameters are chosen so that the probabilities \eqref{configurationprobabilities} are all nonnegative; the stochasticity \eqref{lzsum1} of the $L_z$ weights then ensures that \eqref{configurationprobabilities} is indeed a probability measure.
	
	Choosing $(i_2, j_2)$ according to the above transition probabilities yields a random directed path ensemble $\mathcal{E}_{n + 1}$, now defined on $\mathbb{T}_{n + 1}$; the probability distribution of $\mathcal{E}_{n + 1}$ is then denoted by $\mathbb{P}_{n + 1}$. We define $\mathbb{P}_{\FV} = \lim_{n \rightarrow \infty} \mathbb{P}_n$.\footnote{Here, $\FV$ stands for ``fused vertex,'' and below $\PV$ will stand for ``prefused vertex.''} Then, $\mathbb{P}_{\FV}$ is a probability measure on the set of directed path ensembles that depends on the parameters $t$, $\boldsymbol{u}$, $\boldsymbol{\xi}$, $\boldsymbol{r}$ (equivalently, $\boldsymbol{J}$), and $\boldsymbol{s}$. This measure is called the \emph{fused stochastic higher spin vertex model}; we denote the associated expectation by $\mathbb{E}_{\FV}$. 
	
	If $r_y = t^{-1/2}$ (that is, $J_y = 1$), then this model is known as the \emph{(prefused) stochastic higher spin vertex model} \cite{SHSVML,HSVMSRF}. By \eqref{lz}, $L_z (i_1, j_1; i_2, j_2 \boldsymbol{\mid} r, s) = 0$ if either $j_1 \notin \{ 0, 1 \}$ or $j_2 \notin \{ 0, 1 \}$. In particular, horizontal edges of this model can accommodate at most one arrow (but vertical edges may accommodate arbitrarily many). 
	
	\begin{rem} 
		
	\label{pfvrational}
	
	Although we have assumed above that $(t, \boldsymbol{u}, \boldsymbol{\xi}, \boldsymbol{J}, \boldsymbol{s})$ are chosen to ensure that the weights \eqref{configurationprobabilities} all nonnegative, the probability under $\mathbb{P}_{\FV}$ of any cylinder event is a rational function in these parameters. Therefore, the probability $\mathbb{P}_{\FV}$ and expectation $\mathbb{E}_{\FV}$ remain well-defined by analytic continuation for any complex parameters $(t, \boldsymbol{u}, \boldsymbol{\xi}, \boldsymbol{J}, \boldsymbol{s})$ (with $\boldsymbol{J}$ consisting of nonnegative integers) when this nonnegativty does not hold.
	
	\end{rem}

	Associated with any six-vertex ensemble $\mathcal{E}$ on the positive quadrant $\mathbb{Z}_{> 0}^2$ is a \emph{height function} $\mathfrak{h} : \mathbb{Z}_{> 0}^2 \rightarrow \mathbb{Z}$, defined by setting $\mathfrak{h} (x, y)$ equal to the number of paths in $\mathcal{E}$ that pass either through or below $(x, y)$. Observe that $\mathcal{E}$ is determined uniquely from its height function $\mathfrak{h}$. 
	
	Before proceeding, let us recall the relation between height functions for fused and prefused stochastic higher spin vertex models. To that end, we require some terminology.

	\begin{definition} 
		
		\label{tuxisjr}
		
		Fix sequences of real numbers $\boldsymbol{u} = (u_1, u_2, \ldots )$ and of positive integers $\boldsymbol{J} = (J_1, J_2, \ldots )$. For each $k \ge 1$, set $J_{[1, k]} = \sum_{i = 1}^k J_i$, and define $\boldsymbol{r} = (r_1, r_2, \ldots ) \subset \mathbb{R}$ by setting $r_i = t^{-J_i / 2}$ for each $i \ge 1$. Further set $\boldsymbol{v} = (v_1, v_2, \ldots ) = \bigcup_{k = 0}^{\infty} \{ u_k, t u_k, \ldots , t^{J_k - 1} u_k\}$, that is,  $v_i = t^j u_k$ for each integer $i \ge 1$, where the indices $j = j(i) \in [0, J_k - 1]$ and $k = k(i) \ge 1$ are such that $i = J_{[1, k - 1]} + j$. Letting $\boldsymbol{t}^{-1/2} = (t^{-1/2}, t^{-1/2}, \ldots )$, we call $(\boldsymbol{u}; \boldsymbol{r})$ the \emph{fusion} of $(\boldsymbol{v}, \boldsymbol{t}^{-1/2})$ with respect to $\boldsymbol{J}$. 
		
	\end{definition}

	\begin{figure}
		
		\begin{center} 
			
			\begin{tikzpicture}[
				>=stealth,
				scale = .7
				]
				
				\draw[->,  very thick] (-3, -4.5) -- (2, -4.5);
				\draw[->, very thick] (-3, -4.5) -- (-3, 2.5);

				\draw[] (-3, -3.5) circle [radius = 0] node[left, scale = .8]{$(u_1, t^{-1/2})$};
				\draw[] (-3, -2.5) circle [radius = 0] node[left, scale = .8]{$(t u_1, t^{-1/2})$};
				\draw[] (-3, -1.5) circle [radius = 0] node[left, scale = .8]{$(t^2 u_1, t^{-1 /2})$};
				\draw[] (-3, -.5) circle [radius = 0] node[left, scale = .8]{$(u_2, t^{-1/2})$};
				\draw[] (-3, .5) circle [radius = 0] node[left, scale = .8]{$(u_3, t^{-1/2})$};
				\draw[] (-3, 1.5) circle [radius = 0] node[left, scale = .8]{$(t u_3, t^{-1/2})$};

				\draw[->] (-3, -3.5) -- (-2.1, -3.5);
				\draw[->] (-3, -2.5) -- (-2.1, -2.5);
				\draw[->] (-3, -1.5) -- (-2.1, -1.5);
				\draw[->] (-3, -.5) -- (-2.15, -.5);
				\draw[->] (-3, .5) -- (-2.1, .5);
				\draw[->] (-3, 1.5) -- (-2.1, 1.5);
				
				\draw[->] (-1.9, -3.5) -- (-1.1, -3.5);
				\draw[->] (-1, -3.4) -- (-1, -2.6);
				\draw[->] (-.9, -2.5) -- (-.1, -2.5);
				\draw[->] (0, -2.4) -- (0, -1.65);
				\draw[->] (-2, -2.4) -- (-2, -1.6);
				
				\draw[->] (-1.9, -1.5) -- (-1.1, -1.5);
				\draw[->] (-1, -1.4) -- (-1, -.6);
				\draw[->] (-2, -1.4) -- (-2, -.65);
				\draw[->] (-.95, -.4) -- (-.95, .4);
				\draw[->] (-1.05, -.4) -- (-1.05, .4);
				\draw[->] (-2, -.35) -- (-2, .4);
				\draw[->] (-1.85, -.5) -- (-1.1, -.5);
				\draw[->] (0, -1.35) -- (0, -.6);
				\draw[->] (.1, -.5) -- (.9, -.5);
				\draw[->] (1, -.4) -- (1, .4);
				
				\draw[->] (-1.95, 1.6) -- (-1.95, 2.4);
				\draw[->] (-1.9, .5) -- (-1.1, .5);
				\draw[->] (-2, .6) -- (-2, 1.4);
				\draw[->] (-2.05, 1.6) -- (-2.05, 2.4);
				\draw[->] (-1.05, .6) -- (-1.05, 1.4);
				\draw[->] (-.95, .6) -- (-.95, 1.4);
				\draw[->] (-.9, .5) -- (-.1, .5);
				\draw[->] (-.9, 1.5) -- (-.1, 1.5);
				\draw[->] (0, .6) -- (0, 1.4);
				\draw[->] (-1, 1.6) -- (-1, 2.4); 
				\draw[->] (-.05, 1.6) -- (-.05, 2.4);
				\draw[->] (.05, 1.6) -- (.05, 2.4);
				\draw[->] (1, .6) -- (1, 1.35);
				\draw[->] (1, 1.6) -- (1, 2.4);
				
				\draw[] (-5.5, -3.5) -- (-5.75, -3.5) -- (-5.75, -1.5) -- (-5.5, -1.5);
				\draw[] (-5.5, -.75) -- (-5.75, -.75) -- (-5.75, -.25) -- (-5.5, -.25);
				\draw[] (-5.5, .5) -- (-5.75, .5) -- (-5.75, 1.5) -- (-5.5, 1.5);
				
				\draw[] (-5.75, -2.5) circle [radius = 0] node[left]{$J_1$};
				\draw[] (-5.75, -.5) circle [radius = 0] node[left]{$J_2$};
				\draw[] (-5.75, 1) circle [radius = 0] node[left]{$J_3$};
								
				\filldraw[fill=gray!50!white, draw = black] (-2, -3.5) circle [radius = .1];
				\filldraw[fill=gray!50!white, draw = black] (-1, -3.5) circle [radius = .1]; 
				\filldraw[fill=gray!50!white, draw = black] (0, -3.5) circle [radius = .1]; 
				\filldraw[fill=gray!50!white, draw = black] (1, -3.5) circle [radius = .1];  
				
				\filldraw[fill=gray!50!white, draw = black] (-2, -2.5) circle [radius = .1];
				\filldraw[fill=gray!50!white, draw = black] (-1, -2.5) circle [radius = .1]; 
				\filldraw[fill=gray!50!white, draw = black] (0, -2.5) circle [radius = .1]; 
				\filldraw[fill=gray!50!white, draw = black] (1, -2.5) circle [radius = .1];  
				
				\filldraw[fill=gray!50!white, draw = black] (-2, -1.5) circle [radius = .1];
				\filldraw[fill=gray!50!white, draw = black] (-1, -1.5) circle [radius = .1]; 
				\filldraw[fill=blue, draw = black] (0, -1.5) circle [radius = .15]; 
				\filldraw[fill=gray!50!white, draw = black] (1, -1.5) circle [radius = .1];  
				
				\filldraw[fill=orange, draw = black] (-2, -.5) circle [radius = .15];
				\filldraw[fill=gray!50!white, draw = black] (-1, -.5) circle [radius = .1]; 
				\filldraw[fill=gray!50!white, draw = black] (0, -.5) circle [radius = .1]; 
				\filldraw[fill=gray!50!white, draw = black] (1, -.5) circle [radius = .1];  
				
				\filldraw[fill=gray!50!white, draw = black] (-2, .5) circle [radius = .1];
				\filldraw[fill=gray!50!white, draw = black] (-1, .5) circle [radius = .1]; 
				\filldraw[fill=gray!50!white, draw = black] (0, .5) circle [radius = .1]; 
				\filldraw[fill=gray!50!white, draw = black] (1, .5) circle [radius = .1];  
				
				\filldraw[fill=gray!50!white, draw = black] (-2, 1.5) circle [radius = .1];
				\filldraw[fill=gray!50!white, draw = black] (-1, 1.5) circle [radius = .1]; 
				\filldraw[fill=gray!50!white, draw = black] (0, 1.5) circle [radius = .1]; 
				\filldraw[fill=magenta, draw = black] (1, 1.5) circle [radius = .15];

				\draw[->, very thick] (6, -3) -- (11, -3);
				\draw[->, very thick] (6, -3) -- (6, 1);
				
				\draw[] (6, -2) circle [radius = 0] node[left, scale = .8]{$(u_1, t^{-J_1/2})$};
				\draw[] (6, -1) circle [radius = 0] node[left, scale = .8]{$(u_2, t^{-J_2/2})$};
				\draw[] (6, 0) circle [radius = 0] node[left, scale = .8]{$(u_3, t^{-J_3 / 2})$};
				
				\draw[->] (6, -2.075) -- (6.9, -2.075);
				\draw[->] (6, -2) -- (6.9, -2);
				\draw[->] (6, -1.925) -- (6.9, -1.925);
				
				\draw[->] (6, -1) -- (6.85, -1);
				
				\draw[->] (6, -.05) -- (6.9, -.05);
				\draw[->] (6, .05) -- (6.9, .05);
				
				\draw[->] (7, -1.9) -- (7, -1.15);
				\draw[->] (7.1, -2.05) -- (7.9, -2.05);
				\draw[->] (7.1, -1.95) -- (7.9, -1.95);
				\draw[->] (8, -1.9) -- (8, -1.1);
				\draw[->] (8.1, -2) -- (8.85, -2);
				\draw[->] (9, -1.85) -- (9, -1.1);
				
				\draw[->] (7, -.85) -- (7, -.1);
				\draw[->] (7.15, -1) -- (7.9, -1);
				\draw[->] (7.95, -.9) -- (7.95, -.1);
				\draw[->] (8.05, -.9) -- (8.05, -.1);
				\draw[->] (9.1, -.95) -- (9.9, -.95);
				\draw[->] (10, -.9) -- (10, -.15);
				
				\draw[->] (6.95, .1) -- (6.95, .9);
				\draw[->] (7.1, 0) -- (7.9, 0);
				\draw[->] (7.05, .1) -- (7.05, .9);
				\draw[->] (8, .1) -- (8, .9);
				\draw[->] (8.1, -.05) -- (8.9, -.05);
				\draw[->] (8.1, .05) -- (8.9, .05);
				\draw[->] (9.05, .1) -- (9.05, .9);
				\draw[->] (8.95, .1) -- (8.95, .9);
				\draw[->] (10, .15) -- (10, .9);

				\filldraw[fill=gray!50!white, draw=black] (7, -2) circle [radius=.1];
				\filldraw[fill=gray!50!white, draw=black] (8, -2) circle [radius=.1];
				\filldraw[fill=blue, draw=black] (9, -2) circle [radius=.15];
				\filldraw[fill=gray!50!white, draw=black] (10, -2) circle [radius=.1];

				\filldraw[fill=orange, draw=black] (7, -1) circle [radius=.15];
				\filldraw[fill=gray!50!white, draw=black] (8, -1) circle [radius=.1];
				\filldraw[fill=gray!50!white, draw=black] (9, -1) circle [radius=.1];
				\filldraw[fill=gray!50!white, draw=black] (10, -1) circle [radius=.1];
				
				\filldraw[fill=gray!50!white, draw=black] (7, 0) circle [radius=.1];
				\filldraw[fill=gray!50!white, draw=black] (8, 0) circle [radius=.1];
				\filldraw[fill=gray!50!white, draw=black] (9, 0) circle [radius=.1];
				\filldraw[fill=magenta, draw=black] (10, 0) circle [radius=.15];
				
				\draw[] (8, -3) circle [radius = 0] node[below, scale = .8]{$(\xi_2, s_2)$};
				\draw[] (-1, -4.5) circle [radius = 0] node[below, scale = .8]{$(\xi_2, s_2)$};
			\end{tikzpicture}
			
		\end{center}
		
		\caption{\label{arrowsfused} Shown to the left is a path ensemble from a prefused model, which concatenates to one for a fused model shown on the right. The joint laws of the height function at the corresponding colored vertices coincide.}
		
	\end{figure}
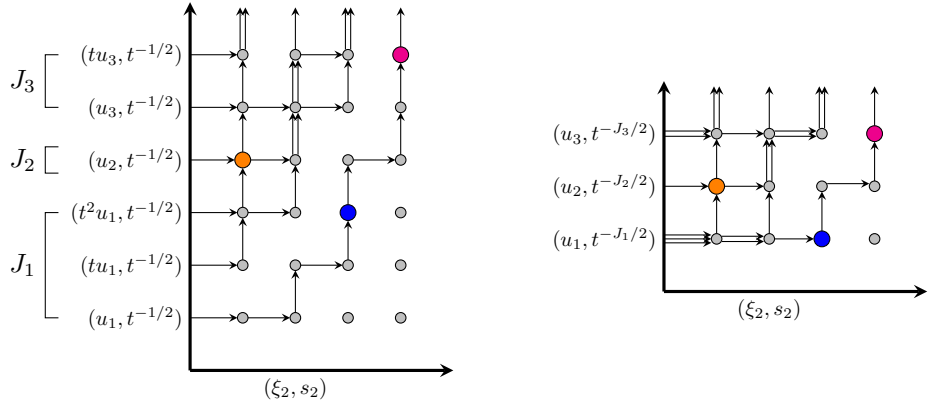

	Under this notation, $\boldsymbol{v}$ is a union of geometric progressions with ratio $t$ started from entries of $\boldsymbol{u}$, with lengths indexed by $\boldsymbol{J}$. The following lemma, essentially originating in \cite{ERT} (though described in the formulation below in \cite{SHSVML,HSVMSRF,ESVM,FSRF}), states that one may view the fused stochastic higher spin vertex model with parameters $(t, \boldsymbol{v}, \boldsymbol{\xi}, \boldsymbol{r}, \boldsymbol{s})$ as obtained from the prefused one with parameters $(t, \boldsymbol{u}, \boldsymbol{\xi}, \boldsymbol{t}^{-1/2}, \boldsymbol{s})$ by ``concatenating'' (or ``fusing'') each family of rows corresponding to a single geometric progression. We refer to \Cref{arrowsfused} for a depiction.

	\begin{lem}[{\cite[Section 5]{HSVMSRF}}]
		
		\label{jmeasure}

		Consider two stochastic higher spin vertex models. The first is prefused under step boundary data with parameters $(t, \boldsymbol{u}, \boldsymbol{\xi}, \boldsymbol{t}^{-1/2}, \boldsymbol{s})$, and the second is fused under $\boldsymbol{J}$-step boundary data with parameters $(t, \boldsymbol{v}, \boldsymbol{\xi}, \boldsymbol{r}, \boldsymbol{s})$; denote their height functions by $\mathfrak{h}_{\PV}$ and $\mathfrak{h}_{\FV}$, respectively. Suppose $(\boldsymbol{v}, \boldsymbol{r})$ is the fusion of $(\boldsymbol{u}, \boldsymbol{t}^{-1/2})$ with respect to $\boldsymbol{J}$. Then, for any vertices $(x_1, y_1), (x_2, y_2), \ldots , (x_m, y_m) \in \mathbb{Z}_{> 0}^2$, the joint laws of $ \big\{ \mathfrak{h}_{\FV} (x_i, y_i) \big\}$ and $\big\{ \mathfrak{h}_{\PV} (x_i, J_{[1, y_i]}) \big\}$ coincide.
		
	\end{lem}

	\subsection{Macdonald Measures and Vertex Models}
	
	\label{MeasureModel}
	
	We now recall a matching result between observables for the stochastic higher spin vertex models and those for certain Macdonald measures from \cite{SHSVMM}; we begin by recalling the latter measures.
	
	Fix complex parameters $q, t \in \mathbb{C}$ with $|q|, |t| < 1$, and let $\boldsymbol{x} = (x_1, x_2, \ldots )$ denote an infinite set of variables. A \emph{partition} $\lambda = (\lambda_1, \lambda_2, \ldots , \lambda_{\ell})$ is a non-increasing sequence of positive integers. The \emph{length} of any partition $\lambda$ is $\ell (\lambda) = \ell$, and the \emph{size} of $\lambda$ is $|\lambda| = \sum_{i = 1}^{\ell} \lambda_i$. For any integer $i \ge 1$, we let $m_i (\lambda)$ denote the multiplicity of $i$ in $\lambda$, that is, it denotes the number of indices $j \ge 1$ such that $\lambda_j = i$. For any integer $n \ge 0$, let $\mathbb{Y}_n$ denote the set of partitions of size $n$, and let $\mathbb{Y} = \bigcup_{n = 0}^{\infty} \mathbb{Y}_n$ denote the set of all partitions. 
	
	Let $\Lambda_{q, t} (\boldsymbol{x})$ denote the ring of symmetric functions in $\boldsymbol{x}$, with coefficients in $\mathbb{C} (q, t)$. Denote the \emph{power sum symmetric function} $p_{\lambda} (\boldsymbol{x})$ and the \emph{$(q, t)$-deformed complete homogeneous symmetric function} $g_k (\boldsymbol{x}; q, t)$ for each partition $\lambda \in \mathbb{Y}$ and integer $k \ge 0$ by
	\begin{flalign*} 
		p_k (\boldsymbol{x}) = \sum_{x \in \boldsymbol{x}} x^k; \qquad p_{\lambda} (\boldsymbol{x}) = \displaystyle\prod_{j = 1}^{\ell (\lambda)} p_{\lambda_j} (\boldsymbol{x}); \qquad g_k (\boldsymbol{x}; q, t) = \displaystyle\sum_{\lambda \in \mathbb{Y}_k} z_{\lambda} (q, t)^{-1} p_{\lambda} (\boldsymbol{x}),
	\end{flalign*} 

	\noindent where 
	\begin{flalign*} 
		z_{\lambda} (q, t) = \displaystyle\prod_{i = 1}^{\infty} i^{m_i (\lambda)} m_i (\lambda)! \displaystyle\prod_{j = 1}^{\ell (\lambda)} \displaystyle\frac{1 - q^{\lambda_j}}{1 - t^{\lambda_j}}.
	\end{flalign*} 

	\noindent Under this notation, $\big\{ p_{\lambda} (\boldsymbol{x}) \big\}_{\lambda \in \mathbb{Y}}$ is a linear basis of $\Lambda_{q, t} (\boldsymbol{x})$, and $\big\{ p_k (\boldsymbol{x}) \big\}_{k \ge 1}$ and $\big\{ g_k (\boldsymbol{x}; q, t) \big\}_{k \ge 1}$ are both algebraic bases of $\Lambda_{q, t} (\boldsymbol{x})$; see \cite{SFP}.
	
	 Next, we recall the Macdonald symmetric functions $P_{\lambda} (\boldsymbol{x}; q, t) \in \Lambda_{q, t} (\boldsymbol{x})$ and $Q_{\lambda} (\boldsymbol{x}; q, t) \in \Lambda_{q, t} (\boldsymbol{x})$, from (6.4.7) and (6.4.12) of \cite{SFP}, respectively. For any two sequences of variables $\boldsymbol{x}$ and $\boldsymbol{y}$, (6.4.13) of \cite{SFP} implies that Macdonald polynomials satisfy the \emph{Cauchy identity}
	\begin{flalign}
		\label{sumpq} 
		\displaystyle\sum_{\lambda \in \mathbb{Y}} P_{\lambda} (\boldsymbol{x}; q, t) Q_{\lambda} (\boldsymbol{y}; q, t) = \Omega_{q, t} (\boldsymbol{x}; \boldsymbol{y}),
	\end{flalign} 

	\noindent where 
	\begin{flalign*} 
		\Omega_{q, t} (\boldsymbol{x}; \boldsymbol{y}) = \displaystyle\prod_{x \in \boldsymbol{x}} \displaystyle\prod_{y \in \boldsymbol{y}} \displaystyle\frac{(txy; q)_{\infty}}{(xy; q)_{\infty}} = \exp \Bigg( \displaystyle\sum_{k = 1}^{\infty} \displaystyle\frac{1}{k} \displaystyle\frac{1 - t^k}{1 - q^k} p_k (\boldsymbol{x}) p_k (\boldsymbol{y}) \Bigg).
	\end{flalign*} 
	
	To define Macdonald measures in the generality we will eventually use, we require \emph{specializations} of $\Lambda_{q, t} (\boldsymbol{x})$, which are algebra homomorphisms $\rho : \Lambda_{q, t} (\boldsymbol{x}) \rightarrow \mathbb{C}$. For any element $f \in \Lambda_{q, t} (\boldsymbol{x})$, we abbreviate $f (\rho) = \rho \big( f(\boldsymbol{x}) \big)$; for example, $P_{\lambda} (\rho; q, t) = \rho \big( P_{\lambda} (\boldsymbol{x}; q, t) \big)$ and $Q_{\lambda} (\rho; q, t) = \rho \big( Q_{\lambda} (\boldsymbol{x}; q, t\big)$. For any two specializations $\rho_1, \rho_2 : \Lambda_{q, t} (\boldsymbol{x}) \rightarrow \mathbb{C}$, the Cauchy identity \eqref{sumpq} implies 
	\begin{flalign}
		\label{sumpqrho} 
		\displaystyle\sum_{\lambda \in \mathbb{Y}} P_{\lambda} (\rho_1; q, t) Q_{\lambda} (\rho_2; q, t) = \Omega_{q, t} (\rho_1; \rho_2), \quad \text{where $\Omega_{q, t} (\rho_1; \rho_2) = \exp \Bigg( \displaystyle\sum_{k = 1}^{\infty} \displaystyle\frac{1}{k} \displaystyle\frac{1 - t^k}{1 - q^k} p_k (\rho_1) p_k (\rho_2) \Bigg)$}.
	\end{flalign}

	\noindent whenever both sides of \eqref{sumpqrho} converge absolutely.
	
	Under this notation, following Definition 2.2.5 of \cite{P}, we define the \emph{Macdonald measure} $\mathbb{P}_{\MM} = \mathbb{P}_{\MM; \rho_1, \rho_2}$ on $\mathbb{Y}$ by setting
	\begin{flalign*}
		\mathbb{P}_{\MM} [\lambda] = \Omega_{q, t} (\rho_1; \rho_2)^{-1} P_{\lambda} (\rho_1; q, t) Q_{\lambda} (\rho_2, q; t),
	\end{flalign*}

	\noindent for any partition $\lambda \in \mathbb{Y}$; the associated expectation is denoted by $\mathbb{E}_{\MM} = \mathbb{E}_{\MM; \rho_1, \rho_2}$. By \eqref{sumpqrho}, the Macdonald measure is a probability measure assuming that $\rho_1$ and $\rho_2$ are \emph{Macdonald nonnegative}, that is, if $P_{\lambda} (\rho_1; q, t), Q_{\lambda} (\rho_2; q, t) \ge 0$ for each $\lambda \in \mathbb{Y}$. Otherwise, $P_{\MM}$ is a possibly signed measure that sums to one, and we can still consider probabilities and expectations with respect to it. 
	
	\begin{rem}
		
		\label{measure0qt}
		
		Two special cases of the Macdonald measure will be of particular use to us. First, if $q = 0$, then it is the \emph{Hall--Littlewood measure}, denoted by $\mathbb{P}_{\HL} = \mathbb{P}_{\HL; \rho_1, \rho_2}$. Second, if $q = t$, then it is the \emph{Schur measure} of \cite{IWRP}, denoted by $\mathbb{P}_{\SM} = \mathbb{P}_{\SM; \rho_1, \rho_2}$.
		
	\end{rem}
	
	Let us describe certain specializations $(\rho_1, \rho_2)$ that we will use. To define a specialization $\rho$, it suffices to fix $g_k (\rho; q, t)$ for each $k \ge 1$, since $\{ g_k \}_{k \ge 1}$ forms an algebraic basis of $\Lambda_{q, t} (\boldsymbol{x})$. For a real number $\gamma$ and (possibly infinite) sequences of nonnegative real numbers $\boldsymbol{\alpha} = (\alpha_1, \alpha_2, \ldots )$ and $\boldsymbol{\beta} = (\beta_1, \beta_2, \ldots )$ with $\sum_{j = 1}^{\infty} (\alpha_j + \beta_j) < \infty$, we write $\rho = (\boldsymbol{\alpha} \boldsymbol{\mid} \boldsymbol{\beta} \boldsymbol{\mid} \gamma)$ if the $g_k (\rho)$ satisfy
	\begin{flalign}
		\label{sumgkrho}
	\displaystyle\sum_{k = 0}^{\infty} z^k g_k (\rho) = e^{\gamma z} \displaystyle\prod_{j = 1}^{\infty} \displaystyle\frac{(t \alpha_j z; q)_{\infty}}{(\alpha_j z; q)_{\infty}} (1 + \beta_j z).
	\end{flalign}
	
	\noindent If $\gamma = 0$, then we abbreviate $\rho = (\boldsymbol{\alpha} \boldsymbol{\mid} \boldsymbol{\beta})$. Under this notation, we will in this section often set $\rho_1 = (\boldsymbol{x} \boldsymbol{\mid} \boldsymbol{0})$, where $\boldsymbol{0} = (0, 0, \ldots )$ is the sequence of infinitely many entries equal to $0$. 
	
	\begin{rem} 
		
		\label{convergealphagamma}
		
		Let $\boldsymbol{\alpha} = \boldsymbol{\alpha}_{\varepsilon} = (\varepsilon \alpha_1, \varepsilon \alpha_2, \ldots )$ be a sequence of nonnegative real numbers, dependent on a parameter $\varepsilon > 0$, and let $\boldsymbol{\beta} = (\beta_1, \beta_2, \ldots )$ be a sequence of nonnegative real numbers independent of $\varepsilon$. Suppose that $A = \varepsilon \sum_{i = 1}^{\infty} \alpha_i$, $\sum_{j = 1}^{\infty} \beta_j$, and $\max_{i \ge 1} \alpha_i$ are bounded above (independently of $\varepsilon$). Then, by \eqref{sumgkrho}, the specialization $(\boldsymbol{\alpha}_{\varepsilon} \boldsymbol{\mid} \boldsymbol{\beta})$ converges to\footnote{By this, we mean that $\lim_{\varepsilon \rightarrow 0} f( \boldsymbol{\alpha}_{\varepsilon} \boldsymbol {\mid} \boldsymbol{\beta}) = f \big( \boldsymbol{0} \boldsymbol{\mid} \boldsymbol{\beta} \boldsymbol{\mid} \frac{1 - t}{1 - q} A \big)$, for any symmetric function $f \in \Lambda_{q, t} (\boldsymbol{x})$.} $\big( \boldsymbol{0} \boldsymbol{\mid} \boldsymbol{\beta} \boldsymbol{\mid} \frac{1 - t}{1 - q} A \big)$, as $\varepsilon$ tends to $0$.
	\end{rem} 
	
	To equate observables for the stochastic higher spin vertex model and the Macdonald measure, we must impose how the parameters underlying these models are related. This is done through the following definition, which originally appeared in \cite{SHSVMM}.
	
	\begin{definition}[{\cite[Definition 4.1]{SHSVMM}}]
	
	\label{parametersmodel}

		Fix sequences of parameters $(t, \boldsymbol{u}, \boldsymbol{\xi}, \boldsymbol{r}, \boldsymbol{s})$ for a fused stochastic higher spin vertex model, such that $r_i = t^{-J_i / 2}$ for each $i \ge 1$ and some positive integer sequence $\boldsymbol{J} = (J_1, J_2, \ldots ) \subseteq \mathbb{Z}_{> 0}$. Further fix an integer $M \ge 1$ and nonnegative parameter sequences $(\boldsymbol{x}, \boldsymbol{\alpha}, \boldsymbol{\beta})$ for a Macdonald measure, with $\boldsymbol{x} = \big( x_1, x_2, \ldots , x_{J_{[1, N]}} \big)$ and $\sum_{j = 1}^{\infty} (\alpha_i + \beta_j) < \infty$. We say that these parameter sequences \emph{match} if the following conditions are satisfied.
		
		\begin{enumerate}
			\item We have $\boldsymbol{x} = \bigcup_{i = 1}^N \{ t^{1 - J_i} u_i^{-1}, t^{2 - J_i} u_i^{-1}, \ldots , u_i^{-1} \}$. 
			
			\item Denoting for any $z \in \mathbb{R}$ and $h \in \mathbb{Z}_{\ge 0}$ the geometric progression 
			\begin{flalign*}
				\mathcal{G} (z; h) = \{ z, tz, \ldots , t^{h - 1} z \},
			\end{flalign*}
			
			\noindent we may partition $\boldsymbol{\alpha}$ and $\boldsymbol{\beta}$ into disjoint unions of geometric progressions 
			\begin{flalign*}
				\boldsymbol{\alpha} = \bigcup_{i = 1}^m \mathfrak{G} (\widehat{\alpha}_i; h_i); \qquad \boldsymbol{\beta} = \bigcup_{i = 1}^n \mathfrak{G} (\widehat{\beta}_i; h_i),
			\end{flalign*}
		
			\noindent such $m + n = M$ that the following holds.  
			
			\begin{enumerate}
				\item For each $i \in \{ 1, 2, \ldots , m \}$, there exists $j = j(i) \ge 1$ such that $s_j = t^{-h_i / 2}$ and $\xi_j = t^{-h_i / 2} \widehat{\alpha}_i^{-1}$.
				\item For each $i \in \{ 1, 2, \ldots , n \}$, there exists $k = k(i) \ge 1$ such that $s_k = - q^{h_i / 2}$ and $\xi_k = q^{-h_i / 2} \widehat{\beta}_i^{-1}$.
				\item We have that $\big\{ j(1), j(2), \ldots \big\} \cup \big\{ k(1), k(2), \ldots \big\} = \{ 1, 2, \ldots , M \}$. 
			\end{enumerate} 
		\end{enumerate}
		
	\end{definition}  

	\noindent In particular, under this notation, each entry of $\{ s_1, s_2, \ldots , s_M \}$ is a possibly negated power of $t$ or $q$; its positive entries index the lengths of the geometric sequences comprising $\boldsymbol{\alpha}$, and its negative entries index the lengths of those comprising $\boldsymbol{\beta}$. Moreover, under the prefused setting where $J_i = 1$ for each $i \ge 1$, $\boldsymbol{u}$ and $\boldsymbol{x}$ coincide as unordered sets, upon inverting each entry of the latter.
	
	Now we have the following result, which was established in \cite{SHSVMM} (where it was stated in the prefused setting), that equates observables of the fused stochastic higher spin vertex model with a Macdonald measure, assuming their parameter sets match. 
	
	\begin{prop}[{\cite{SHSVMM}}]
		
		\label{pvm} 
		
		Let $t \in \mathbb{C}$ denote a complex number and  $\boldsymbol{u}, \boldsymbol{\xi}, \boldsymbol{r}, \boldsymbol{s}$ be infinite sequences of complex parameters, such that $r_i = t^{-J_i / 2}$ for some positive integer sequence $\boldsymbol{J} = (J_1, J_2, \ldots )$. Further let $M \ge 1$ be an integer and $(\boldsymbol{x}, \boldsymbol{\alpha}, \boldsymbol{\beta})$ be parameters sequences for a Macdonald measure with specializations $\rho_1 = (\boldsymbol{x} \boldsymbol{\mid} \boldsymbol{0})$ and $\rho_2 = (\boldsymbol{\alpha} \boldsymbol{\mid} \boldsymbol{\beta})$. Assume that these parameter sequences match in the sense of \Cref{parametersmodel}. Denoting $K = J_{[1, N]}$, we have for any $\zeta \in \mathbb{C} \setminus \{ -1, -t^{-1}, -t^{-2}, \ldots \}$ that
	\begin{flalign}
		\label{fvexpectation} 
		\mathbb{E}_{\FV} \Bigg[ \displaystyle\frac{1}{(- \zeta t^{\mathfrak{h} (M, N)}; t)_{\infty}} \Bigg] = \mathbb{E}_{\MM} \Bigg[ \displaystyle\frac{1}{(-\zeta; t)_{\infty}} \displaystyle\prod_{j = 0}^{K - 1} (1 + \zeta q^{\lambda_{K - j}} t^j) \Bigg].
	\end{flalign}
	
	\noindent Here, the left side denotes the expectation with respect to the fused stochastic higher spin vertex model with parameters $(t, \boldsymbol{u}, \boldsymbol{\xi}, \boldsymbol{r}, \boldsymbol{s})$, and the right side denotes the expectation with respect to the Macdonald measure with specializations $\rho_1$ and $\rho_2$.
	\end{prop}

	\begin{proof}
		Corollary 4.4 of \cite{SHSVMM} establishes \eqref{fvexpectation} in the prefused case, that is, when $J_i = 1$ for each $i \ge 1$. This, together with \Cref{jmeasure}, establishes the result in general. 
	\end{proof}

	The above proof of \Cref{pvm} ends up being rather heavy as it is based on finding and matching
	explicit integral representations of both sides of \eqref{fvexpectation}. In \Cref{EquationProof} below, we offer a more direct 
	and less formulaic argument that proves \Cref{pvm} in the Schur $q=t$ case. The general 
	Macdonald case then easily follows as well.

	By setting $q = 0$ in \Cref{pvm}, we deduce the following corollary that states a distributional equality between a fused stochastic higher spin vertex model and a Hall--Littlewood measure (recall \Cref{measure0qt}). Let us mention that (a multi-dimensional extension of) this corollary could also be derived from Theorem 4.1 of \cite{SSVMP} (or from the framework of \cite{SSE}).
	
	\begin{cor}
		
		\label{pvm1}
		
		Adopt the notation of \Cref{pvm}. Then, $K - \mathfrak{h} (M, N)$ under the fused stochastic higher spin vertex model with parameters $(t, \boldsymbol{u}, \boldsymbol{\xi}, \boldsymbol{r}, \boldsymbol{s})$ has the same law as $\ell (\lambda)$, sampled under the Hall--Littlewood measure with specializations $(\rho_1, \rho_2)$. 
	\end{cor}
	
	\begin{proof} 
		
		By taking $q = 0$ in \eqref{fvexpectation}, we deduce for any $\zeta \in \mathbb{C} \setminus \{ -1, -t^{-1}, \ldots \}$ that 
		\begin{flalign*}
			\mathbb{E}_{\FV} \bigg[ \displaystyle\frac{1}{(-\zeta t^{\mathfrak{h} (M, N)}; t)_{\infty}} \bigg] = \mathbb{E}_{\HL} \bigg[ \displaystyle\frac{1}{(-\zeta t^{K - \ell (\lambda)}; t)_{\infty}} \bigg],
		\end{flalign*} 
	
		\noindent where we have used \Cref{measure0qt} and the fact that $q^{\lambda_{K - j}} = 0$ unless $j < K - \ell (\lambda)$. This, together with the $t$-binomial theorem, implies 
		\begin{flalign}
			\label{sumzeta} 
			\displaystyle\sum_{j = 0}^{\infty} \displaystyle\frac{\zeta^j}{(t; t)_j} \mathbb{E}_{\FV} [t^{j \mathfrak{h} (M, N)}] = \displaystyle\sum_{j = 0}^{\infty} \displaystyle\frac{\zeta^j}{(t; t)_j} \mathbb{E}_{\HL} [t^{j(K - \ell (\lambda))}].
		\end{flalign} 
	
		\noindent In particular, the coefficients of $\zeta^j$ on either side of \eqref{sumzeta} must coincide for each $j \ge 0$; this implies that all moments of $t^{\mathfrak{h} (M, N)}$ and $t^{K - \ell (\lambda)}$ coincide. Since both are random variables bounded in $[0, 1]$, it follows that they have the same law; so, $\mathfrak{h} (M, N)$ and $K - \ell (\lambda)$ have the same law, from which we deduce the corollary.
	\end{proof}

	\section{Limits of the Fused Weights}
		
	\label{WeightsFused}
	
	In this section we analyze the fused vertex models from \Cref{PathEnsembles} under the limiting regime where the $J_i$ tend to $\infty$ and the $s_i$ each tend to either $0$ or $\infty$. We first explicitly evaluate the limiting $L_z$ vertex weights (from \eqref{lzsr}) in \Cref{LimitWeight1} and then explain an interpretation for the associated vertex model in \Cref{ModelArrow}.

	\subsection{Limiting Weights}

	\label{LimitWeight1} 
	
	In this section we consider limits of the $L_z (i_1, j_1; i_2, j_2 \boldsymbol{\mid} t^{-J / 2}, s)$ vertex weights from \eqref{lzsr} under the regimes where $(J, s) = (\infty, \infty)$ or $(J, s) = (\infty, 0)$ (see \Cref{llimit} and \Cref{llimit2} below, respectively). These quantities will admit limits if both the spectral parameter is of the form $z = t^{-J} s A$ for some fixed  $A \in \mathbb{C}$ and $(j_1, j_2) = (J - h_1, J - h_2)$ for some fixed $(h_1, h_2) \in \mathbb{Z}_{\ge 0}^2$. The first condition is closely related to the initial terms $q^{1 - J_i} u_i^{-1}$ in the sequence $\boldsymbol{x}$ from the first part of \Cref{parametersmodel}, and the second to the appearance of $K - \mathfrak{h} (M, N) = J_{[1, N]} - \mathfrak{h} (M, N)$ in \Cref{pvm1}. Observe that the second condition corresponds to horizontal edges being ``almost saturated'' with arrows (as \eqref{lzsr} implies $L_z (i_1, j_1; i_2, j_2 \boldsymbol{\mid} t^{-J / 2}, s) \ne 0$ only if $j_1, j_2 \le J$). In what follows, for any complex numbers $A, t \in \mathbb{C}$ and integers $i_1, h_1, i_2, h_2 \ge 0$ we define the quantities 
	\begin{flalign}
		\label{psia}
		\begin{aligned}
		\Psi_A (i_1, h_1; i_2, h_2) & = A^{-i_2} t^{i_2 (i_2 + h_1)} \displaystyle\frac{(A^{-1} t^{i_2 + h_1 + 1}; t)_{\infty} (t^{h_2 - i_2 + 1}; t)_{i_2} (t; t)_{h_1}}{(t; t)_{i_2} (t; t)_{h_2}} \\
		& \qquad \times \textbf{1}_{i_1 - h_1 = i_2 - h_2} \displaystyle\sum_{k = 0}^{i_2} (A t)^k \displaystyle\frac{(t^{-i_2}; t)_k (t^{-i_1}; t)_k}{(t; t)_k (t^{h_2 - i_2 + 1}; t)_k},
		\end{aligned} 
	\end{flalign}

	\noindent and 
	\begin{flalign} 
		\label{thetaa} 
		\begin{aligned} 
		\Theta_A (i_1, h_1; i_2, h_2) & = t^{\binom{i_2 + 1}{2} + i_2 h_1} A^{-i_2} \displaystyle\frac{(t^{h_2 - i_2 + 1}; t)_{i_2} (t; t)_{h_1}}{(-t^{h_2 - i_2 + 1} A^{-1}; t)_{\infty} (t; t)_{i_2} (t; t)_{h_2}} \\
		& \qquad \times \textbf{1}_{i_1 + j_1 = i_2 + j_2} \displaystyle\sum_{k = 0}^{i_2} t^k \displaystyle\frac{(t^{-i_1}; t)_k (t^{-i_2}; t)_k (-A; t)_k}{(t; t)_k (t^{h_2 - i_2 + 1}; t)_k}.
		\end{aligned}
	\end{flalign} 

	\begin{rem} 
		
	Observe for any real numbers $t \in (0, 1)$ and $A > 0$, and any integers $i_1, h_1, i_2, h_2 \ge 0$, that $\Psi_A (i_1, h_1; i_2, h_2) \ge 0$. Indeed, for $\Psi_A$, the only possibly negative factors on the right side of \eqref{psia0} are given by $(t^{h_2 - i_2 + 1}; t)_{i_2} (t^{h_2 - i_2 + 1}; t)_k^{-1} = (t^{h_2 - i_2 + k + 1}; t)_{i_2 - k}$ and $(t^{-i_2}; t)_k (t^{-i_1}; t)_k$. The first is nonzero only if $h_2 - i_2 + k \ge 0$, in which case it is positive; the second is also nonnegative since $(t^{-i_1}; t)_k$ and $(t^{-i_2}; t)_k$ are both nonzero only if $k \le \min \{ i_1, i_2 \}$, in which case they are of the same sign $(-1)^k$. Similar reasoning indicates that $\Theta_A (i_1, h_1; i_2, h_2) \ge 0$ under the same conditions. 
	
	\end{rem} 

	\begin{lem}
	
	\label{llimit} 
	
	For any complex numbers $A, t \in \mathbb{C}$ and integers $i_1, h_1, i_2, h_2 \in \mathbb{Z}_{\ge 0}$, we have  
	\begin{flalign*}
		\displaystyle\lim_{J \rightarrow \infty} \bigg( \displaystyle\lim_{s \rightarrow \infty} L_{As / t^J} & (i_1, J - h_1; i_2, J - h_2 \boldsymbol{\mid} q^{-J / 2}, s) \bigg) = \Psi_A (i_1, h_1; i_2, h_2),
	\end{flalign*} 

	\noindent where $\Psi_A$ is defined by \eqref{psia}. 

	\end{lem} 

	\begin{proof} 
		
		Throughout, we abbreviate $L = L_{As / t^J} (i_1, J - h_1; i_2, J - h_2 \boldsymbol{\mid} q^{-J / 2}, s)$. If $i_1 - h_1 \ne i_2 - h_2$, then $L = 0$ due to the factor of $\textbf{1}_{i_1 + j_1 = i_2 + j_2}$ in \eqref{lz}, in which case the lemma holds. 
		
		Thus, we will assume in what follows that $i_1 - h_1 = i_2 - h_2$. Then, setting $(u, j_1, j_2) = (As t^{-J}, J - h_1, J - h_2)$ in the definition \eqref{lz} of the $L_z$ weights gives
		\begin{flalign*}
			L & = (-1)^{i_1} t^{\binom{i_1}{2} - i_1 h_1} \displaystyle\frac{A^{i_1} s^{2J - 2h_2} (A t^{-J}; t)_{J - h_2 - i_1}}{(t; t)_{i_2} (As^2 t^{-J}; t)_{i_2 + J - h_2} (t^{h_1 + 1}; t)_{h_2 - h_1}} (t^{J - h_2 - i_1 + 1}; t)_{i_2} \\
			& \qquad \times (s^2; t)_{i_2} (t^{h_2 - i_2 + 1}; t)_{i_2} \displaystyle\sum_{k = 0}^{i_2} t^k \displaystyle\frac{(t^{-i_2}; t)_k (t^{-i_1}; t)_k (A s^2; t)_k (t^{J + 1} A^{-1}; t)_k}{(t; t)_k (s^2; t)_k (t^{J - h_2 - i_1 + 1}; t)_k (t^{h_2 - i_2 + 1}; t)_k},
		\end{flalign*}
		
	\noindent where we have used the fact that $i_1 + j_1 + j_2 - i_2 = 2j_2 = 2J - 2h_2$. Letting $s$ tend to $\infty$ and using the facts that 
	\begin{flalign*}
		\displaystyle\lim_{s \rightarrow \infty} \displaystyle\frac{(As^2 t^{-J}; t)_{i_2 + J - h_2}}{(A s^2 t^{-J})^{i_2 + J - h_2} t^{\binom{i_2 + J - h_2}{2}}} = 1; \qquad \displaystyle\lim_{s \rightarrow \infty} \displaystyle\frac{(s^2; t)_{i_2}}{(-s^2)^{i_2} t^{\binom{i_2}{2}}} = 1; \qquad \displaystyle\lim_{s \rightarrow \infty} \displaystyle\frac{(A s^2; t)_k}{(s^2; t)_k} = A^k,
	\end{flalign*}

	\noindent we find 
	\begin{flalign}
		\label{ls1} 
		\begin{aligned} 
		\displaystyle\lim_{s \rightarrow \infty} L & = (-1)^{i_1 + J + h_2} t^{\binom{i_1}{2} + \binom{i_2}{2} - i_1 h_1} \displaystyle\frac{A^{i_1 - i_2 - J + h_2} (A t^{-J}; t)_{J - h_2 - i_1} (t^{h_2 - i_2 + 1}; t)_{i_2}}{t^{J (h_2 - i_2 - J) + \binom{i_2 + J - h_2}{2}} (t; t)_{i_2} (t^{h_1 + 1}; t)_{h_2 - h_1}} \\
		& \qquad \times (t^{J - h_2 - i_1 + 1}; t)_{i_2} \displaystyle\sum_{k = 0}^{i_2} (At)^k \displaystyle\frac{(t^{-i_2}; t)_k (t^{-i_1}; t)_k (t^{J + 1} A^{-1}; t)_k}{(t; t)_k (t^{J - h_2 - i_1 + 1}; t)_k (t^{h_2 - i_2 + 1}; t)_k},
		\end{aligned}
	\end{flalign}

	\noindent Next, we let $J$ tend to $\infty$. Since for any $k \ge 0$ we have 
	\begin{flalign*}
		(A t^{-J}; t)_{J - k} = (-A)^{J - k} t^{\binom{k + 1}{2} - \binom{J + 1}{2}} (A^{-1} t^{k + 1}; t)_{J - k},
	\end{flalign*}

	\noindent we have 
	\begin{flalign*}
		\displaystyle\lim_{J \rightarrow \infty} \displaystyle\frac{(A t^{-J}; t)_{J - h_2 - i_1}}{(-A)^{J - h_2 - i_1} t^{\binom{i_1 + h_2 + 1}{2} - \binom{J + 1}{2}} (t^{i_1 + h_2 + 1} A^{-1}; t)_{\infty}} = 1.
	\end{flalign*}

	\noindent Inserting this into \eqref{ls1} gives
	\begin{flalign*}
		\displaystyle\lim_{J \rightarrow \infty} \Big( \displaystyle\lim_{s \rightarrow \infty} L \Big) & = A^{-i_2} t^{\binom{i_1}{2} + \binom{i_2}{2} - i_1 h_1} \displaystyle\frac{t^{\binom{i_1 + h_2 + 1}{2} - \binom{J + 1}{2}} (A^{-1} t^{i_1 + h_2 + 1}; t)_{\infty} (t^{h_2 - i_2 + 1}; t)_{i_2}}{t^{J (h_2 - i_2 - J) + \binom{i_2 + J - h_2}{2}} (t; t)_{i_2} (t^{h_1 + 1}; t)_{h_2 - h_1}} \\
		& \qquad \times \displaystyle\sum_{k = 0}^{i_2} (At)^k \displaystyle\frac{(t^{-i_2}; t)_k (t^{-i_1}; t)_k}{(t; t)_k (t^{h_2 - i_2 + 1}; t)_k}.
	\end{flalign*}
	
	\noindent Since 
	\begin{flalign*}
		& J (i_2 - h_2 + J) - \binom{J + 1}{2} - \binom{i_2 + J - h_2}{2} = - \binom{i_2 - h_2}{2} = - \binom{i_1 - h_1}{2}; \\
		& \binom{i_1}{2} - i_1 h_1 - \binom{i_1 - h_1}{2} = -\binom{h_1 + 1}{2}; \qquad i_1 + h_2 = i_2 + h_1,
	\end{flalign*}	
	
	\noindent it follows that 
	\begin{flalign}
		\label{ls2} 
		\begin{aligned}
		\displaystyle\lim_{J \rightarrow \infty} \Big( \displaystyle\lim_{s \rightarrow \infty} L \Big) & = A^{-i_2} t^{\binom{i_2}{2} + \binom{i_2 + h_1 + 1}{2} - \binom{h_1 + 1}{2}} \displaystyle\frac{(A^{-1} t^{i_1 + h_2 + 1}; t)_{\infty} (t^{h_2 - i_2 + 1}; t)_{i_2}}{(t; t)_{i_2} (t^{h_1 + 1}; t)_{h_2 - h_1}} \\
		& \qquad \times \displaystyle\sum_{k = 0}^{i_2} (At)^k \displaystyle\frac{(t^{-i_2}; t)_k (t^{-i_1}; t)_k}{(t; t)_k (t^{h_2 - i_2 + 1}; t)_k}.
		\end{aligned}
	\end{flalign} 
	
	\noindent By the equalities
	\begin{flalign*}
		\binom{i_2}{2} + \binom{i_2 + h_1 + 1}{2} - \binom{h_1 + 1}{2} = i_2 (i_2 + h_1); \qquad (t^{h_1 + 1}; t)_{h_2 - h_1} = \displaystyle\frac{(t; t)_{h_2}}{(t; t)_{h_1}},
	\end{flalign*} 

	\noindent \eqref{ls2} implies the lemma.
	\end{proof}

	\begin{lem} 
		
		\label{llimit2} 
		
		We have that 
		\begin{flalign*}
			\displaystyle\lim_{J \rightarrow \infty} \bigg( \displaystyle\lim_{s \rightarrow 0} L_{-A / st^J} (i_1, J - h_1; i_2, J - h_2 \boldsymbol{\mid} q^{-J / 2}, s) \bigg) = \Theta_A (i_1, h_1; i_2, h_2),
		\end{flalign*} 
	
		\noindent where $\Theta_A$ is defined by \eqref{thetaa}. 
	\end{lem} 
	
	\begin{proof} 
	
	Similarly to in the proof of \Cref{llimit}, we abbreviate $L = L_{-A / st^J} (i_1, J - h_1; i_2, J - h_2)$. Again, if $i_1 - h_1 \ne i_2 - h_2$, then $L = 0$ due to the factor of $\textbf{1}_{i_1 + j_1 = i_2 + j_2}$ in \eqref{lz}, so the lemma holds. 
	
	Thus, we assume in what follows that $i_1 - h_1 = i_2 - h_2$. Then, setting $(u, j_1, j_2) = (-As^{-1} t^{-J}, J - h_1, J - h_2)$ in \eqref{lz}, we deduce
	\begin{flalign*}
		L & = t^{\binom{i_1}{2} - i_1 h_1} A^{i_1} s^{2J - 2 i_1 - 2 h_2} \displaystyle\frac{(-s^{-2} t^{-J} A; t)_{J - h_2 - i_1} (t^{h_2 - i_2 + 1}; t)_{i_2}}{(t; t)_{i_2} (-t^{-J} A; t)_{J - h_2 + i_2} (t^{h_1 + 1}; t)_{h_2 - h_1}} \\
		& \qquad \times (s^2; t)_{i_2} (t^{J - i_1 - h_2 + 1}; t)_{i_2} \displaystyle\sum_{k = 0}^{i_2} t^k \displaystyle\frac{(t^{-i_1}; t)_k (t^{-i_2}; t)_k (-A; t)_k (s^2 t^{J + 1} A^{-1}; t)_k}{(t; t)_k (s^2; t)_k (t^{J - i_1 - h_2 + 1}; t)_k (t^{h_2 - i_2 + 1}; t)_k},
	\end{flalign*}

	\noindent where we have used the fact that $i_1 - h_1 = i_2 - h_2$. Letting $s$ tend to $0$ and using the fact that 
	\begin{flalign*}
		\displaystyle\lim_{s \rightarrow 0} s^{2 J - 2 i_1 - 2 h_2} (- s^{-2} t^{-J} A; t)_{J - h_2 - i_1} =  t^{J (i_1 + h_2 - J) + \binom{J - h_2 - i_1}{2}} A^{J - h_2 - i_1},
	\end{flalign*}

	\noindent we obtain
	\begin{flalign*}
		\displaystyle\lim_{s \rightarrow 0} L & = t^{\binom{i_1}{2} - i_1 h_1 + J (i_1 + h_2 - J) + \binom{J - i_1 - h_2}{2}} A^{J - h_2} \displaystyle\frac{(t^{h_2 - i_2 + 1}; t)_{i_2}}{(t; t)_{i_2} (-t^{-J} A; t)_{J - h_2 + i_2} (t^{h_1 + 1}; t)_{h_2 - h_1}} \\
		& \qquad \times (t^{J - i_1 - h_2 + 1}; t)_{i_2} \displaystyle\sum_{k = 0}^{i_2} t^k \displaystyle\frac{(t^{-i_1}; t)_k (t^{-i_2}; t)_k (-A; t)_k}{(t; t)_k (t^{J - i_1 - h_2 + 1}; t)_k (t^{h_2 - i_2 + 1}; t)_k}.
	\end{flalign*}

	\noindent Next letting $J$ tend to $\infty$ and using the facts that
	\begin{flalign*}
	 	& \displaystyle\lim_{J \rightarrow \infty} A^{-J} t^{J (J - h_2 + i_2) - \binom{J - h_2 + i_2}{2}}(- t^{-J} A; t)_{J - h_2 + i_2} = A^{i_2 - h_2} (-t^{h_2 - i_2 + 1} A^{-1}; t)_{\infty}; \\
	 	& \binom{i_1}{2} + \binom{J - i_1 - h_2}{2} - \binom{J - h_2 + i_2}{2} = \binom{i_2 + 1}{2} + h_1 (i_1 + i_2) - J (i_1 + i_2),
	\end{flalign*}

	\noindent where the first holds since 
	\begin{flalign*} 
		(-t^{-J} A; t)_{J - k} = t^{J (k - J) + \binom{J - k}{2}} A^{J - k} (-t^{k + 1} A^{-1}; t)_{J - k},
	\end{flalign*} 

	\noindent for any integer $k$, and the second holds since $i_1 - h_1 = i_2 - h_2$, we deduce  
	\begin{flalign*}
		\displaystyle\lim_{J \rightarrow \infty} \Big( \displaystyle\lim_{s \rightarrow 0} L \Big) & = t^{\binom{i_2 + 1}{2} + i_2 h_1} A^{-i_2} \displaystyle\frac{(t^{h_2 - i_2 + 1}; t)_{i_2}}{(-t^{h_2 - i_2 + 1} A^{-1}; t)_{\infty} (t; t)_{i_2} (t^{h_1 + 1}; t)_{h_2 - h_1}} \\
		& \qquad \times \displaystyle\sum_{k = 0}^{i_2} t^k \displaystyle\frac{(t^{-i_1}; t)_k (t^{-i_2}; t)_k (-A; t)_k}{(t; t)_k (t^{h_2 - i_2 + 1}; t)_k}.
	\end{flalign*} 

	\noindent Since 
	\begin{flalign*}
		(t^{h_1 + 1}; t)_{h_2 - h_1} = \displaystyle\frac{(t; t)_{h_2}}{(t; t)_{h_1}},
	\end{flalign*}

	\noindent this implies the lemma.
	\end{proof}

	\subsection{Corresponding Vertex Model}
	
	\label{ModelArrow} 
	
	In this section we explain how to interpret the vertex model associated with the limiting fused vertex weights derived in \Cref{LimitWeight1}. In what follows, for any complex numbers $A, t \in \mathbb{C}$ and integers $i_1, h_1, i_2, h_2 \ge 0$, we recall the quantities $\Psi_A (i_1, h_1; i_2, h_2)$ and $\Theta_A (i_1, h_1; i_2, h_2)$ from \eqref{psia} and \eqref{thetaa}, respectively. We further fix sequences $\boldsymbol{A} = (A_1, A_2, \ldots ) \subset \mathbb{C}$; $\boldsymbol{\omega} = (\omega_1, \omega_2, \ldots ) \subset \mathbb{C}$; and $\boldsymbol{s} = (s_1, s_2, \ldots )$ with $s_i \in \{ 0, \infty \}$ for each $i \ge 1$, such that each $\Psi_{A_y \omega_x} \in [0, 1]$ if $s_y = \infty$ and each $\Theta_{A_y \omega_x} \in [0, 1]$ if $s_y = 0$.
	
	The vertex models discussed here will be obtained as follows. Let $J \ge 1$ be a large integer and $s \in (0, 1)$ be a small real number. We first consider a fused stochastic higher spin vertex model, as in \Cref{PathEnsembles}, with rapidity parameters $(u_y, r_y) = (t^{-J} A_y, t^{-J / 2})$ in the $y$-th row. The rapidity parameters $(\xi_x, s_x)$ in the $x$-th column depend on whether $s_x = \infty$ or $s_x = 0$; if $s_x = \infty$ then set $(\xi_x, s_x) = (s^{-1} \omega_x, s^{-1})$, and otherwise if $s_x = 0$ then set  $(\xi_x, s_x) = (-s^{-1} \omega_x, s)$. Next, we \emph{horizontally complement} this vertex model, that is, we replace any arrow configuration $(i_1, j_1; i_2, j_2)$ with $(i_1, h_1; i_2, h_2) = (i_1, J - j_1; i_2, J - j_2)$. Then, we let $s$ tend to $0$ and $J$ tend to $\infty$. 
	
	Observe under the above complementation that $(J, J, \ldots )$-step boundary data on the quadrant becomes \emph{empty boundary data}, in which no arrows enter through either the $x$-axis or $y$-axis. Still, paths can exist in this model within the interior of the quadrant $\mathbb{Z}_{> 0}^2$. Indeed, due the complementation, the form of arrow conservation satisfied by this model will be $i_1 - h_1 = i_2 -  h_2$ (instead of $i_1 + h_1 = i_2 + h_2)$. As such, vertices admit the possibility to ``create'' two exiting arrows or ``destroy'' two entering ones\footnote{The arrow conservation $i_1 + h_2 = i_2 + h_1$ can alternatively be interpreted as directing paths up-left, instead of up-right.}; see \Cref{arrows2} for an example.

	\begin{figure}
		
		\begin{center} 
			
			\begin{tikzpicture}[
				>=stealth,
				scale = .7
				]
			
				\draw[->, very thick] (6, -3) -- (13, -3);
				\draw[->, very thick] (6, -3) -- (6, 3);
				
				\draw[->] (6, -1.95) -- (6.9, -1.95);
				\draw[->] (6, -2.05) -- (6.9, -2.05);
				
				\draw[->] (6, -.95) -- (6.9, -.95);
				\draw[->] (6, -1.05) -- (6.9, -1.05);
				
				\draw[->] (6, -.05) -- (6.9, -.05);
				\draw[->] (6, .05) -- (6.9, .05);
				
				\draw[->] (6, .95) -- (6.9, .95);
				\draw[->] (6, 1.05) -- (6.9, 1.05);
				
				\draw[->] (6, 1.95) -- (6.9, 1.95);
				\draw[->] (6, 2.05) -- (6.9, 2.05);
				
				\draw[->] (7, -1.9) -- (7, -1.1);
				\draw[->] (7.1, -2) -- (7.9, -2);
				\draw[->] (8.1, -2) -- (8.9, -2);
				\draw[->] (9, -1.9) -- (9, -1.1);
				
				\draw[->] (6.95, -.9) -- (6.95, -.1);
				\draw[->] (7.05, -.9) -- (7.05, -.1);
				\draw[->] (7.1, -1) -- (7.9, -1);
				\draw[->] (8, -.9) -- (8, -.1);
				\draw[->] (9.1, -1) -- (9.9, -1);
				\draw[->] (10, -.9) -- (10, -.1);
				
				\draw[->] (6.925, .1) -- (6.925, .9);
				\draw[->] (6.975, .1) -- (6.975, .9);
				\draw[->] (7.025, .1) -- (7.025, .9);
				\draw[->] (7.075, .1) -- (7.075, .9);
				\draw[->] (8.1, 0) -- (8.9, 0);
				\draw[->] (9, .1) -- (9, .9);
				\draw[->] (10, .1) -- (10, .9);
				
				\draw[->] (6.925, 1.1) -- (6.925, 1.9);
				\draw[->] (6.975, 1.1) -- (6.975, 1.9);
				\draw[->] (7.075, 1.1) -- (7.075, 1.9);
				\draw[->] (7.1, .95) -- (7.9, .95);
				\draw[->] (7.025, 1.1) -- (7.025, 1.9);
				\draw[->] (7.1, 1.05) -- (7.9, 1.05);
				\draw[->] (8.1, .95) -- (8.9, .95);
				\draw[->] (8.1, 1.05) -- (8.9, 1.05);
				\draw[->] (8.95, 1.1) -- (8.95, 1.9);
				\draw[->] (9.05, 1.1) -- (9.05, 1.9);
				\draw[->] (9.1, 1) -- (9.9, 1);
				\draw[->] (10, 1.1) -- (10, 1.9);
				\draw[->] (10.1, 1) -- (10.9, 1);
				\draw[->] (11, 1.1) -- (11, 1.9);
				
				\draw[->] (6.925, 2.1) -- (6.925, 2.9);
				\draw[->] (6.975, 2.1) -- (6.975, 2.9);
				\draw[->] (7.025, 2.1) -- (7.025, 2.9);
				\draw[->] (7.075, 2.1) -- (7.075, 2.9);
				\draw[->] (7.1, 1.95) -- (7.9, 1.95);
				\draw[->] (7.1, 2.05) -- (7.9, 2.05);
				\draw[->] (7.95, 2.1) -- (7.95, 2.9);
				\draw[->] (8.05, 2.1) -- (8.05, 2.9);
				\draw[->] (9, 2.1) -- (9, 2.9);
				\draw[->] (9.1, 2) -- (9.9, 2);
				\draw[->] (10, 2.1) -- (10, 2.9);
				\draw[->] (10.1, 2) -- (10.9, 2);
				\draw[->] (11, 2.1) -- (11, 2.9);
				\draw[->] (11.1, 2) -- (11.9, 2);
				\draw[->] (12, 2.1) -- (12, 2.9);
				
				\filldraw[fill=gray!50!white, draw=black] (7, -2) circle [radius=.1];
				\filldraw[fill=gray!50!white, draw=black] (8, -2) circle [radius=.1];
				\filldraw[fill=gray!50!white, draw=black] (9, -2) circle [radius=.1];
				\filldraw[fill=gray!50!white, draw=black] (10, -2) circle [radius=.1];
				\filldraw[fill=gray!50!white, draw=black] (11, -2) circle [radius=.1];
				\filldraw[fill=gray!50!white, draw=black] (12, -2) circle [radius=.1];
				
				\filldraw[fill=gray!50!white, draw=black] (7, -1) circle [radius=.1];
				\filldraw[fill=gray!50!white, draw=black] (8, -1) circle [radius=.1];
				\filldraw[fill=gray!50!white, draw=black] (9, -1) circle [radius=.1];
				\filldraw[fill=gray!50!white, draw=black] (10, -1) circle [radius=.1];
				\filldraw[fill=gray!50!white, draw=black] (11, -1) circle [radius=.1];
				\filldraw[fill=gray!50!white, draw=black] (12, -1) circle [radius=.1];
				
				\filldraw[fill=gray!50!white, draw=black] (7, 0) circle [radius=.1];
				\filldraw[fill=gray!50!white, draw=black] (8, 0) circle [radius=.1];
				\filldraw[fill=gray!50!white, draw=black] (9, 0) circle [radius=.1];
				\filldraw[fill=gray!50!white, draw=black] (10, 0) circle [radius=.1];
				\filldraw[fill=gray!50!white, draw=black] (11, 0) circle [radius=.1];
				\filldraw[fill=gray!50!white, draw=black] (12, 0) circle [radius=.1];
				
				\filldraw[fill=gray!50!white, draw=black] (7, 1) circle [radius=.1];
				\filldraw[fill=gray!50!white, draw=black] (8, 1) circle [radius=.1];
				\filldraw[fill=gray!50!white, draw=black] (9, 1) circle [radius=.1];
				\filldraw[fill=gray!50!white, draw=black] (10, 1) circle [radius=.1];
				\filldraw[fill=gray!50!white, draw=black] (11, 1) circle [radius=.1];
				\filldraw[fill=gray!50!white, draw=black] (12, 1) circle [radius=.1];
				
				\filldraw[fill=gray!50!white, draw=black] (7, 2) circle [radius=.1];
				\filldraw[fill=gray!50!white, draw=black] (8, 2) circle [radius=.1];
				\filldraw[fill=gray!50!white, draw=black] (9, 2) circle [radius=.1];
				\filldraw[fill=gray!50!white, draw=black] (10, 2) circle [radius=.1];
				\filldraw[fill=gray!50!white, draw=black] (11, 2) circle [radius=.1];
				\filldraw[fill=gray!50!white, draw=black] (12, 2) circle [radius=.1];

				\draw[->, very thick] (16, -3) -- (23, -3);
				\draw[->, very thick] (16, -3) -- (16, 3);
				
				\draw[->] (17, -1.9) -- (17, -1.1);
				\draw[->] (17.1, -2) -- (17.9, -2);
				\draw[->] (18.1, -2) -- (18.9, -2);
				\draw[->] (19.1, -2.05) -- (19.9, -2.05);
				\draw[->] (19.1, -1.95) -- (19.9, -1.95);
				\draw[->] (20.1, -2.05) -- (20.9, -2.05);
				\draw[->] (20.1, -1.95) -- (20.9, -1.95);
				\draw[->] (21.1, -2.05) -- (21.9, -2.05);
				\draw[->] (21.1, -1.95) -- (21.9, -1.95);
				\draw[->] (22.1, -2.05) -- (22.9, -2.05);
				\draw[->] (22.1, -1.95) -- (22.9, -1.95);

				\draw[->] (19, -1.9) -- (19, -1.1);
				
				\draw[->] (16.95, -.9) -- (16.95, -.1);
				\draw[->] (17.05, -.9) -- (17.05, -.1);
				\draw[->] (17.1, -1) -- (17.9, -1);
				\draw[->] (18, -.9) -- (18, -.1);
				\draw[->] (19.1, -1) -- (19.9, -1);
				\draw[->] (20, -.9) -- (20, -.1);
				
				\draw[->] (18.1, -1.05) -- (18.9, -1.05);
				\draw[->] (18.1, -.95) -- (18.9, -.95);
				\draw[->] (20.1, -1.05) -- (20.9, -1.05);
				\draw[->] (20.1, -.95) -- (20.9, -.95);
				\draw[->] (21.1, -1.05) -- (21.9, -1.05);
				\draw[->] (21.1, -.95) -- (21.9, -.95);
				\draw[->] (22.1, -1.05) -- (22.9, -1.05);
				\draw[->] (22.1, -.95) -- (22.9, -.95);
				
				\draw[->] (19.1, -.05) -- (19.9, -.05);
				\draw[->] (19.1, .05) -- (19.9, .05);
				\draw[->] (20.1, -.05) -- (20.9, -.05);
				\draw[->] (20.1, .05) -- (20.9, .05);
				\draw[->] (21.1, -.05) -- (21.9, -.05);
				\draw[->] (21.1, .05) -- (21.9, .05);
				\draw[->] (22.1, -.05) -- (22.9, -.05);
				\draw[->] (22.1, .05) -- (22.9, .05);
				
				\draw[->] (22.1, .95) -- (22.9, .95);
				\draw[->] (22.1, 1.05) -- (22.9, 1.05);
				\draw[->] (21.1, .95) -- (21.9, .95);
				\draw[->] (21.1, 1.05) -- (21.9, 1.05);
				
				\draw[->] (22.1, 1.95) -- (22.9, 1.95);
				\draw[->] (22.1, 2.05) -- (22.9, 2.05);
				
				\draw[->] (16.925, .1) -- (16.925, .9);
				\draw[->] (16.975, .1) -- (16.975, .9);
				\draw[->] (17.025, .1) -- (17.025, .9);
				\draw[->] (17.075, .1) -- (17.075, .9);
				\draw[->] (18.1, 0) -- (18.9, 0);
				\draw[->] (19, .1) -- (19, .9);
				\draw[->] (20, .1) -- (20, .9);
				
				\draw[->] (16.925, 1.1) -- (16.925, 1.9);
				\draw[->] (17.025, 1.1) -- (17.025, 1.9);
				\draw[->] (17.075, 1.1) -- (17.075, 1.9);
				\draw[->] (17.1, -.05) -- (17.9, -.05);
				\draw[->] (16.975, 1.1) -- (16.975, 1.9);
				\draw[->] (17.1, .05) -- (17.9, .05);
				\draw[->] (18.95, 1.1) -- (18.95, 1.9);
				\draw[->] (19.05, 1.1) -- (19.05, 1.9);
				\draw[->] (19.1, 1) -- (19.9, 1);
				\draw[->] (20, 1.1) -- (20, 1.9);
				\draw[->] (20.1, 1) -- (20.9, 1);
				\draw[->] (21, 1.1) -- (21, 1.9);
				
				\draw[->] (16.925, 2.1) -- (16.925, 2.9);
				\draw[->] (16.975, 2.1) -- (16.975, 2.9);
				\draw[->] (17.025, 2.1) -- (17.025, 2.9);
				\draw[->] (17.075, 2.1) -- (17.075, 2.9);
				
				\draw[->] (17.95, 2.1) -- (17.95, 2.9);
				\draw[->] (18.05, 2.1) -- (18.05, 2.9);
				
				\draw[->] (18.1, 1.95) -- (18.9, 1.95);
				\draw[->] (18.1, 2.05) -- (18.9, 2.05);
				\draw[->] (19, 2.1) -- (19, 2.9);
				\draw[->] (19.1, 2) -- (19.9, 2);
				\draw[->] (20, 2.1) -- (20, 2.9);
				\draw[->] (20.1, 2) -- (20.9, 2);
				\draw[->] (21, 2.1) -- (21, 2.9);
				\draw[->] (21.1, 2) -- (21.9, 2);
				\draw[->] (22, 2.1) -- (22, 2.9);
				
				\filldraw[fill=gray!50!white, draw=black] (17, -2) circle [radius=.1];
				\filldraw[fill=gray!50!white, draw=black] (18, -2) circle [radius=.1];
				\filldraw[fill=gray!50!white, draw=black] (19, -2) circle [radius=.1];
				\filldraw[fill=gray!50!white, draw=black] (20, -2) circle [radius=.1];
				\filldraw[fill=gray!50!white, draw=black] (21, -2) circle [radius=.1];
				\filldraw[fill=gray!50!white, draw=black] (22, -2) circle [radius=.1];
				
				\filldraw[fill=gray!50!white, draw=black] (17, -1) circle [radius=.1];
				\filldraw[fill=gray!50!white, draw=black] (18, -1) circle [radius=.1];
				\filldraw[fill=gray!50!white, draw=black] (19, -1) circle [radius=.1];
				\filldraw[fill=gray!50!white, draw=black] (20, -1) circle [radius=.1];
				\filldraw[fill=gray!50!white, draw=black] (21, -1) circle [radius=.1];
				\filldraw[fill=gray!50!white, draw=black] (22, -1) circle [radius=.1];
				
				\filldraw[fill=gray!50!white, draw=black] (17, 0) circle [radius=.1];
				\filldraw[fill=gray!50!white, draw=black] (18, 0) circle [radius=.1];
				\filldraw[fill=gray!50!white, draw=black] (19, 0) circle [radius=.1];
				\filldraw[fill=gray!50!white, draw=black] (20, 0) circle [radius=.1];
				\filldraw[fill=gray!50!white, draw=black] (21, 0) circle [radius=.1];
				\filldraw[fill=gray!50!white, draw=black] (22, 0) circle [radius=.1];
				
				\filldraw[fill=gray!50!white, draw=black] (17, 1) circle [radius=.1];
				\filldraw[fill=gray!50!white, draw=black] (18, 1) circle [radius=.1];
				\filldraw[fill=gray!50!white, draw=black] (19, 1) circle [radius=.1];
				\filldraw[fill=gray!50!white, draw=black] (20, 1) circle [radius=.1];
				\filldraw[fill=gray!50!white, draw=black] (21, 1) circle [radius=.1];
				\filldraw[fill=gray!50!white, draw=black] (22, 1) circle [radius=.1];
				
				\filldraw[fill=gray!50!white, draw=black] (17, 2) circle [radius=.1];
				\filldraw[fill=gray!50!white, draw=black] (18, 2) circle [radius=.1];
				\filldraw[fill=gray!50!white, draw=black] (19, 2) circle [radius=.1];
				\filldraw[fill=gray!50!white, draw=black] (20, 2) circle [radius=.1];
				\filldraw[fill=gray!50!white, draw=black] (21, 2) circle [radius=.1];
				\filldraw[fill=gray!50!white, draw=black] (22, 2) circle [radius=.1];

			\end{tikzpicture}
			
		\end{center}
		
		\caption{\label{arrows2}  Shown to the left is a vertex model with $\boldsymbol{J} = (2, 2, \ldots )$ and $\boldsymbol{J}$-step boundary data. Shown to the right is its horizontal complementation.} 
	\end{figure}
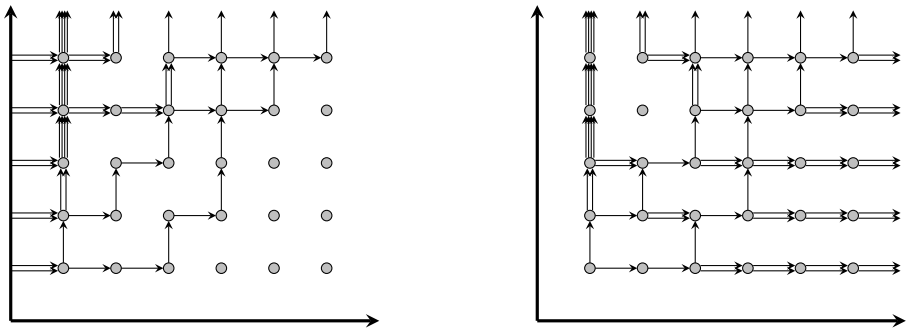

	Using the weights $\Psi_A$ and $\Theta_A$, we can explicitly sample this limiting complemented path ensemble as in \Cref{PathEnsembles}, namely, by randomly assigning arrow configurations to vertices in triangles of the form $\mathbb{T}_n = \big\{ (x, y) \in \mathbb{Z}_{> 0}^2 : x + y = n \}$. As previously, to extend an assignment from $\mathbb{T}_n$ to $\mathbb{T}_{n + 1}$, we must explain how to sample the last two coordinates $(i_2, h_2)_{(x, y)}$ of an arrow configuration at any vertex $(x, y) \in \mathbb{D}_n = \big\{ (x, y) \in \mathbb{Z}_{> 0}^2 : x + y = n + 1 \big\}$, given the first two coordinates $(i_1, h_1)_{(x, y)}$. This is done by producing $(i_2, h_2)_{(x, y)}$ from $(i_1, h_1)_{(x, y)}$ according to the transition probabilities
	\begin{flalign*}
		\mathbb{P} \big[ (i_2, h_2)_{(x, y)} \big| (i_1, h_1)_{(x, y)} \big] = \Psi_{A_y \omega_x} (i_1, h_1; i_2, h_2), \quad & \text{if $s_x = \infty$}; \\
		\mathbb{P} \big[ (i_2, h_2)_{(x, y)} \big| (i_1, h_1)_{(x, y)} \big] = \Theta_{A_y \omega_x} (i_1, h_1; i_2, h_2), \quad & \text{if $s_x = 0$}.
	\end{flalign*}

	\noindent That we use $\Psi_{A_y \omega_x}$ as a probability if $s_x = \infty$ and $\Theta_{A_y \omega_x}$ as one if $s_x = 0$ is in accordance with the limits considered in \Cref{llimit} and \Cref{llimit2}, respectively.
	
	Letting $n$ tend to $\infty$ then yields a random (horizontally complemented) path ensemble on all of $\mathbb{Z}_{> 0}^2$, which is the vertex model corresponding to the limit weights derived in \Cref{LimitWeight1}.

	\section{A $t$-Deformed Polynuclear Growth Model} 
	
	\label{Modelq} 
	
	In this section we apply a further limit, as the $A_i$ parameters to $\infty$, to the vertex model described in \Cref{ModelArrow}. By further taking $s_x = \infty$ for each $x$, we will see in \Cref{Growthq} this gives rise to a $t$-deformation of the polynuclear growth (PNG) model. Then, in \Cref{BoundaryTheta}, we will explain how one can incorporate boundary conditions in this growth model by taking the first several $s_x$ to be $0$. Throughout this section, we fix parameters $t \in [0, 1)$ and $\theta > 0$.

	\subsection{The $t$-PNG Model} 
	
	\label{Growthq}
	
	In this section we analyze a limit of the model considered in \Cref{ModelArrow} as the $A_y$ tend to $\infty$, which will give rise to a $t$-PNG model. Here, we take $s_x = \infty$ for each $x \ge 1$. To implement this limit, let $\varepsilon \in (0, 1)$ denote some parameter (which we will eventually let tend to $0$), and set 
	\begin{flalign}
		\label{ai2} 
		A_y = \frac{t}{1 - t} (\varepsilon \theta)^{-1}, \quad \text{for each integer $y \ge 1$}; \qquad \omega_x = (\varepsilon \theta)^{-1}, \quad \text{for each integer $x \ge 1$}.
	\end{flalign} 
	
	\noindent Further set 
	\begin{flalign}
		\label{ai} 
		A = A_y \omega_x = \displaystyle\frac{t}{1 - t} (\varepsilon \theta)^{-2}.
	\end{flalign}

	\noindent Let us analyze the weights $\Psi_A$ under \eqref{ai} for small $\varepsilon$.

	\begin{lem}
		
		\label{psia0} 
		
		The following statements hold for $\Psi_A (i_1, h_1; i_2, h_2)$ under \eqref{ai}. 
		
		\begin{enumerate}
			\item For $(i_1, h_1) = (0, 0)$, we have
			\begin{flalign}
				\label{psia00} 
				\Psi_A (0, 0; 0, 0) = 1 - \mathcal{O} (\varepsilon^2); \quad \Psi_A (0, 0; 1, 1) = (\varepsilon \theta)^2 - \mathcal{O} (\varepsilon^4); \quad \displaystyle\sum_{k = 2}^{\infty} \big| \Psi_A (0, 0; k, k) \big| = \mathcal{O} (\varepsilon^4). 
			\end{flalign}
			\item For $(i_1, h_1) = (1, 0)$ or $(i_1, h_1) = (0, 1)$, we have
			\begin{flalign}
				\label{psia10psia01}
				\begin{aligned} 
				&\Psi_A (1, 0; 1, 0) = 1 - \mathcal{O} (\varepsilon^2); \qquad \displaystyle\sum_{k = 2}^{\infty} \Psi_A (1, 0; k + 1, k) = \mathcal{O} (\varepsilon^2); \\
				& \Psi_A (0, 1; 0, 1) = 1 - \mathcal{O} (\varepsilon^2); \qquad \displaystyle\sum_{k = 2}^{\infty} \Psi_A (0, 1; k, k + 1) = \mathcal{O} (\varepsilon^2). 
				\end{aligned} 
			\end{flalign}
			\item For $(i_1, h_1) = (1, 1)$, we have
			\begin{flalign}
				\label{psia11}
				\Psi_A (1, 1; 0, 0) = 1 - t - \mathcal{O} (\varepsilon^2); \quad \Psi_A (1, 1; 1, 1) = t - \mathcal{O} (\varepsilon^2); \quad \displaystyle\sum_{k = 2}^{\infty} \Psi_A (1, 1; k, k) = \mathcal{O} (\varepsilon^2). 
			\end{flalign}
		\end{enumerate} 
	\end{lem} 
	
	\begin{proof}
		
		Let us show \eqref{psia00}. To that end, we insert $(i_1, h_1) = (0, 0)$ into \eqref{psia} to obtain
		\begin{flalign}
			\label{psiaih}
			\Psi_A (0, 0; i_2, h_2) = A^{-i_2} t^{i_2^2} \displaystyle\frac{(A^{-1} t^{i_2 + 1}; t)_{\infty}}{(t; t)_{i_2}} \textbf{1}_{h_1 = h_2 - i_2},
		\end{flalign} 
	
		\noindent where we have used the facts $h_1 = h_1 - i_1 = h_2 - i_2$ for $i_1 = 0$; that the sum on the right side of \eqref{psia} is supported on the term $k = 0$, due to the factor of $(t^{-i_1}; t)_k$ there; and the fact that $(t^{h_2 - i_2 + 1}; t)_{i_2} (t; t)_{h_1} = (t; t)_{h_2}$, since $h_2 = h_2 + i_1 = i_2 + h_1$. Setting $i_2 = 0$ or $i_2 = 1$, we find 
		\begin{flalign*} 
			& \Psi_A (0, 0; 0, 0) = (A^{-1} t; t)_{\infty} = 1 - \mathcal{O} (A^{-1}) = 1 - \mathcal{O} (\varepsilon^2); \\
			& \Psi_A (0, 0; 1, 1) = A^{-1} \displaystyle\frac{t}{1 - t} (A^{-1} t^2; t)_{\infty} = (\varepsilon \theta)^2 - \mathcal{O} (A^{-2}) = (\varepsilon \theta)^2 - \mathcal{O} (\varepsilon^4),
		\end{flalign*} 
	
		\noindent where for both statements we used the expression \eqref{ai} for $A$. This verifies the first and second statements of \eqref{psia00}. To verify the last, observe by \eqref{psiaih} that $\Psi_A (0, 0; k, k) = \mathcal{O} \big( A^{-k} (t; t)_k^{-1} \big)$ (where the implicit constant only depends on $t$ and not on $k$), yielding by the $t$-binomial theorem that
		\begin{flalign*}
			\displaystyle\sum_{k = 2}^{\infty} \Psi_A (0, 0; k, k) = \displaystyle\sum_{k = 2}^{\infty} \mathcal{O} \bigg( \displaystyle\frac{A^{-k}}{(t; t)_k} \bigg) = \mathcal{O} \bigg( \displaystyle\frac{A^{-2}}{(A^{-1}; t)_{\infty}}  \bigg) = \mathcal{O} (A^{-2}) = \mathcal{O} (\varepsilon^4).
		\end{flalign*}

		\noindent This establishes \eqref{psia00}; the proofs of \eqref{psia10psia01} and \eqref{psia11} are very similar and therefore omitted.
	\end{proof}

	Next, fix real numbers $\chi, \eta > 0$ and define the integers $X = X_{\varepsilon}$ and $Y = Y_{\varepsilon}$ by
	\begin{flalign}
		\label{xy} 
		X = \lceil \varepsilon^{-1} \chi \rceil; \qquad Y = \lceil \varepsilon^{-1} \eta \rceil. 
	\end{flalign}

	\noindent Let us use \Cref{psia0} to interpret the small $\varepsilon$ limit vertex model from \Cref{ModelArrow} on the rectangle $[1, X] \times [1, Y] \subset \mathbb{Z}_{> 0}^2$, under empty boundary data, with $(\xi_x, s_x) = (\omega_x, \infty)$ for each $x \ge 1$ and $A_y$ defined by \eqref{ai2}. Since $A = A_y \omega_x$, the first statement of \eqref{psia00} implies that, if $(i_1, h_1) = (0, 0)$, then with probability about $1 - \varepsilon^2 \theta^2$ we have $(i_2, h_2) = (0, 0)$. In view of the empty boundary data, this indicates that most vertices in $[1, X] \times [1, Y]$ have arrow configuration $(0, 0; 0, 0)$. However, with probability about $\varepsilon^2 \theta^2$, we have $(i_2, h_2) = (1, 1)$; in this case, a pair of exiting arrows is created, or \emph{nucleates}. Since $[1, X] \times [1, Y]$ constitutes $\mathcal{O} (\varepsilon^{-2})$ vertices, there are $\mathcal{O}(1)$ such nucleation events; in the limit as $\varepsilon$ tends to $0$, they become distributed according to a Poisson point process with intensity $\theta^2$.
	
	If $(i_1, h_1) = (1, 0)$, then \eqref{psia10psia01} implies that $(i_2, h_2) = (1, 0)$ almost deterministically. In particular, whenever a vertical exiting arrow is created it proceeds vertically until $h_1 \ne 0$, that is, until it meets a horizontal arrow. Similarly, if $(i_1, h_1) = (0, 1)$ then $(i_2, h_2) = (1, 0)$ almost deterministically, meaning that any created horizontal arrow proceeds horizontally until it meets a vertical arrow. 
	
	The event $(i_1, h_1) = (1, 1)$ corresponds to the collision of a horizontal and vertical arrow. Then, \eqref{psia11} implies that $(i_2, h_2) = (0, 0)$ with probability about $1 - t$; in this case, the two arrows annihilate each other. With the complementary probability of about $t$, we have $(i_2, h_2) = (1, 1)$, in which case the arrows pass through each other. This description gives rise to the following growth model.
	
	\begin{definition}
		
		\label{modelq} 
		
		The \emph{$t$-deformed polynuclear growth ($t$-PNG) model} on the rectangle $\mathcal{R} = \mathcal{R}_{\chi; \eta} = [0, \chi] \times [0, \eta] \subset \mathbb{R}^2$, with intensity $\theta^2$, is described as follows. 
		
		\begin{enumerate}
			\item Sample a Poisson point process with intensity $\theta^2$ on $\mathcal{R}$, denoted by $\mathcal{V} = \{ v_1, v_2, \ldots , v_K \}$, where $v_i = (x_i, y_i) \in \mathcal{R}$ for each index $i \in [1, K]$. 
			
			\item For each point $v \in \mathcal{V}$, draw two rays emanating from $v$, one directed north and the other directed east. Whenever two rays emanating from different vertices meet, the following occurs.
			\begin{enumerate}
				\item With probability $1 - t$, they annihilate each other.
				\item With probability $t$, they pass through each other.
			\end{enumerate}
		\end{enumerate}
	\end{definition}
	
	We refer to \Cref{model1} for a depiction. 
	
	\begin{rem}
		
		\label{modelquadrant} 
		
		Observe for $\chi > \chi' > 0$ and $\eta > \eta' > 0$ that the $t$-PNG models on $\mathcal{R}_{\chi; \eta}$ and $\mathcal{R}_{\chi'; \eta'}$ are consistent in the following sense. The restriction to $\mathcal{R}_{\chi'; \eta'}$ of the $t$-PNG model on $\mathcal{R}_{\chi; \eta}$ is the $t$-PNG model on $\mathcal{R}_{\chi'; \eta'}$. Therefore, one may take the limit as $\chi$ and $\eta$ tend to $\infty$ to obtain a $t$-PNG model on the (infinite) nonnegative quadrant $\mathbb{R}_{\ge 0}^2$.
		
	\end{rem}

	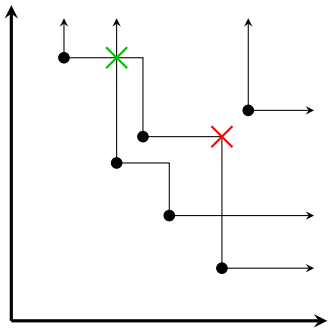
\begin{figure}
		
		\begin{center} 
			
			\begin{tikzpicture}[
				>=stealth,
				scale = .7
				]
				
				\draw[->, very thick] (0, 0) -- (0, 6);
				\draw[->, very thick] (0, 0) -- (6, 0);
				
				\draw[->] (3, 2) -- (5.75, 2);
				\draw[->] (3, 2) -- (3, 3) -- (2, 3) -- (2, 5.75);
				\draw[->] (4.5, 4) -- (5.75, 4);
				\draw[->] (4.5, 4) -- (4.5, 5.75);
				\draw[->] (1, 5) -- (2.5, 5) -- (2.5, 3.5) -- (4, 3.5) -- (4, 1) -- (5.75, 1);
				\draw[->] (1, 5) -- (1, 5.75);
				
				\draw[-, thick, green!75!black] (1.8, 4.8) -- (2.2, 5.2);
				\draw[-, thick, green!75!black] (2.2, 4.8) -- (1.8, 5.2);
				
				\draw[-, thick, red] (3.8, 3.3) -- (4.2, 3.7);
				\draw[-, thick, red] (4.2, 3.3) -- (3.8, 3.7);

				\filldraw[fill=black, draw=black] (4, 1) circle [radius=.1];
				\filldraw[fill=black, draw=black] (3, 2) circle [radius=.1];
				\filldraw[fill=black, draw=black] (2, 3) circle [radius=.1];
				\filldraw[fill=black, draw=black] (1, 5) circle [radius=.1];
				\filldraw[fill=black, draw=black] (2.5, 3.5) circle [radius=.1];
				\filldraw[fill=black, draw=black] (4.5, 4) circle [radius=.1];
				
			\end{tikzpicture}
			
		\end{center}
		
		\caption{\label{model1} Shown above is a possible sample of the $t$-PNG model. The nucleation events are the black points; the red cross is a location where two paths annihilate each other; and the green cross is a location where two paths cross through each other.}
	\end{figure}

	Under the $t$-PNG model for $t = 0$, two colliding rays must always annihilate each other. In this case, the model coincides with the standard PNG model analyzed in \cite{SIDP}; see also Section 4 of \cite{IP}. This $t = 0$ model is known to admit an interpretation through patience sorting \cite{LSPST}; we will provide an analogous interpretation for the more general $t$-PNG model in \Cref{ModelSort} below. 
	
	The following proposition more precisely states the convergence of the vertex model from \Cref{ModelArrow} to the $t$-PNG model. A heuristic for it was provided above \Cref{modelq}; we will give a more careful proof in \Cref{ConvergeModel} below.

	\begin{prop}
		
	\label{modelconverge} 
	
	Fix real numbers $\chi, \eta > 0$, and define $X, Y \in \mathbb{Z}_{> 0}$ as in \eqref{xy}. Consider the vertex model described in \Cref{ModelArrow} on the rectangle $[1, X] \times [1, Y] \subseteq \mathbb{Z}_{> 0}^2$, under empty boundary data, with $s_x = \infty$ for each $x \ge 1$ and $A_y$ chosen according to \eqref{ai}. When both its coordinates are multiplied by $\varepsilon$, this model converges, as $\varepsilon$ tends to $0$, to the $t$-PNG model with intensity $\theta^2$ on $\mathcal{R}_{\chi; \eta}$ from \Cref{modelq}. 
	\end{prop}

		Although we will not pursue this here, it should also be possible to introduce a multi-layer version of this $t$-PNG model through Dynamics 8 of \cite{NNDP} and a colored variant of this model (along the same lines of the colored particle systems analyzed in \cite{SVMST}). Let us also mention that it should be possible to derive this $t$-PNG model as a limit of either the $t$-Push TASEP \cite{NNDP,IMTD} or a more intricate version of the $t$-PNG model defined in Section 4.2 of \cite{AP}.

	\subsection{Boundary Conditions} 
	
	\label{BoundaryTheta} 
	
	In this section we fix an integer $m \ge 1$ and a sequence of $m$ positive real numbers $\boldsymbol{\beta} = (\beta_1, \beta_2, \ldots , \beta_m)$; further let $\varepsilon > 0$ be a small real parameter. We then consider the vertex model as described in \Cref{ModelArrow}, with parameters given by
	\begin{flalign}
		\label{msomega} 
		& (s_x, \omega_x) = (0, \beta_j^{-1}), \quad \text{for each $x \in [1, m]$}; \qquad (s_x, \omega_x) = \big( \infty, (\varepsilon \theta)^{-1} \big) \quad \text{for each $x > m$},
	\end{flalign} 

	\noindent and 
	\begin{flalign}
		\label{a2y}
		A_y = \displaystyle\frac{t}{1 - t} (\varepsilon \theta)^{-1}, \quad \text{for each integer $i \ge 1$}.
	\end{flalign}

	\noindent In particular, for $x > m$, we have $A_y \omega_x = A$ from \eqref{ai}. Thus, the dynamics of this model to the right of its $m$-th column are governed by the $\Psi_A$ weights, whose small $\varepsilon$ asymptotics are given by \Cref{psia0} (and give rise to the $t$-PNG model as in \Cref{modelq}, under scaling by $\varepsilon$). 
	
	For $x \le m$, we have $s_x = 0$ so, at and to the left of the $m$-th column, this model is governed by the $\Theta$ weights defined by \eqref{thetaa}. We will see these different $s_x$ parameters in the leftmost $m$ column of the model will give rise to a boundary condition for the $t$-PNG model; a similar phenomenon in the context of the asymmetric simple exclusion process (ASEP) and stochastic six-vertex model was observed in \cite{PTASSVM}.
	
	The below lemma analyzes the asymptotics of the $\Theta$ weights as $\varepsilon$ tends to $0$.

	\begin{lem}
		
		\label{theta0} 
		
		Fix a real number $\beta > 0$ and an integer $i \ge 0$. Denoting $B = t (1 - t)^{-1} (\varepsilon \theta \beta)^{-1}$, we have the following.
		
		\begin{enumerate} 
			
			\item For $(i_1, h_1) = (i, 0)$, we have
		\begin{flalign}
			\label{thetai0b}
			\begin{aligned}
			\Theta_B (i, 0; i, 0) = 1 & - \mathcal{O} (\varepsilon); \qquad \Theta_B (i, 0; i + 1, 0) = \varepsilon \theta \beta + \mathcal{O} (\varepsilon^2); \\
			& \displaystyle\sum_{k = 2}^{\infty} \Theta_B (i, 0; i + k, k) = \mathcal{O} (\varepsilon^2).
			\end{aligned} 
		\end{flalign} 
		
			\item For $(i_1, h_1) = (i, 1)$, we have
			\begin{flalign} 
				\label{thetai1b}
				\begin{aligned}
			 \Theta_B (i, 1; i - 1, 0) &= 1 - t^i - \mathcal{O} (\varepsilon); \qquad \Theta_B (i, 1; i, 1) = t^i - \mathcal{O} (\varepsilon); \\
			&  \displaystyle\sum_{k = 2}^{\infty} \Theta_B (i, 1; i  + k - 1, k) = \mathcal{O} (\varepsilon).
			\end{aligned}
			\end{flalign}

		\end{enumerate} 
		
	\end{lem}

	\begin{proof}
		
		By the choice of $B$, we have  
		\begin{flalign*}
			(-t^{h_2 - i_2 + 1} B^{-1}; t)_{\infty} = 1 + \mathcal{O} (\varepsilon); \qquad B^{-i_2} (B; t)_k \textbf{1}_{k \le i_2}= \textbf{1}_{k = i_2} t^{\binom{i_2}{2}} + \mathcal{O} (\varepsilon),
		\end{flalign*} 
	
		\noindent which when combined with the identity $t^{\binom{i_2}{2} + i_2} (t^{-i_2}; t)_{i_2} = (-1)^{i_2} (t; t)_{i_2}$ gives  
		\begin{flalign*}
			\displaystyle\frac{B^{-i_2}}{(-t^{h_2 - i_1 + 1} B^{-1}; t)_{\infty}} \displaystyle\sum_{k = 0}^{i_2} t^k \displaystyle\frac{(t^{-i_1}; t)_k (t^{-i_2}; t)_k (-B; t)_k}{(t; t)_k (t^{h_2 - i_2 + 1}; t)_k} = (-1)^{i_2} \displaystyle\frac{(t^{-i_1}; t)_{i_2}}{(t^{h_2 - i_2 + 1}; t)_{i_2}} + \mathcal{O} (\varepsilon).
		\end{flalign*}
	
		\noindent Inserting this into \eqref{thetaa} yields
		\begin{flalign}
			\label{thetab1}
			\Theta_B (i_1, h_1; i_2, h_2) = (-1)^{i_2} t^{\binom{i_2 + 1}{2} + i_2 h_1} \displaystyle\frac{(t; t)_{h_1}}{(t; t)_{h_2}} \displaystyle\frac{(t^{-i_1}; t)_{i_2}}{(t; t)_{i_2}} + \mathcal{O} (\varepsilon).
		\end{flalign}
		
		\noindent Next, observe that 
		\begin{flalign*}
			& (-1)^{i_2} \displaystyle\frac{(t^{-i_1}; t)_{i_2}}{(t; t)_{i_2}} = t^{\binom{i_1 - i_2 + 1}{2} - \binom{i_1 + 1}{2}} \displaystyle\frac{(t; t)_{i_1}}{(t; t)_{i_2} (t; t)_{i_1 - i_2}} \textbf{1}_{i_1 \ge i_2}; \\
			& \binom{i_2 + 1}{2} + \binom{i_1 - i_2 + 1}{2} - \binom{i_1 + 1}{2} + i_2 h_1 = i_2 (i_2 - i_1 + h_1) = i_2 h_2,
		\end{flalign*}
	
		\noindent where in the last equality we used the fact that $i_1 - h_1 = i_2 - h_2$. This, together with \eqref{thetab1} yields 
		\begin{flalign}
			\label{thetab} 
		\Theta_B (i_1, h_1; i_2, h_2) = t^{i_2 h_2} \displaystyle\frac{(t; t)_{h_1}}{(t; t)_{h_2}} \displaystyle\frac{(t; t)_{i_1}}{(t; t)_{i_2} (t; t)_{i_1 - i_2}} \textbf{1}_{i_1 \ge i_2} + \mathcal{O} (\varepsilon).
		\end{flalign} 
	
		\noindent By inserting $(i_1, h_1; i_2, h_2) \in \big\{ (i, 0; i, 0), (i, 1; i - 1, 0), (i, 1; i, 1) \big\}$ into \eqref{thetab}, we deduce the first statement of \eqref{thetai0b} and the first two statements of \eqref{thetai1b}. 
		
		To deduce the second statement of \eqref{thetai0b}, we insert $(i_1, h_1; i_2, h_2) = (i, 0; i + 1, 1)$ into \eqref{thetaa} to obtain
		\begin{flalign}
			\label{thetabi0i11} 
			\Theta_B (i, 0; i + 1, 1) = \displaystyle\frac{t^{\binom{i + 2}{2}} B^{-i - 1}}{(-t^{1 - i} B^{-1}; t)_{\infty} (t; t)_{i + 1} (1 - t)} \displaystyle\sum_{k = 0}^i t^k \displaystyle\frac{(t^{-i}; t)_k}{(t; t)_k} (t^{-i - 1}; t)_k (-B; t)_k (t^{k - i + 1}; t)_{i - k + 1},
		\end{flalign}
	
		\noindent where we have used the facts that the summand in \eqref{thetabi0i11} corresponding to $k = i + 1$ is equal to $0$ (since $(t^{-i}; t)_{i + 1} = 0$) and that $(t^{h_2 - i_2 + 1}; t)_{i_2} (t^{h_2 - i_2 + 1}; t)_k^{-1} = (t^{h_2 - i_2 + k + 1}; t)_{i_2 - k}$. Since 
		\begin{flalign*}
			& (-t^{1 - i} B^{-1}; t)_{\infty} = 1 + \mathcal{O} (\varepsilon); \qquad B^{- i - 1} (-B; t)_k \textbf{1}_{k \le i} = t^{\binom{i}{2}} B^{-1} \textbf{1}_{k = i} + \mathcal{O} (\varepsilon^2); \\
			& (-1)^i t^{\binom{i + 1}{2}} (t^{-i}; t)_i = (t; t)_i; \qquad (-1)^i t^{\binom{i + 2}{2}} \displaystyle\frac{(t^{-i - 1}; t)_i}{(t; t)_{i + 1}} = \displaystyle\frac{t}{1 - t}
		\end{flalign*}
		
		\noindent it follows that 
		\begin{flalign*} 
			\Theta_B (i, 0; i + 1, 1) = \displaystyle\frac{t}{1 - t} B^{-1} + \mathcal{O} (\varepsilon^2) = \varepsilon \theta \beta_j + \mathcal{O} (\varepsilon^2),
		\end{flalign*}
		
		\noindent which yields the second statement of \eqref{thetai0b}. The proofs of the third statements of \eqref{thetai0b} and \eqref{thetai1b} are very similar to those of the analogous estimates in \Cref{psia0} and are therefore omitted.
	\end{proof}

	Now, as in \Cref{Growthq}, we fix real numbers $\chi, \eta > 0$ and define the integers $X = X_{\varepsilon}  = \lceil \varepsilon^{-1} \chi \rceil$ and $Y = Y_{\varepsilon} = \lceil \varepsilon^{-1} \eta \rceil$ as in \eqref{xy}. Let us use \Cref{theta0} to interpret the behavior in the first $m$ columns of the vertex model from \Cref{ModelArrow} on the rectangle $[1, X] \times [1, Y]$, under empty boundary data with the parameter choices \eqref{msomega} and \eqref{a2y}, in the limit as $\varepsilon$ tends to $0$. 
	
	If $h_1 = 0$ at some vertex in the $k$-th column, then by the first statement of \eqref{thetai0b} we have $h_2 = 0$ with probability about $1 - \varepsilon \theta \beta_k$. In the leftmost column of the vertex model we have $h_1 = 0$ at all sites due to the empty boundary data; so, most vertices will also have $h_2 = 0$. However, with probability about $\varepsilon \theta \beta_1$, we have $(i_2, h_2) = (i_1 + 1, h_1 + 1)$, that is, a pair of a horizontal and vertical arrow nucleates (is created). Since any column of the model has $Y = \mathcal{O} (\varepsilon^{-1})$ vertices, there are $\mathcal{O} (1)$ such nucleation events along the leftmost column. In the limit as $\varepsilon$ tends to $0$, they become (after scaling the vertical coordinate by $\varepsilon$) distributed according to a Poisson point process with intensity $\theta \beta_1$. 
	
	A similar effect occurs in the $k$-th column, for any $k \in [2, m]$; the first statement of \Cref{theta0} again implies that pairs of horizontal and vertical paths nucleate along this column according to a Poisson process with intensity $\theta \beta_k$. However, now there may be some sites with $h_1 = 1$, corresponding to locations where a horizontal arrow enters the column. Letting $i$ denote the number of vertical arrows in the column at such a site, the second statement of \eqref{thetai1b} implies that with probability $t^i$ we have $(i_2, h_2) = (i, 1)$, meaning that this arrow ``passes through'' the column. The first statement of \eqref{thetai1b} implies that with the complementary probability $1 - t^i$ we have $(i_2, h_2) = (i - 1, 0)$, meaning that this arrow is annihilated, along with one vertical arrow in the column.\footnote{An equivalent interpretation is that the entering horizontal arrow attempts to pass through each of the $i$ vertical arrows in the column, one at a time. As in the $t$-PNG model, with probability $t$ this horizontal arrow successfully passes through the vertical arrow, and with probability $1 - t$ they annihilate each other.}
	
	These dynamics proceed in the first $m$ columns of the model. Since $m$ is uniformly bounded in $\varepsilon$, when we scale the horizontal coordinate by $\varepsilon$ and let $\varepsilon$ tend to $0$, these $m$ columns all converge to the $y$-axis, that is, the west boundary of the rectangle $\mathcal{R} = \mathcal{R}_{\chi; \eta} = [0, \chi] \times [0, \eta]$. This boundary therefore acts as an ``external source'' for paths, releasing a horizontal ray into the interior of $\mathcal{R}$ at every site along the $m$-th column at which $h_2 = 1$. In view of the choice \eqref{a2y} and \Cref{modelconverge}, the $t$-deformed PNG model then occurs in the interior of $\mathcal{R}$. This gives rise to the following definition.

	\begin{definition}
		
		\label{boundarypaths} 
		
		Fix an integer $m \ge 1$ and a sequence of positive real numbers $\boldsymbol{\beta} = (\beta_1, \beta_2, \ldots , \beta_m)$. The $t$-PNG model on $\mathcal{R}_{\chi; \eta}$ with intensity $\theta^2$, under \emph{$(\boldsymbol{\beta}; \theta; t)$-boundary conditions}, is the $t$-PNG model as described in \Cref{modelq}, with additional horizontal rays entering at points along the east boundary of $\mathcal{R}_{\chi, \eta}$, given by $\big\{ (0, \kappa_1), (0, \kappa_2), \ldots , (0, \kappa_r) \big\} \subset \{ 0 \} \times [0, \eta]$. Here, the sequence $\boldsymbol{\kappa} = (\kappa_1, \kappa_2, \ldots , \kappa_r) \subset [0, \eta]$ is random and sampled as follows.
		
		\begin{enumerate}
			
			\item Consider $m$ columns $\mathcal{C}_1, \mathcal{C}_2, \ldots , \mathcal{C}_m$, where $\mathcal{C}_j = \{ j - m \} \times [0, \eta] \subset \mathbb{R}^2$. For each $j \in [1, m]$, sample a Poisson point process on $[0, \eta]$, denoted by $\mathcal{Y}_j = (y_1, y_2, \ldots , y_{K(j)})$. Define $\mathcal{V}_j = (v_1, v_2, \ldots , v_{K(j)}) \subset \mathcal{C}_j$ by setting $v_h = (j - m, y_h)$ for each $1 \le h \le K(j)$. 
			
			\item For each point $v \in \mathcal{V}_1 \cup \mathcal{V}_2 \cup \cdots \cup \mathcal{V}_m$, draw two rays emanating from $v$, one directed north and the other directed east. Whenever a horizontal ray intersects a column $\mathcal{C}_j$ containing $i = i(v)$ vertical rays, the following occurs.
			\begin{enumerate}
				\item With probability $1 - t^i$, this horizontal ray and one of the vertical rays in $\mathcal{C}_j$ (that it intersects) are annihilated.
				\item With probability $t^i$, the horizontal ray passes through $\mathcal{C}_j$.
			\end{enumerate}
			
			\item Define $\boldsymbol{\kappa} = (\kappa_1, \kappa_2, \ldots , \kappa_r)$ by setting $(0, \kappa_1), (0, \kappa_2), \ldots , (0, \kappa_r)$ to be the vertices at which a horizontal ray exits through $\mathcal{C}_m$.
		\end{enumerate}

		We refer to \Cref{model2} for a depiction. 
			
		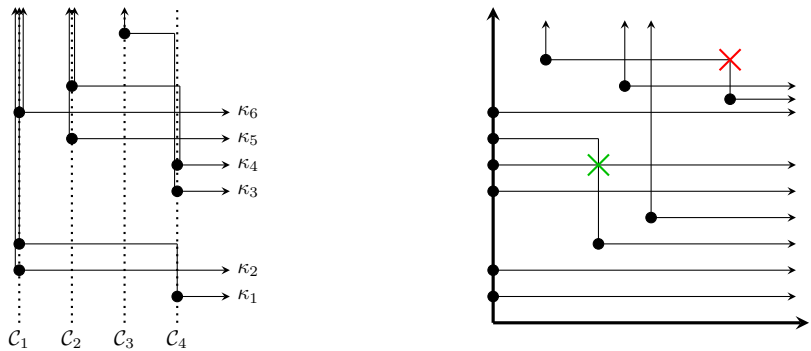
\begin{figure}
			
			\begin{center} 
				
				\begin{tikzpicture}[
					>=stealth,
					scale = .7
					]
					
					\draw[thick, dotted] (-9, 0) node[scale = .8, below]{$\mathcal{C}_1$} -- (-9, 6);
					\draw[thick, dotted] (-8, 0) node[scale = .8, below]{$\mathcal{C}_2$} -- (-8, 6);
					\draw[thick, dotted] (-7, 0) node[scale = .8, below]{$\mathcal{C}_3$} -- (-7, 6);
					\draw[thick, dotted] (-6, 0) node[scale = .8, below]{$\mathcal{C}_4$} -- (-6, 6);
				
					\filldraw[fill = black] (-9, 4) circle [radius = .1];
					\filldraw[fill = black] (-9, 1.5) circle [radius = .1];
					\filldraw[fill = black] (-9, 1) circle [radius = .1];
					
					\filldraw[fill = black] (-8, 3.5) circle [radius = .1];
					\filldraw[fill = black] (-8, 4.5) circle [radius = .1];
					
					\filldraw[fill = black] (-6, .5) circle [radius = .1];
					\filldraw[fill = black] (-6, 2.5) circle [radius = .1];
					\filldraw[fill = black] (-6, 3) circle [radius = .1];
					
					\filldraw[fill = black] (-7, 5.5) circle [radius = .1];
					
					\draw[->] (-9.075, 1) -- (-9.075, 6);
					\draw[->] (-9, 1.5) -- (-9, 6);
					\draw[->] (-8.925, 4) -- (-8.925, 6);
					
					\draw[->] (-9, 1) -- (-5, 1) node[right, scale = .8]{$\kappa_2$};
					\draw[-] (-9, 1.5) -- (-6, 1.5);
					\draw[-] (-6, .5) -- (-6, 1.5);
					
					\draw[->] (-9, 4) -- (-5, 4) node[right, scale = .8]{$\kappa_6$};
					
					\draw[->] (-8.05, 3.5) -- (-8.05, 6);
					\draw[->] (-7.95, 4.5) -- (-7.95, 6);
					
					\draw[->] (-7, 5.5) -- (-7, 6);
					
					\draw[->] (-8, 3.5) -- (-5, 3.5) node[right, scale = .8]{$\kappa_5$};
					
					\draw[->] (-6, .5) -- (-5, .5) node[right, scale = .8]{$\kappa_1$};
					\draw[->] (-6, 2.5) -- (-5, 2.5) node[right, scale = .8]{$\kappa_3$};
					\draw[->] (-6, 3) -- (-5, 3) node[right, scale = .8]{$\kappa_4$};
					
					\draw[-] (-7, 5.5) -- (-6.05, 5.5);
					
					\draw[-] (-8, 4.5) -- (-5.95, 4.5);
					
					\draw[-] (-5.95, 3) -- (-5.95, 4.5);
					\draw[-] (-6.05, 2.5) -- (-6.05, 5.5);

					\draw[->, very thick] (0, 0) -- (0, 6);
					\draw[->, very thick] (0, 0) -- (6, 0);
					
					\draw[->] (0, .5) -- (5.75, .5);
					\draw[->] (1, 5) -- (1, 5.75);
					\draw[-] (1, 5) -- (4.5, 5);
					\draw[-] (4.5, 4.25) -- (4.5, 5);
					\draw[->] (4.5, 4.25) -- (5.75, 4.25);
					\draw[->] (2.5, 4.5) -- (2.5, 5.75);
					\draw[->] (2.5, 4.5) -- (5.75, 4.5);
					\draw[-] (0, 3.5) -- (2, 3.5);
					\draw[-] (2, 1.5) -- (2, 3.5);
					\draw[->] (0, 4) -- (5.75, 4);
					\draw[->] (0, 1) -- (5.75, 1);
					
					\draw[->] (0, 2.5) -- (5.75, 2.5);
					\draw[->] (0, 3) -- (5.75, 3);
					\draw[->] (2, 1.5) -- (5.75, 1.5);
					\draw[->] (3, 2) -- (3, 5.75);
					\draw[->] (3, 2) -- (5.75, 2);
					
					\draw[-, thick, green!75!black] (1.8, 2.8) -- (2.2, 3.2);
					\draw[-, thick, green!75!black] (2.2, 2.8) -- (1.8, 3.2);
					
					\draw[-, thick, red] (4.3, 4.8) -- (4.7, 5.2);
					\draw[-, thick, red] (4.7, 4.8) -- (4.3, 5.2);

					\filldraw[fill = black] (0, .5) circle [radius = .1]; 
					\filldraw[fill = black] (0, 2.5) circle [radius = .1];
					\filldraw[fill = black] (0, 3) circle [radius = .1];
					\filldraw[fill = black] (0, 3.5) circle [radius = .1];
					\filldraw[fill = black] (0, 4) circle [radius = .1]; 
					\filldraw[fill=black, draw=black] (0, 1) circle [radius=.1];
					\filldraw[fill=black, draw=black] (3, 2) circle [radius=.1];
					\filldraw[fill=black, draw=black] (2, 1.5) circle [radius=.1];
					\filldraw[fill=black, draw=black] (1, 5) circle [radius=.1];
					\filldraw[fill=black, draw=black] (2.5, 4.5) circle [radius=.1];
					\filldraw[fill=black, draw=black] (4.5, 4.25) circle [radius=.1];
					
				\end{tikzpicture}
				
			\end{center}
			
			\caption{\label{model2} Shown to the left is the procedure described by \Cref{boundarypaths} used to sample $\boldsymbol{\kappa} = (\kappa_1, \kappa_2, \ldots , \kappa_r)$, where here $r = 6$. Shown to the right is a sample of the $t$-PNG model with the corresponding boundary data.}
		\end{figure}

	\end{definition}
	
	\begin{rem}
		
		Suppose $m = 1$, and denote $\beta = \theta \beta_1$. Then, the $t$-PNG model with $(\boldsymbol{\beta}; \theta; t)$-boundary data is the $t$-PNG model in which horizontal paths additionally enter through the $y$-axis according to a Poisson point process with intensity $\beta$. In the case $t = 0$, this model was studied in \cite{LDPGMES,FPGMES} (see also \cite{KTSP} for the general $m \ge 1$ case of a last passage percolation model slightly different from, but closely related to, the $t = 0$ PNG model). 
		 
	\end{rem}

	The following proposition then states convergence of the vertex model from \Cref{ModelArrow} to the $t$-PNG model with $(\boldsymbol{\beta}; \theta; t)$-boundary data. A heuristic for it was provided above; a careful proof for it is very similar to that of \Cref{modelconverge} and is thus omitted.

	\begin{prop}
		
		\label{modelconverge2} 
		
		Fix real numbers $\chi, \eta > 0$, and define $X, Y \in \mathbb{Z}_{> 0}$ as in \eqref{xy}. Consider the vertex model described in \Cref{ModelArrow} on the rectangle $[1, X] \times [1, Y] \subseteq \mathbb{Z}_{> 0}^2$, under empty boundary data, with parameters choices as in \eqref{msomega} and \eqref{a2y}. When both its coordinates are multiplied by $\varepsilon$, this model converges, as $\varepsilon$ tends to $0$, to the $t$-PNG model with intensity $\theta^2$ on $\mathcal{R}_{\chi; \eta}$, under $(\boldsymbol{\beta}; \theta; t)$-boundary data, from \Cref{boundarypaths}. 
		
	\end{prop}

	\section{Asymptotics}

	\label{Limit}
	
	In this section we describe asymptotic results for the $t$-PNG model. We begin in \Cref{PolynomialsFuse} by explaining a matching result between an observable of the $t$-PNG model with that of a Schur and of a Hall--Littlewood measure. We then analyze the large scale asymptotics for the $t$-PNG model in \Cref{LimitModel} and a limit to the KPZ equation in \Cref{LimitEquation}.

	\subsection{Matching With Schur and Hall--Littlewood Measures}
	
	\label{PolynomialsFuse}
	
	In this section we describe matching results between the $t$-PNG model and both the Schur and Hall--Littlewood measures, given by the following theorem. In the below, the height function $\mathfrak{H} (x, y)$ for the $t$-PNG model is defined to be the number of horizontal rays in the model that intersect the vertical interval $\{ x \} \times [0, y] \subset \mathbb{R}^2$.
	
	\begin{thm}
	
		\label{modelequation} 
		
		Fix an integer $m \ge 0$; real numbers $t \ge 0$ and $\chi, \eta, \theta > 0$; and a sequence of real numbers $\boldsymbol{\beta} = (\beta_1, \beta_2, \ldots , \beta_m) \subset \mathbb{R}_{> 0}$. Let $\mathfrak{H} (x, y)$ denote the height function for the $t$-PNG model on $\mathbb{R}_{> 0}^2$ with intensity $\theta^2$, under $(\boldsymbol{\beta}; \theta; t)$-boundary data, from \Cref{boundarypaths} (or \Cref{modelq}, if $m = 0$). Further let $\widetilde{\boldsymbol{\beta}} = \bigcup_{j = 1}^m \{ \beta_j, t \beta_j, \ldots  \}$. Define the specializations 
		\begin{flalign*} 
			\rho_1 = ( \boldsymbol{0} \boldsymbol{\mid} \boldsymbol{0} \boldsymbol{\mid} \eta \theta); \quad \rho_1' = \big(\boldsymbol{0} \boldsymbol{\mid} \boldsymbol{0} \boldsymbol{\mid} (1 - t) \eta \theta \big); \quad \rho_2 = \big( \boldsymbol{0} \boldsymbol{\mid} \widetilde{\boldsymbol{\beta}} \boldsymbol{\mid} (1 - t)^{-1} \chi \theta \big); \quad \rho_2' = (\boldsymbol{0} \boldsymbol{\mid} \boldsymbol{\beta} \boldsymbol{\mid} \chi \theta).
		\end{flalign*} 
	
		\noindent Then, the following two statements hold.
		
		\begin{enumerate}
			\item We have that 
			\begin{flalign}
				\label{1n}
				\mathbb{E} \Bigg[ \displaystyle\frac{1}{(-\zeta t^{-\mathfrak{H} (\chi, \eta)}; t)_{\infty}} \Bigg] = \mathbb{E}_{\SM} \Bigg[ \displaystyle\frac{1}{(-t^{-\ell (\lambda)} \zeta; t)_{\infty}} \displaystyle\prod_{j = 1}^{\ell (\lambda)} (1 + \zeta t^{\lambda_j - j}) \Bigg],
			\end{flalign}
		
			\noindent where the expectation on the left side is with respect to the $t$-PNG model, and the right side is with respect to the Schur measure with specializations $\rho_1$ and $\rho_2$.
			
			\item Let $\lambda$ denote a random partition sampled under the Hall--Littlewood measure with specializations $\rho_1'$ and $\rho_2'$. Then, $\mathfrak{H} (\chi, \eta)$ has the same law as $\ell (\lambda)$.
		\end{enumerate}
			
	\end{thm}

	\begin{proof}
		
		In what follows, we let $q \in [0, 1)$ be a real number. By \Cref{modelconverge2}, the $t$-PNG model is the limit as $\varepsilon$ tends to $0$ of the vertex model described in \Cref{ModelArrow}, with parameters given by \eqref{msomega} and \eqref{a2y}. The latter is the horizontal complementation of the fused stochastic higher spin vertex model with parameters $(t, \boldsymbol{u}, \boldsymbol{\xi}, \boldsymbol{r}, \boldsymbol{s})$ given by setting 
		\begin{flalign*}
			(u_y, r_y) = \bigg( \displaystyle\frac{t^{1-J}}{1 - t} (\varepsilon \theta)^{-1}, t^{-J/2} \bigg), & \qquad \text{for any integer $y \ge 1$}; \\
			(\xi_x, s_x) = \big( -(s' \beta_j)^{-1}, s' \big), \qquad \quad & \qquad \text{for any integer $x \in [1, m]$}; \\
			(\xi_x, s_x) = \big( s (\varepsilon \theta)^{-1}, s \big), \qquad \qquad & \qquad \text{for any integer $x \ge m$},
		\end{flalign*}
	
		\noindent first letting $(s, s')$ tend to $(\infty, 0)$, and then letting $J$ tend to $\infty$. Let $J \ge 1$ remain an arbitrary integer for the moment, and denote the height function for the associated vertex model by $\mathfrak{h}_{\FV(J)} (x, y)$. The height function for its horizontal complementation is then $h_{\CV (J)} (x, y) = Jx - \mathfrak{h}_{\FV(J)} (x, y)$.
		
		Next, set $X = X_{\varepsilon} = \lceil \varepsilon^{-1} \chi \rceil$ and $Y = Y_{\varepsilon} = \lceil \varepsilon^{-1} \eta \rceil$. Define the parameter (multi-)sets 
		\begin{flalign*}
			& \boldsymbol{x} = \boldsymbol{x}_{\varepsilon} = \bigcup_{i = 1}^Y \big\{ (1 - t) \varepsilon \theta, t (1 - t) \varepsilon \theta, t^2 (1 - t) \varepsilon \theta , \ldots , t^{J - 1} (1 - t) \varepsilon \theta \big\}; \\
			& \boldsymbol{\psi} = \boldsymbol{\psi}_{\varepsilon} = \bigcup_{i = 1}^X \{ \varepsilon \theta, t \varepsilon \theta, t^2 \varepsilon \theta, \ldots \}; \qquad \boldsymbol{\varphi} = \bigcup_{j = 1}^m \{ \beta_j, q \beta_j, q^2 \beta_j, \ldots \},
		\end{flalign*}
	
		\noindent where the geometric progressions with ratio $t$ do not depend on $i$ (meaning that $\boldsymbol{x}$ and $\boldsymbol{\psi}$ contain $Y$ and $X$ copies of them, respectively). Then, the above parameter sets $(t, \boldsymbol{u}, \boldsymbol{\xi}, \boldsymbol{r}, \boldsymbol{s})$ for $\FV (J)$ and $(\boldsymbol{x}, \boldsymbol{\alpha}, \boldsymbol{\psi})$ match in the sense of \Cref{parametersmodel}, by taking $(N, M)$ there to be $(X, Y)$ here; $\boldsymbol{x}$ there to be $\boldsymbol{x}$ here; $\boldsymbol{u}$ there to constitute $Y$ copies of $t^{1 - J} (1 - t)^{-1} (\varepsilon \theta)^{-1}$ here; $\boldsymbol{r}$ to constitute $J$ copies of $t^{-J/2}$; each $\widehat{\alpha}_i$ there to be $\varepsilon \theta$ here; each $\widehat{\beta}_j$ there to be $\beta_j$ here; and each $h_i$ there to be $\infty$ here.
		
		We may therefore apply \Cref{pvm}, with $(K, \zeta)$ there equal to $(XJ, q^{-XJ} \zeta)$ here to deduce 
		\begin{flalign}
			\label{cjvexpectation}
			\mathbb{E}_{\CV(J)} \Bigg[ \displaystyle\frac{1}{(-\zeta t^{-\mathfrak{h}_{\CV (J)} (X, Y)}; t)_{\infty}} \Bigg] = \mathbb{E}_{\MM} \Bigg[ \displaystyle\frac{1}{(-t^{- \ell (\lambda)} \zeta; t)_{\infty}} \displaystyle\prod_{j = 1}^{\ell (\lambda)} (1 + \zeta q^{\lambda_j} t^{-j}) \Bigg],
		\end{flalign}
		
		\noindent where the expectation on the right side is with respect to the Macdonald measure with specializations $(\boldsymbol{x}_{\varepsilon} \boldsymbol{\mid} \boldsymbol{0})$ and $(\boldsymbol{\psi}_{\varepsilon} \boldsymbol{\mid} \boldsymbol{\varphi})$, and we used the fact that $1 + \zeta q^{\lambda_j} t^j = 1 + \zeta t^j$ for $j > \ell (\lambda)$.
		
		Now let us take the limit of both sides of \eqref{cjvexpectation} as first $J$ tends to $\infty$, and then $\varepsilon$ tends to $0$. By \Cref{modelconverge2}, the left side converges to $\mathbb{E} \big[ (-\zeta t^{-\mathfrak{H} (\chi, \eta)}; t)_{\infty}^{-1} \big]$, where the expectation is with respect to the $t$-PNG model. To analyze the right side, observe by \Cref{convergealphagamma} that under this limit $(\boldsymbol{x}_{\varepsilon} \boldsymbol{\mid} \boldsymbol{0})$ and $(\boldsymbol{\psi}_{\varepsilon} \boldsymbol{\mid} \boldsymbol{\beta})$ converge to $\rho_1 = \big(\boldsymbol{0} \boldsymbol{\mid} \boldsymbol{0} \boldsymbol{\mid} \frac{1 - t}{1 - q} (1 - t^J) \eta \theta \big)$ and $\rho_3 = \big(\boldsymbol{0} \boldsymbol{\mid} \boldsymbol{\varphi} \boldsymbol{\mid} \frac{\chi \theta}{1 - q} \big)$, respectively. Thus, taking the limit first as $J$ tends to $\infty$ and next as $\varepsilon$ tends to $0$ in \eqref{cjvexpectation} gives
		\begin{flalign}
			\label{11n}
			\mathbb{E} \Bigg[ \displaystyle\frac{1}{(-\zeta t^{-{\mathfrak{H} (\chi, \eta)}; t)_{\infty}}} \Bigg] = \mathbb{E}_{\MM} \Bigg[ \displaystyle\frac{1}{(-t^{- \ell (\lambda)} \zeta; t)_{\infty}} \displaystyle\prod_{j = 1}^{\ell (\lambda)} (1 + \zeta q^{\lambda_j} t^{-j}) \Bigg],
		\end{flalign}
	
		\noindent where the expectation on the left side is with respect to the $t$-PNG model, and the expectation on the right side is with respect to the Macdonald measure with specializations $\rho_1$ and $\rho_3$. Applying \eqref{11n} with $q = t$ then yields \eqref{1n} (since at $q = t$ we have $\rho_3 = \rho_2$), by \Cref{measure0qt}; this establishes the first statement of the theorem. 
		
		To establish the second, we apply \Cref{pvm1}. This implies that $\mathfrak{h}_{\CV (J)} (X, Y)$ has the same law as $\ell (\lambda)$, where $\lambda$ is distributed according to a Hall--Littlewood measure with specializations $(\boldsymbol{x}_{\varepsilon} \boldsymbol{\mid} \boldsymbol{0})$ and $(\boldsymbol{\psi}_{\varepsilon} \boldsymbol{\mid} \boldsymbol{\beta})$ (where we used the fact that at $q = 0$ we have $\boldsymbol{\varphi} = \boldsymbol{\beta}$). Again taking the limits as first $J$ tends to $\infty$ and then as $\varepsilon$ tends to $0$, and applying \Cref{modelconverge2} and \Cref{convergealphagamma}, we deduce the second statement of the theorem.
	\end{proof}

	\subsection{Large Scale Asymptotics}

	\label{LimitModel}
	
	In this section we analyze the large scale asymptotics for the height function $\mathfrak{H} (\chi, \eta)$ of the $t$-PNG model, as $\chi$ and $\eta$ tend to $\infty$. To that end, we use \Cref{modelequation} to compare $\mathfrak{H} (\chi, \eta)$ with a Schur measure (which will correspond to the standard PNG model at $t = 0$). 
	
	To implement this, we recall from Definition 5.2 of \cite{SHSVMM} that two sequences of real-valued random variables $\{\mathfrak{a}_n \}$ and $\{ \mathfrak{b}_n \}$ are called \emph{asymptotically independent} if the following two conditions hold.
	
	\begin{enumerate}
		\item We have 
		\begin{flalign} 
			\label{anbn} 
			\displaystyle\lim_{n \rightarrow \infty} \displaystyle\sup_{z \in \mathbb{R}} \mathbb{P} [z < \mathfrak{a}_n \le z + 1] = 0, \quad \text{if and only if} \quad \displaystyle\lim_{n \rightarrow \infty} \displaystyle\sup_{z \in \mathbb{R}} \mathbb{P} [z < \mathfrak{b}_n \le z + 1] = 0.
		\end{flalign} 
		\item If both limits in \eqref{anbn} hold, then $\lim_{n \rightarrow \infty} \sup_{z \in \mathbb{R}} \big( \mathbb{P} [\mathfrak{a}_n \le z] - \mathbb{P} [\mathfrak{b}_n \le z] \big) = 0$. 
	\end{enumerate}
	
	The below lemma, which quickly follows from \Cref{modelequation} with results from \cite{SHSVMM}, states an asymptotic equivalence between the height function $\mathfrak{H} (\chi, \eta)$ and the length of a partition sampled from a Schur measure. 
	
	\begin{lem}
		
		\label{lengthheight}
		
		Recall the notation of \Cref{modelequation}; let $N \ge 1$ be an integer; assume that $(\chi, \eta) = (\chi_N, \eta_N) = (xN, yN)$; and sample a partition $\lambda \in \mathbb{Y}$ under the Schur measure with specializations $\rho_1 = (\boldsymbol{0} \boldsymbol{\mid} \boldsymbol{0} \boldsymbol{\mid} y \theta N)$ and $\rho_2 = \big( \boldsymbol{0} \boldsymbol{\mid} \boldsymbol{\beta} \boldsymbol{\mid} (1 - t)^{-1} x \theta N \big)$. Then $\mathfrak{H} (\chi, \eta) = \mathfrak{H} (xN, yN)$ is asymptotically equivalent to $\ell (\lambda)$ (where $N \ge 1$ is the index variable for both sequences).
	\end{lem}

	\begin{proof}
				
		Throughout this proof, we say that a sequence of random variables $\{ \mathfrak{a}_n \}$ is asymptotically equivalent to a sequence of cumulative distribution functions $\{ F_n \}$ if the following holds. Letting $\mathfrak{b}_n$ denote the random variable such that $\mathbb{P} [\mathfrak{b}_n \le z] = F_n (z)$, the random variable sequences $\{ \mathfrak{a}_n \}$ and $\{ \mathfrak{b}_n \}$ are asymptotically equivalent.
		
		Now, for any real number $z \in \mathbb{R}$, set 
		\begin{flalign}
		F_N (x) = \mathbb{E} \Bigg[\displaystyle\frac{1}{(-t^{z - \mathfrak{H} (\chi_N, \eta_N)}; t)_{\infty}} \Bigg]; \qquad G_N (z) = \mathbb{E} \Bigg[ \displaystyle\frac{1}{(-t^{z - \ell (\lambda)}; t)_{\infty}} \displaystyle\prod_{j = 1}^{\ell (\lambda)} (1 + t^{\lambda_j - j + z}) \Bigg]. 
		\end{flalign}
	
		\noindent By the $\zeta = t^z$ case of \eqref{1n}, we have $F_N (z) = G_N (z)$ for each $x \in \mathbb{R}$. It is quickly verified that $F_N$ and $G_N$ are nondecreasing in $z$, and further satisfy $\lim_{z \rightarrow -\infty} F_N (z) = 0 = \lim_{z \rightarrow -\infty} G_N (z)$ and $\lim_{z \rightarrow \infty} F_N (z) = 1 = \lim_{z \rightarrow \infty} G_N (z)$. 
		
		Corollary 5.7 of \cite{SHSVMM} implies that the random variable $\mathfrak{H} (xN, yN)$ is asymptotically equivalent to the function $F_N$, and that $\ell (\lambda)$ is asymptotically equivalent to $G_N$. The lemma then follows from the fact that $F_N = G_N$.	
	\end{proof}

	When $m = 0$, the sequence $\boldsymbol{\beta}$ is empty, and the Schur measure considered in \Cref{lengthheight} reduces to the \emph{Poissonized Plancherel measure}. The latter was analyzed in detail in \cite{DLLSRP,AMSG,DOPM} relating to the longest increasing subsequence of a random permutation and to the $t = 0$ PNG model. Its asymptotics are therefore well understood, which gives rise to the following result for the asymptotic behavior of the $t$-PNG model without boundary conditions.

	\begin{thm} 
		
		\label{hxnynlimit}
	
	Fix positive real numbers $x, y, \theta \in \mathbb{R}_{> 0}$ and $t \in [0, 1)$. Let $\mathfrak{H}$ denote the height function for the $t$-PNG model with intensity $\theta^2$ on $\mathbb{R}_{> 0}^2$ (without boundary conditions) from \Cref{modelq}. Then, 
	\begin{flalign*}
		\displaystyle\lim_{N \rightarrow \infty} \mathbb{P} \bigg[ \displaystyle\frac{\mathfrak{H} (xN, yN) - \mu N}{\sigma N^{1/3}} \le s \bigg] = F_{\TW} (s),
	\end{flalign*} 

	\noindent where $F_{\TW}$ denotes the Tracy--Widom Gaussian Unitary Ensemble (GUE) distribution, and $\mu = \mu (t, \theta, x, y)$ and $\sigma = \sigma(t, x, y, \theta)$ are defined by
	\begin{flalign*} 
		\mu = 2 \theta (xy)^{1/2}(1 - t)^{-1/2}; \qquad \sigma = \theta^{1/3} (xy)^{1/6} (1 - t)^{-1/6}.
	\end{flalign*}

	\end{thm}

	\begin{proof} 
	
	Under the Schur measure with specializations $\big( \boldsymbol{0} \boldsymbol{\mid} \boldsymbol{0} \boldsymbol{\mid} (1 - t) y \theta N \big)$ and $(\boldsymbol{0} \boldsymbol{\mid} \boldsymbol{0} \boldsymbol{\mid} x \theta N)$, Theorem 5 of \cite{AMSG} or Proposition 1.5 and Theorem 1.7 of \cite{DOPM} (see also Remark 2 of the survey \cite{IP}) gives
	\begin{flalign*} 
		\displaystyle\lim_{N \rightarrow \infty} \mathbb{P} \bigg[ \displaystyle\frac{\ell (\lambda) - \mu N}{\sigma N^{1/3}} \le s \bigg] = F_{\TW} (s).
	\end{flalign*}

	\noindent This, together with \Cref{lengthheight}, implies the theorem.
	\end{proof} 
	
	 The $m \ge 1$ case of \Cref{lengthheight} corresponds to the $t$-PNG model with boundary conditions, as in \Cref{boundarypaths}. Although we will not pursue this here, the associated Schur measure can be analyzed to access the large scale asymptotics for this model. For example, if $\beta_1 = \beta_2 = \cdots = \beta_m = \beta$ (or are more generally within $N^{-1/3}$ of one another), then the height function $\mathfrak{H} (x, y)$ will exhibit a \emph{Baik--Ben Arous--P\'{e}ch\'{e} transition} \cite{PTLECSCM} across a characteristic line. To the left of this line, it will exhibit $N^{1/2}$ fluctuations scaling to the largest eigenvalue of an $m \times m$ GUE matrix; to the right of this line, it will exhibit $N^{1/3}$ fluctuations scaling to the Tracy--Widom GUE distribution; and along this line it will converge to an interpolation between the two, known as a level $m$ Baik--Ben Arous--P\'{e}ch\'{e} distribution. For $m = 1$, this was established for the $t = 0$ PNG model in \cite{LDPGMES}.

	 \subsection{Limit to the KPZ Equation} 
	 
	 \label{LimitEquation} 
	 
	 The fact that our PNG model is dependent on a parameter $t \in [0, 1)$ enables us to consider its scaling limit as $t$ tends to $1$. In this section we explain how, under this scaling limit, the $t$-PNG height function converges to the Cole--Hopf solution of the \emph{Kardar--Parisi--Zhang (KPZ) equation} with narrow wedge initial data. The latter is defined as $\mathcal{H}_t (x) = - \log \mathcal{Z}_t (x)$, where $\mathcal{Z}_t (x)$ is the solution of the stochastic heat equation with multiplicative noise, given by
	 \begin{flalign*}
	 	\partial_t \mathcal{Z}_t (x) = \displaystyle\frac{1}{2} \partial_x^2 \mathcal{Z} (x) + \mathcal{Z}_t (x) \cdot \dot{\mathcal{W}}_t (x), \qquad \text{with initial data $\mathcal{Z}_0 (x) = \delta (x)$},
	 \end{flalign*} 
 
 	\noindent where $\dot{\mathcal{W}}_t (x)$ denotes space-time white noise, and $\delta (x)$ denotes the delta function; we refer to \cite{EUC,I} for surveys on the KPZ equation and universality class.
	 
	 Given this notation, we have the following theorem stating convergence of a (normalization) of the $t$-PNG height function $\mathfrak{H} (\chi, \eta)$ (recall \Cref{PolynomialsFuse}) to the solution $\mathcal{H}$ of the KPZ equation, in the limit as $t$ tends to $1$. Observe in this theorem that we also let the intensity of the model simultaneously tend to $\infty$, while keeping the coordinates $(\chi, \eta)$ fixed. We only outline the proof of the below theorem, since it follows from \Cref{modelequation} and results of \cite{AMSG,MMEP} in a similar way to what was done in the proof of Theorem 11.6  (and Remark 11.7) of \cite{ADPP}.

	\begin{thm}
		
		\label{equationlimit}
		
		Fix real numbers $\chi, \eta > 0$, let $\varepsilon > 0$ be a parameter, and denote
		\begin{flalign}
			\label{ttheta} 
			T = 2 \sqrt{\chi \eta}; \qquad t =  t_{\varepsilon} = e^{-\varepsilon}; \qquad \theta = \theta_{\varepsilon} = \varepsilon^{-3}.
		\end{flalign}
	
		\noindent Consider the $t$-PNG model with intensity $(1 - t) \theta^2$ on $\mathbb{R}_{> 0}^2$, as in \Cref{modelq}, and define the normalization $\mathfrak{n} (\chi, \eta)$ of its height function by 
		\begin{flalign} 
			\label{nchieta}
			\mathfrak{n}_{\varepsilon} (\chi, \eta) = \varepsilon \big( \mathfrak{H} (\chi, \eta) - \varepsilon^{-3} T \big) - \log \varepsilon.
		\end{flalign}
	
		\noindent Then, as $\varepsilon$ tends to $0$, the random variable $\mathfrak{n}_{\varepsilon} (\chi, \eta)$ converges weakly to $\frac{T}{24} - \mathcal{H}_T (0)$.
	\end{thm} 

	\begin{proof}[Proof (Outline)]
	
	By the discussion at the end of the proof of Theorem 11.6 of \cite{ADPP}, it suffices to verify that the limit as $\varepsilon$ tends to $0$ of the Laplace transform of $e^{\mathfrak{n}_{\varepsilon} (\chi, \eta)}$ is given by that of $e^{T / 24  - \mathcal{H}_T (0)} = e^{T/24} \mathcal{Z}_T (0)$.\footnote{Indeed, it is quickly verified from \eqref{zeta0} that the sequence of random variables $\{ e^{\mathfrak{n}_{\varepsilon} (\chi, \eta)} \}$ is tight. Since any probability distribution is uniquely characterized by its Laplace transform, \eqref{zeta0} also implies that any limit point must converge to $e^{T/24} \mathcal{Z}_T (0)$. Thus, by taking the logarithm, we deduce that $\mathfrak{n}_{\varepsilon} (\chi, \eta)$ converges to $\frac{T}{24} - \mathcal{H}_T (0)$.} So, for a fixed $\zeta_0 \in \mathbb{R}_{> 0}$ we will show that
	\begin{flalign}
	\label{zeta0} 
	\displaystyle\lim_{\varepsilon \rightarrow 0} \mathbb{E} \Big[ \exp \big(-\zeta_0 e^{\mathfrak{n}_{\varepsilon} (\chi, \eta)} \big) \Big] = \mathbb{E} \Big[ \exp \big( -\zeta_0 e^{T/24} \mathcal{Z}_T (0) \big) \Big].
	\end{flalign}

	\noindent The right side of \eqref{zeta0} is expressible in terms of the Airy point process\footnote{This is the determinantal point process on $\mathbb{R}$ with correlation kernel given by $K_{\text{Ai}} (x, y) = \int_0^{\infty} \text{Ai} (x + u) \text{Ai} (y + u) du$, where $\text{Ai} (x): \mathbb{R} \rightarrow \mathbb{R}$ denotes the Airy function.} $\mathcal{A} = (\mathfrak{a}_1, \mathfrak{a}_2, \ldots )$. In particular, Theorem 2.1 of \cite{MMEP} states that 
	\begin{flalign*}
		\mathbb{E} \Big[ \exp \big( - \zeta_0 e^{T / 24} \mathcal{Z}_T (0) \big) \Big] = \mathbb{E} \Bigg[ \displaystyle\prod_{j = 1}^{\infty} \displaystyle\frac{1}{1 + \zeta_0 \exp (2^{-1/3} T^{1/3} \mathfrak{a}_j)} \Bigg],
	\end{flalign*} 
	
	\noindent so it suffices to establish
	\begin{flalign}
		\label{n2lambda} 
		\displaystyle\lim_{\varepsilon \rightarrow 0} \mathbb{E} \Big[ \exp \big(-\zeta_0 \mathfrak{n}_{\varepsilon} (\chi, \eta) \big) \Big] =	\mathbb{E} \Bigg[ \displaystyle\prod_{j = 1}^{\infty} \displaystyle\frac{1}{1 + \zeta_0 \exp (2^{-1/3} T^{1/3} \mathfrak{a}_j)} \Bigg].
	\end{flalign}
	
	To that end, set $\zeta = \zeta_0 t^{2 \theta \sqrt{\chi \eta}} = t^{\theta T}$. Observe for any partition $\lambda = (\lambda_1, \lambda_2, \ldots , \lambda_{\ell}) \in \mathbb{Y}$ with conjugate (transpose) $\lambda' = (\lambda_1', \lambda_2', \ldots , \lambda_{\ell'}')$ that $\{ \lambda_i - i \}_{i \ge 1} \cup \{ j - \lambda_j' - 1 \}_{j \ge 1} = \mathbb{Z}$ (where $\lambda_i = 0$ and $\lambda_j' = 0$ for $i \ge \ell$ and $j \ge \ell'$, respectively). This implies
	\begin{flalign*} 
		\displaystyle\frac{1}{(-\zeta^{-\ell (\lambda)}; t)_{\infty}} \displaystyle\prod_{j = 1}^{\ell (\lambda)} (1 + \zeta t^{\lambda_j - j}) = \displaystyle\prod_{j = 1}^{\infty} \displaystyle\frac{1}{1 + \zeta t^{j - \lambda_j' - 1}}.
	\end{flalign*} 
	
	\noindent This, together with \eqref{1n} (with the $\theta$ there replaced by $(1 - t)^{-1/2} \theta$ here), yields
	\begin{flalign}
		\label{n1lambda} 
		\mathbb{E} \Bigg[ \displaystyle\frac{1}{(-\zeta t^{-\mathfrak{H} (\chi, \eta)}; t)_{\infty}} \Bigg] = \mathbb{E} \Bigg[ \displaystyle\prod_{j = 1}^{\infty} \displaystyle\frac{1}{1 + \zeta t^{j - \lambda_j' - 1}} \Bigg],
	\end{flalign}

	\noindent where the expectation on the left side is with respect to the $t$-PNG model with intensity $\theta^2 (1 - t)$ and that on the right side is with respect to the Schur measure with specializations $\rho_1 = ( \boldsymbol{0} \boldsymbol{\mid} \boldsymbol{0} \boldsymbol{\mid} \eta \theta)$ and $\rho_2 = ( \boldsymbol{0} \boldsymbol{\mid} \boldsymbol{0} \boldsymbol{\mid} \chi \theta)$.
	
	We will show that, as $\varepsilon$ tends to $0$, the left and right sides of \eqref{n1lambda} converge to those of \eqref{n2lambda}, respectively. We begin with the right sides, which will follow from results of \cite{AMSG}. In particular, from the choices $\zeta = \zeta_0 t^{2 \theta \sqrt{\chi \eta}}$, $t = e^{-\varepsilon}$, and $\varepsilon = \theta^{-1/3}$ (recall \eqref{ttheta}), we find that 
	\begin{flalign}
		\label{jlambda1} 
		\mathbb{E} \Bigg[ \displaystyle\prod_{j = 1}^{\infty} \displaystyle\frac{1}{1 + \zeta t^{j - \lambda_j' - 1}} \Bigg] = \mathbb{E} \Bigg[ \displaystyle\prod_{j = 1}^{\infty} \Bigg( 1 + \zeta_0 \exp \bigg( \displaystyle\frac{\lambda_j' - j - 2 \theta \sqrt{\chi \eta} + 1}{\theta^{1/3}} \bigg) \Bigg)^{-1} \Bigg].
	\end{flalign}
	
	\noindent Next, Theorem 4 of \cite{AMSG} states that $\big\{ (\chi \eta)^{-1/6} \theta^{-1/3} (\lambda_j' - 2 \theta \sqrt{\chi \eta}) \big\}_{j \ge 1}$ converges weakly to $\mathcal{A}$ as $\varepsilon$ tends to $0$. A slightly stronger form of this convergence, given by Proposition 4.3 of \cite{AMSG} (see also the proof of Theorem 11.6 of \cite{ADPP}), quickly implies that 
	\begin{flalign}
		\label{jlambda2}
		\displaystyle\lim_{\varepsilon \rightarrow 0} \Bigg[ \displaystyle\prod_{j = 1}^{\infty} \Bigg(1 + \zeta_0 \exp \bigg( \displaystyle\frac{\lambda_j' - j - 2 \theta \sqrt{\chi \eta} + 1}{\theta^{1/3}} \bigg) \Bigg)^{-1} \Bigg] = \mathbb{E} \Bigg[ \displaystyle\prod_{j = 1}^{\infty} \displaystyle\frac{1}{1 + \zeta_0 \exp \big( (\chi \eta)^{1/6} \mathfrak{a}_j \big)} \Bigg].
	\end{flalign}

	\noindent By the choice of $T = 2 \sqrt{\chi \eta}$, \eqref{jlambda1}, and \eqref{jlambda2} together imply that the right side of \eqref{n1lambda} converges to that of \eqref{n2lambda} as $\varepsilon$ tends to $0$. 
	
	Next, we analyze the left side of \eqref{n1lambda}. The $t$-binomial theorem, \eqref{ttheta}, and \eqref{nchieta} together give 
	\begin{flalign*}
		\displaystyle\frac{1}{(-\zeta t^{-\mathfrak{H} (\chi, \eta)}; t)_{\infty}} = \displaystyle\sum_{j = 0}^{\infty} \displaystyle\frac{(-\zeta)^j t^{-j \mathfrak{H} (\chi, \eta)}}{(t; t)_j} & = \displaystyle\sum_{j = 0}^{\infty} \displaystyle\frac{(1 - t)^j}{(t; t)_j} \Bigg( \displaystyle\frac{\zeta_0}{t - 1} \exp \bigg( \displaystyle\frac{1}{\theta^{1/3}} \big( \mathfrak{H} (\chi, \eta) - 2 \theta \sqrt{\chi \eta} \big) \bigg) \Bigg)^j \\
		& = \displaystyle\sum_{j = 0}^{\infty} \displaystyle\frac{(1 - t)^j}{(t; t)_j} \bigg( \displaystyle\frac{\zeta_0 e^{\mathfrak{n}_{\varepsilon} (\chi, \eta)}}{\varepsilon (t - 1)} \bigg)^j.
	\end{flalign*}
	
	\noindent Since
	\begin{flalign*} 
	\displaystyle\lim_{\varepsilon \rightarrow 0} \displaystyle\frac{(1 - t)^j}{(t; t)_j} = \displaystyle\frac{1}{j!}; \qquad \displaystyle\lim_{\varepsilon \rightarrow 0} \varepsilon (t - 1) = -1,
	\end{flalign*} 

	\noindent it follows that 
	\begin{flalign}
		\label{zetatn}
		\displaystyle\lim_{\varepsilon \rightarrow 0} \displaystyle\frac{1}{(-\zeta t^{-\mathfrak{H} (\chi, \eta)}; t)_{\infty}} = \displaystyle\sum_{j = 0}^{\infty} \displaystyle\frac{1}{j!} \big( - \zeta_0 e^{\mathfrak{n} (\chi, \eta)} \big)^j = \exp \big(- \zeta_0 e^{\mathfrak{n} (\chi, \eta)} \big).
	\end{flalign}
	
	\noindent This indicates that the left side of \eqref{n1lambda} converges to \eqref{n2lambda} as $\varepsilon$ tends to $0$. Hence, \eqref{n1lambda}, \eqref{zetatn}, \eqref{jlambda1}, and \eqref{jlambda2} together imply \eqref{n2lambda} and thus the theorem.
	\end{proof}

	Although we will not pursue this here, let us mention that a similar scaling limit as considered in \Cref{equationlimit}, for the $t$-PNG model with $(\boldsymbol{\beta}; \theta; t)$-boundary conditions should give rise to the solution of the KPZ equation with spiked initial data, as considered in \cite{FEF}.

	\appendix 
	
	\section{The $t$-PNG Model and Patience Sorting} 
	
	\label{ModelSort}
	
	In this section we provide an interpretation for the $t$-PNG model through patience sorting ``with errors.'' Before explaining this in more detail, we first recall the standard patience sorting algorithm; see Section 1.1 of \cite{LSPST}. Starting with a deck of $N$ cards labeled $\{ 1, 2, \ldots , N \}$, one begins drawing cards from it and sorting them into piles as follows. 
	
	\begin{enumerate} 
		\item \label{1} Suppose the card drawn has label $i$, and search for the pile with the smallest top card that is greater than $i$.
		\begin{enumerate}
			\item If such a pile exists, then place card $i$ on top of that pile.
			\item If no such pile exists, create a new pile consisting of card $i$.
		\end{enumerate}
		\item Repeat this procedure until all cards are sorted into piles.
	\end{enumerate}

	Observe in particular that, if these piles are ordered according to their time of creation, then their top cards are increasing. Thus, the search from part \ref{1} of this algorithm scans through the piles in order and stops upon reaching one whose top card exceeds $i$. 

	\begin{example}
		
		\label{6n} 
		
		Suppose $N = 6$ and the deck is ordered $(5, 2, 1, 3, 4, 6)$ from top to bottom. Then, after the first card is drawn, the set of piles is $ \big\{ (5) \big\}$; after the second, it is $\big\{ (2, 5) \big\}$; after the third, it is $\big\{ (1, 2, 5) \big\}$; after the fourth, it is $\big\{ (1, 2, 5), (3) \big\}$; after the fifth, it is $\big\{ (1, 2, 5), (3), (4) \big\}$; and after the sixth, it is $\big\{ (1, 2, 5), (3), (4), (6) \big\}$. 
		
	\end{example}
	
	The $t = 0$ PNG model with intensity $\theta^2$ is known to be closely related to patience sorting, applied to a uniformly randomly shuffled deck of $N$ cards, where $N$ is selected according to an independent exponential distribution with parameter $\theta^2$. In particular, the total number of piles created under the patience sorting algorithm (equivalently, the longest increasing subsequence of the deck) has the same law as the height function $\mathfrak{H} (1, 1)$ (recall the beginning of \Cref{PolynomialsFuse}) for the $t = 0$ PNG model. The analogous equivalence for the $t$-PNG model will be with the following variant of patience sorting that allows for random ``errors'' to occur with probability $t$. 
 	
 	\begin{definition} 
 		
 		\label{sortn}
 		
 		Starting with a deck of $N$ cards labeled $\{ 1, 2, \ldots , N \}$, the \emph{patience sorting algorithm with error probability $t$} draws cards from the deck and sorts them into piles as follows. 
 	\begin{enumerate}
		\item Suppose the card drawn has label $i$, and consider all piles $\mathcal{P}_1, \mathcal{P}_2, \ldots , \mathcal{P}_g$ whose top cards are greater than $i$; denote their top cards by $c_1, c_2, \ldots , c_g$, respectively, where $c_1 < c_2 < \cdots < c_g$. Set $k = 1$.
		\begin{enumerate}
			\item \label{cg} Suppose $k \le g$.
			\begin{enumerate}  
				\item With probability $1 - t$, place card $i$ on top of pile $\mathcal{P}_k$.
				\item With probability $t$, ``miss'' this pile by changing $k$ to $k + 1$ and repeating step \ref{cg}.
				\end{enumerate}
			\item If $k > g$, then create a new pile consisting of card $i$.
		\end{enumerate}
		\item Repeat this procedure until all cards are sorted into piles.
 	\end{enumerate}
 
 \end{definition} 
	
	Observe in the case $t = 0$ (corresponding to no ``misses''), \Cref{sortn} reduces to the original patience sorting algorithm described above. 
	
	\begin{example}
		
		\label{6n1} 
		
		Again suppose $N = 6$ and the deck is ordered $(5, 2, 1, 3, 4, 6)$ from top to bottom. After the first card is drawn, a pile $(5)$ is deterministically formed. After the second card is drawn, with probability $1 - t$ it is placed on top of this pile (forming $(2, 5)$); with probability $t$, pile $(5)$ is ``missed,'' and card $2$ is placed in its own pile (forming $\big\{ (2), (5) \big\}$). Suppose that the latter event occurs. The third card is then placed on pile $(2)$ with probability $1 - t$; on pile $(5)$ with probability $t - t^2$; and in its own pile with probability $t^2$. One continues in this way until all cards are placed.
		
	\end{example}

	The following proposition explains a relation between the $t$-PNG height function and the number of piles created under applying this variant of patience sorting to a random permutation; its proof more generally explains that broken lines in the former directly correspond to piles in the latter.
	
	\begin{prop}
		
		\label{modelsort} 
		
		Fix parameters $t \ge 0$ and $\theta > 0$; let $N$ be a $\theta^2$-exponentially distributed random variable; and apply the patience sorting algorithm with error probability $t$ (from \Cref{sortn}) to a uniformly randomly shuffled deck with $N$ cards. The number of piles created under this algorithm has the same law as the height function $\mathfrak{H} (1, 1)$ of the $t$-PNG model with intensity $\theta^2$ on $[0, 1] \times [0, 1]$, under empty boundary conditions (from \Cref{modelq}). 
		
	\end{prop} 
 	
 	\begin{proof}
 		
 		Sample the $t$-PNG model on $[0, 1] \times [0, 1]$ through \Cref{modelq}, and denote the associated Poisson point process (corresponding to the locations of nucleation events) by $\mathcal{V} = (v_1, v_2, \ldots, v_N) \subset [0, 1] \times [0, 1]$; then $N$ is a $\theta^2$-exponentially distributed random variable. Order $\mathcal{V}$ so that $v_i = (x_i, y_{\sigma (i)})$, where $x_1 < x_2 < \cdots < x_N$ and $y_1 < y_2 < \cdots < y_N$. Then, $\sigma$ is a uniformly random permutation on $\{ 1, 2, \ldots , N \}$; we associate it with the order of the deck to be sorted.
 		
 		A \emph{broken line} in the $t$-PNG model is defined to be a maximal increasing sequence $(i_1, i_2, \ldots , i_k) \subseteq \{ 1, 2, \ldots , N \}$ such that the horizontal ray emanating from $v_{i_j}$ annihilates with the vertical one emanating from $v_{i_{j + 1}}$, for each $j \in [1, k - 1]$. For example, in the sample depicted in \Cref{model1}, $\sigma = (6, 3, 4, 2, 1, 5)$ and there are three broken lines given by $(1, 3, 5)$, $(2, 4)$, and $(6)$. We associate with any broken line $(i_1, i_2, \ldots , i_k)$ a pile of cards $\big( \sigma (i_k), \sigma (i_{k - 1}), \ldots , \sigma (i_1) \big)$ (ordered from top to bottom). When a nucleation occurs at some point $v_i = (x_i, y_{\sigma (i)})$, the vertical ray emanating from $v_i$ can collide with a horizontal ray along a broken line; the latter corresponds to some pile $\mathcal{P}$, with top card greater than $i$. With probability $1 - t$, the two rays annihilate each other, meaning that card $\sigma (i)$ is appended to the top of $\mathcal{P}$. With probability $t$, the two rays pass through each other, meaning that pile $\mathcal{P}$ is ``missed'' when sorting $\sigma (i)$, and we repeat the procedure on the next broken line that intersects the vertical ray emanating from $v_i$. 
 		
 		These sorting dynamics induced by the $t$-PNG model coincide with those of the patience sorting algorithm with error probability $t$. Thus, the family of broken lines sampled under the former has the same law as the family of piles created under applying the latter to $\sigma$. This implies the proposition, since $\mathfrak{H} (1, 1)$ counts the number of such broken lines. 
 	\end{proof}

	\section{Proof of \Cref{modelconverge}}
	
	\label{ConvergeModel} 
	
	In this section we establish \Cref{modelconverge}. Let $\mathcal{E}$ denote a (complemented) path ensemble on $\mathbb{Z}_{> 0}^2$, sampled under the vertex model described in \Cref{modelconverge}. We will couple $\mathcal{E}$ with an ensemble $\mathcal{F}$ sampled from a slightly different vertex model that can be more directly seen to converge to the $t$-PNG process. The weights of this latter vertex model are $\Phi (i_1, h_1; i_2, h_2)$, defined by setting
	\begin{flalign}
		\label{probabilityf} 
		\begin{aligned}
			\Phi (0, 0; 0, 0) = 1 - (\theta \varepsilon)^2; \qquad & \Phi (0, 0; 1, 1) = (\theta \varepsilon)^2; \qquad \Phi (1, 0; 1, 0) = 1; \qquad \Phi (0, 1; 0, 1) = 1; \\
			& \Phi (1, 1; 0, 0) = 1 - t; \qquad \Phi (1, 1; 1, 1) = t,
		\end{aligned} 
	\end{flalign}
	
	\noindent and $\Phi (i_1, h_1; i_2, h_2) = 0$ for all other integer quadruples $(i_1, h_1; i_2, h_2)$. 
	
	It is quickly verified that a random ensemble $\mathcal{F}$ sampled under the vertex model with weights \eqref{probabilityf}, under empty boundary data, converges to the $t$-PNG model, as $\varepsilon$ tends to $0$. Indeed, by the first two probabilities in \eqref{probabilityf}, upon scaling the rectangle $[1, X] \times [1, Y]$ by $\varepsilon$, the law for the set of locations with arrow configuration $(0, 0; 1, 1)$ converges to a Poisson point process on $\mathcal{R}_{\chi; \theta}$ with intensity $\theta^2$. Under the $t$-PNG model, these correspond to nucleation events when a vertical and horizontal ray are created. By the second two probaiblities in \eqref{probabilityf}, these rays proceed until meeting another ray. By the last two probabilities in \eqref{probabilityf}, when a horizontal ray collides with a vertical one, they are annihilated with probability $1 - t$ and continue through each other with probability $t$. This description matches with that of the $t$-PNG model provided in \Cref{modelq}.
	
	Thus, it remains to couple $\mathcal{E}$ and $\mathcal{F}$ on $[1, X] \times [1, Y]$ off of an event with probability $o(1)$, as $\varepsilon$ tends to $0$. To that end, it suffices to establish lemma. In the below, for any vertex $v \in [1, X] \times [1, Y]$ and path ensemble $\mathcal{G} \in \{ \mathcal{E}, \mathcal{F} \}$, we let $\big( i_1 (v; \mathcal{G}), h_1 (v; \mathcal{G}); i_2 (v; \mathcal{G}), h_2 (v; \mathcal{G}) \big)$ denote the arrow configuration at $v$ under  $\mathcal{G}$.
	
	\begin{lem} 
		
		\label{ab} 
		
		The following two statements hold as $\varepsilon$ tends to $0$.
		
		\begin{enumerate} 
			\item Let $\mathscr{A}$ denote the event that there exist at most $\varepsilon^{-3/2} = o(\varepsilon^{-2})$ vertices $v \in [1, X] \times [1, Y]$ with $\big( i_1 (v; \mathcal{E}), h_1 (v; \mathcal{E}) \big) \ne (0, 0)$. Then, $\mathbb{P} [\mathscr{A}] = 1 - o(1)$.
			\item Let $\mathscr{B}$ denote the event that there does not exist any vertex $v \in [1, X] \times [1, Y]$ with arrow configuration satisfying $\max \big\{ i_1 (v), h_1 (v) \big\} \ge 2$. Then, $\mathbb{P} [\mathscr{B}] = 1 - o(1)$. 
		\end{enumerate}
		
	\end{lem} 
	
	\begin{proof}[Proof of \Cref{modelconverge} Assuming \Cref{ab}]
		For any $v \in [1, X] \times [1, Y]$ such that we have $\big( i_1 (v; \mathcal{E}), h_1 (v; \mathcal{E}) \big) = \big( i_1 (v; \mathcal{F}), h_1 (v; \mathcal{F}) \big) \in \big\{ (0, 0), (1, 0), (0, 1), (1, 1) \big\}$, \Cref{psia0} and \eqref{probabilityf} together yield a coupling between $\mathcal{E}$ and $\mathcal{F}$ so that $\big( i_2 (v; \mathcal{E}), h_2 (v; \mathcal{E}) \big) = \big( i_2 (v; \mathcal{F}), h_2 (v; \mathcal{F}) \big)$ with probability at least $1 - \mathcal{O} (\varepsilon^2)$. By the first part of \Cref{psia0} and the first two statements of \eqref{probabilityf}, this coupling probability is improved to $1 - \mathcal{O} (\varepsilon^4)$ if $\big( i_1 (v; \mathcal{E}), h_1 (v; \mathcal{E}) \big) = (0, 0) = \big( i_1 (v; \mathcal{F}), h_1 (v; \mathcal{F}) \big)$. 
		
		Since the empty boundary conditions for $\mathcal{E}$ and $\mathcal{F}$ coincide, we may apply a union bound to couple $\mathcal{E}$ and $\mathcal{F}$ with probability at least $1 - (V + 1) \mathcal{O} (\varepsilon^2)$, where $V$ denotes the number of vertices $v \in [1, X] \times [1, Y]$ such that $i_1 (v; \mathcal{E}) + h_1 (v; \mathcal{E}) \ge 1$. Restricting to the event $\mathscr{A} \cap \mathscr{B}$ from \Cref{ab}, we have $V = o(\varepsilon^{-2})$, meaning that we may couple $\mathscr{E}$ and $\mathscr{F}$ to coincide on $[1, X] \times [1, Y]$ with probability at least $\mathbb{P} [\mathscr{A} \cap \mathscr{B}] - o(1) = 1 - o(1)$. As mentioned above, this gives the proposition.
	\end{proof} 
	
	To verify the bounds $\mathbb{P} [\mathscr{A}] = 1 - o(1)$ and $\mathbb{P} [\mathscr{B}] = 1 - o(1)$, for any integer $D \ge 0$ we define the set $\mathcal{V}_d \subset \mathbb{Z}_{> 0}^2$, the integer $V_D$, and event $\mathscr{B}_d$ by 
	\begin{flalign*}
		\mathcal{V}_D & = \big\{ v = (x, y) \in \mathbb{Z}_{> 0}^2: x + y = D, i_1 (v; \mathcal{E}) + h_1 (v; \mathcal{E}) \ge 1 \big\}; \qquad V_D = |\mathcal{V}_D|; \\
		\mathscr{B}_D & = \bigcap_{d = 0}^D \bigcap_{v \in \mathcal{V}_d} \Big\{ \max \big\{ i_1 (v; \mathcal{E}), h_1 (v; \mathcal{E}) \big\} \le 1 \Big\}.
	\end{flalign*}
	
	The following lemma provides inductive estimates on $V_D$ and on the probability of $\mathscr{B}_D$. 
	
	\begin{lem} 
		
		\label{bv} 
		
		There exists a constant $C = C (t, \theta) > 1$ such that
		\begin{flalign}
			\label{bdvd} 
			\mathbb{E} \big[ \textbf{\emph{1}}_{\mathscr{B}_D} | V_{D + 1} - V_D| \big] \le C D \varepsilon^2; \qquad \mathbb{P} [\mathscr{B}_{D + 1}] \ge \mathbb{P} [\mathscr{B}_D] - C \Big( \varepsilon^2 \ \mathbb{E} \big[ \textbf{\emph{1}}_{\mathscr{B}_D} V_D  \big] + D \varepsilon^4 \big).
		\end{flalign}
		
	\end{lem} 
	
	\begin{proof} 
		
		To verify the first statement of \eqref{bdvd}, observe on the event $\mathscr{B}_D$ that the quantity $V_{D + 1} - V_D = |\mathcal{V}_{D + 1}| - |\mathcal{V}_D|$ is bounded from above by the number of vertices $v = (x, y) \in \mathbb{Z}_{> 0}$ with $x + y = D$ such that we either have $\big( i_1 (v; \mathcal{E}), h_1 (v; \mathcal{E}); i_2 (v; \mathcal{E}), h_2 (v; \mathcal{E}) \big) = (0, 0; 1, 1)$ or $\max \big\{ i_2 (v; \mathcal{E}), h_2 (v; \mathcal{E}) \big\} \ge 2$. By \Cref{psia0}, there exists a constant $C > 1$ such that the probability of any $v$ satisfying either event is at most $C \varepsilon^2$. Applying a union bound over all $v = (x, y)$ with $x + y = D$ yields the first statement of \eqref{bdvd}.
		
		To verify the second, we again restrict to the event $\mathscr{B}_D$. Observe that the event $\mathscr{B}_{D + 1}$ does not hold only if there exists some vertex $v \in \mathbb{Z}_{> 0}^2$ with $x + y = D$ such that $\max \big\{ i_1 (v; \mathcal{E}), h_1 (v; \mathcal{E}) \big\} \le 1$ and $\max \big\{ i_2 (v; \mathcal{E}); h_2 (v; \mathcal{E}) \big\} \ge 2$. By \Cref{psia0}, there exist a constant $C > 1$ such that the probability of any $v$ satisfying this event is at most $C \varepsilon^2$. Moreover, by the first statement of \Cref{psia0}, this probability at most $C \varepsilon^4$ if $\big( i_1 (v), h_1 (v) \big) = (0, 0)$. Applying a union bound over at most $D$ vertices satisfying the latter statement and at most $V_D$ remaining ones then yields the second statement of \eqref{bdvd}. 
	\end{proof} 
	
	Now we can establish \Cref{ab}.
	
	\begin{proof}[Proof of \Cref{ab}]
		
		Let us use \eqref{bdvd} to bound $\mathbb{P} [\mathscr{A}]$ and $\mathbb{P} [\mathscr{B}]$. To that end, observe for any integer $D \ge 1$ that 
		\begin{flalign}
			\label{vd} 
			\mathbb{E} \big[ \textbf{1}_{\mathscr{B}_D} V_D \big] = \displaystyle\sum_{d = 0}^{D - 1} \mathbb{E} \big[ \textbf{1}_{\mathscr{B}_{d + 1}} V_{d + 1} - \textbf{1}_{\mathscr{B}_d} V_d \big] \le \displaystyle\sum_{d = 0}^{D - 1} \mathbb{E} \big[ \textbf{1}_{\mathscr{B}_d} (V_{d + 1} - V_d) \big] \le C D^2 \varepsilon^2,
		\end{flalign}
		
		\noindent where to deduce the second bound we used the fact that $\mathscr{B}_{d + 1} \subseteq \mathscr{B}_d$ and $V_{d + 1} \ge 0$, and to deduce the third we used the first statement of \eqref{bdvd}. By \eqref{vd} and the second statement of \eqref{bdvd}, we obtain
		\begin{flalign}
			\label{bd} 
			\begin{aligned} 
				1 - \mathbb{P} [\mathscr{B}_D] = \mathbb{P} [\mathscr{B}_0] - \mathbb{P} [\mathscr{B}_D] = \displaystyle\sum_{d = 0}^{D - 1} \big( \mathbb{P} [\mathscr{B}_d] - \mathbb{P} [\mathscr{B}_{d + 1}] \big) & \le C \displaystyle\sum_{d = 0}^{D - 1} \Big( \varepsilon^2 \mathbb{E} \big[ \textbf{1}_{\mathscr{B}_d} V_d \big] + D \varepsilon^4 \Big) \\
				& \le C \varepsilon^4 (D^3 + D^2) \le 2 C \varepsilon^4 D^3.
			\end{aligned} 
		\end{flalign} 
		
		\noindent By taking $D = X + Y = \mathcal{O} (\varepsilon^{-1})$, it follows from \eqref{bd} that 
		\begin{flalign}
			\label{b} 
			\mathbb{P} [\mathscr{B}] \ge \mathbb{P} [\mathscr{B}_{X + Y}] \ge 1 - 2 C \varepsilon^4 (X + Y)^3 = 1 - \mathcal{O} (\varepsilon) = 1 - o(1).
		\end{flalign} 
		
		Moreover, denoting the complement of any event $E$ by $E^c$, we have 
		\begin{flalign}
			\label{a} 
			\begin{aligned} 
				\mathbb{P} [\mathscr{A}^c] \le \mathbb{P} \Bigg[ \displaystyle\sum_{D = 0}^{X + Y - 1} V_D \ge \varepsilon^{-3/2} \Bigg] & \le \mathbb{P}\Bigg[ \bigg\{ \displaystyle\sum_{D = 0}^{X + Y} V_D \ge \varepsilon^{-3/2} \bigg\} \cap \mathscr{B}_{X + Y} \Bigg] + \mathbb{P} [\mathscr{B}_{X + Y}^c] \\
				& \le \varepsilon^{3/2} \mathbb{E} \Bigg[ \displaystyle\sum_{D = 0}^{X + Y} \textbf{1}_{\mathscr{B}_D} V_D \Bigg] + \mathbb{P} [\mathscr{B}_{X + Y}^c] \\
				& = \varepsilon^{3/2} \displaystyle\sum_{D = 0}^{X + Y} \mathbb{E} \big[ \textbf{1}_{\mathscr{B}_D} V_D \big] + \mathcal{O} (\varepsilon) \\
				& \le C (X + Y)^3 \varepsilon^{7/2} + \mathcal{O} (\varepsilon) = \mathcal{O} (\varepsilon^{1/2}) = o(1).
			\end{aligned} 
		\end{flalign} 
		
		\noindent Here, to deduce the first bound we used the definitions of $\mathscr{A}$, $\mathcal{V}_D$, and $V_D$; to deduce the second we applied a union bound; to deduce the third, we used with the fact that $\mathscr{B}_{X + Y} \subseteq \mathscr{B}_D$ for $D \le X + Y$, together with a Markov estimate; to deduce the fourth we applied \eqref{b}; to deduce the fifth  we applied \eqref{vd}; and to deduce the sixth we used the fact that $X + Y = \mathcal{O} (\varepsilon^{-1})$.
		
		Since \eqref{a} and \eqref{b} imply $\mathbb{P} [\mathscr{A}] = 1 - o(1)$ and $\mathbb{P} [\mathscr{B}] = 1 - o(1)$, they yield the proposition.
	\end{proof}

	\section{Matching Expectations}
	
	\label{EquationProof}

	In this appendix we provide an alternative, direct proof of Proposition \ref{pvm} in the case of Schur measures $(q=t)$, when all parameters $r_i = s_i = t^{-1/2}$, and for $M=N$. The proof is carried out by noticing that a certain partition function \eqref{hybrid-pf} in the quadrant is equal to the expectation on the left hand side of \eqref{fvexpectation}, and that this partition function may be evaluated as an $N\times N$ determinant, borrowing a result from \cite{WZJ}. Performing the expansion of this determinant over the Schur basis via the Cauchy--Binet identity, we then obtain the right hand side of \eqref{fvexpectation}, with $q=t$.
	
	Extending this result to generic Macdonald measures, generic higher spin weights, and $M\not=N$, as in \eqref{fvexpectation}, is then straightforward. The passage to generic Macdonald measures is achieved by noting that the right hand side of \eqref{fvexpectation} is in fact independent of $q$, and therefore equal to the Schur expectation; this is an easy consequence of acting with Macdonald difference operators on the Macdonald Cauchy kernel. Passing to the general spin setting, with arbitrary $r_i$ and $s_i$, is achieved by performing fusion of the partition function \eqref{hybrid-pf}. Finally, the case $M\not=N$ may be accessed by certain reductions of the match \eqref{final-match}, as we briefly mention in Section \ref{ssec:final}.
	
	%
	%

	\subsection{Reduction to $t$-Boson Model}
	
	Fix integers $j_1,j_2 \in \{0,1\}$ and $i_1,i_2 \in \mathbb{Z}_{\geq 0}$. We define
	\begin{align}
		\label{L-reduce}
		\lim_{s \rightarrow 0}
		L_{x}(i_1,j_1;i_2,j_2 \boldsymbol{\mid} t^{-1/2}, s)
		(-s)^{-j_2}
		=
		\mathcal{L}_x(i_1,j_1;i_2,j_2).
	\end{align}
	We denote these weights graphically by
	\begin{align}
		\label{L-vert}
		\mathcal{L}_x(i_1,j_1;i_2,j_2)
		=
		\tikz{0.7}{
			\node[left,circle,scale=.75,draw]  at (-1.75,0) {$x$};
			\draw[lgray,line width=1.5pt,->] (-1,0) -- (1,0);
			\draw[lgray,line width=4pt,->] (0,-1) -- (0,1);
			\node[left] at (-1,0) {\tiny $j_1$};\node[right] at (1,0) {\tiny $j_2$};
			\node[below] at (0,-1) {\tiny $i_1$};\node[above] at (0,1) {\tiny $i_2$};
		},
		\qquad
		j_1,j_2 \in \{0,1\},
		\qquad
		i_1,i_2 \in \mathbb{Z}_{\geq 0}.
	\end{align}
	The vertex \eqref{L-vert} vanishes unless $i_1+j_1 = i_2+j_2$; when this constraint is met, we obtain the following table of nonzero weights:
	\begin{align}
		\label{rank1-weights}
		\begin{tabular}{|c|c|c|c|}
			\hline
			\quad
			\tikz{0.7}{
				\draw[lgray,line width=1.5pt,->] (-1,0) -- (1,0);
				\draw[lgray,line width=4pt,->] (0,-1) -- (0,1);
				\node[left] at (-1,0) {\tiny $0$};\node[right] at (1,0) {\tiny $0$};
				\node[below] at (0,-1) {\tiny $i$};\node[above] at (0,1) {\tiny $i$};
			}
			\quad
			&
			\quad
			\tikz{0.7}{
				\draw[lgray,line width=1.5pt,->] (-1,0) -- (1,0);
				\draw[lgray,line width=4pt,->] (0,-1) -- (0,1);
				\node[left] at (-1,0) {\tiny $0$};\node[right] at (1,0) {\tiny $1$};
				\node[below] at (0,-1) {\tiny $i$};\node[above] at (0,1) {\tiny $i-1$};
			}
			\quad
			&
			\quad
			\tikz{0.7}{
				\draw[lgray,line width=1.5pt,->] (-1,0) -- (1,0);
				\draw[lgray,line width=4pt,->] (0,-1) -- (0,1);
				\node[left] at (-1,0) {\tiny $1$};\node[right] at (1,0) {\tiny $0$};
				\node[below] at (0,-1) {\tiny $i$};\node[above] at (0,1) {\tiny $i+1$};
			}
			\quad
			&
			\quad
			\tikz{0.7}{
				\draw[lgray,line width=1.5pt,->] (-1,0) -- (1,0);
				\draw[lgray,line width=4pt,->] (0,-1) -- (0,1);
				\node[left] at (-1,0) {\tiny $1$};\node[right] at (1,0) {\tiny $1$};
				\node[below] at (0,-1) {\tiny $i$};\node[above] at (0,1) {\tiny $i$};
			}
			\\[1.3cm]
			\quad
			$1$
			\quad
			& 
			\quad
			$x(1-t^i)$
			\quad
			&
			\quad
			$1$
			\quad
			& 
			\quad
			$x$
			\quad
			\\[0.7cm]
			\hline
		\end{tabular} 
	\end{align}

	\subsection{Reduction to Stochastic Six-Vertex Model}
	
	Fix integers $i_1,i_2,j_1,j_2 \in \{0,1\}$. We define
	\begin{align}
		\label{R-reduce}
		L_{t^{-1/2}\cdot x/y}(i_1,j_1;i_2,j_2 \boldsymbol{\mid} t^{-1/2}, t^{-1/2})
		=
		\mathcal{R}_{y/x}(i_1,j_1;i_2,j_2).
	\end{align}
	We denote these weights graphically by
	\begin{align}
		\label{R-vert}
		\mathcal{R}_{y/x}(i_1,j_1;i_2,j_2)
		=
		\tikz{0.6}{
			\node[left,circle,scale=0.75,draw]  at (-1.75,0) {$x$};
			\node[below,circle,scale=0.75,draw] at (0,-1.75) {$y$};
			\draw[lgray,line width=1.5pt,->] (-1,0) -- (1,0);
			\draw[lgray,line width=1.5pt,->] (0,-1) -- (0,1);
			\node[left] at (-1,0) {\tiny $j$};\node[right] at (1,0) {\tiny $\ell$};
			\node[below] at (0,-1) {\tiny $i$};\node[above] at (0,1) {\tiny $k$};
		},
		\qquad
		i_1,i_2,j_1,j_2 \in \{0,1\}.
	\end{align}
	The vertex \eqref{R-vert} vanishes unless $i_1+j_1 = i_2+j_2$; when this constraint is met, we obtain the following table of nonzero weights:
	\begin{align}
		\label{six-vert}
		\begin{tabular}{|c|c|c|}
			\hline
			\quad
			\tikz{0.6}{
				\draw[lgray,line width=1.5pt,->] (-1,0) -- (1,0);
				\draw[lgray,line width=1.5pt,->] (0,-1) -- (0,1);
				\node[left] at (-1,0) {\tiny $0$};\node[right] at (1,0) {\tiny $0$};
				\node[below] at (0,-1) {\tiny $0$};\node[above] at (0,1) {\tiny $0$};
			}
			\quad
			&
			\quad
			\tikz{0.6}{
				\draw[lgray,line width=1.5pt,->] (-1,0) -- (1,0);
				\draw[lgray,line width=1.5pt,->] (0,-1) -- (0,1);
				\node[left] at (-1,0) {\tiny $0$};\node[right] at (1,0) {\tiny $0$};
				\node[below] at (0,-1) {\tiny $1$};\node[above] at (0,1) {\tiny $1$};
			}
			\quad
			&
			\quad
			\tikz{0.6}{
				\draw[lgray,line width=1.5pt,->] (-1,0) -- (1,0);
				\draw[lgray,line width=1.5pt,->] (0,-1) -- (0,1);
				\node[left] at (-1,0) {\tiny $0$};\node[right] at (1,0) {\tiny $1$};
				\node[below] at (0,-1) {\tiny $1$};\node[above] at (0,1) {\tiny $0$};
			}
			\quad
			\\[1.3cm]
			\quad
			$1$
			\quad
			& 
			\quad
			$\dfrac{t(1-y/x)}{1-ty/x}$
			\quad
			& 
			\quad
			$\dfrac{1-t}{1-ty/x}$
			\quad
			\\[0.7cm]
			\hline
			\quad
			\tikz{0.6}{
				\draw[lgray,line width=1.5pt,->] (-1,0) -- (1,0);
				\draw[lgray,line width=1.5pt,->] (0,-1) -- (0,1);
				\node[left] at (-1,0) {\tiny $1$};\node[right] at (1,0) {\tiny $1$};
				\node[below] at (0,-1) {\tiny $1$};\node[above] at (0,1) {\tiny $1$};
			}
			\quad
			&
			\quad
			\tikz{0.6}{
				\draw[lgray,line width=1.5pt,->] (-1,0) -- (1,0);
				\draw[lgray,line width=1.5pt,->] (0,-1) -- (0,1);
				\node[left] at (-1,0) {\tiny $1$};\node[right] at (1,0) {\tiny $1$};
				\node[below] at (0,-1) {\tiny $0$};\node[above] at (0,1) {\tiny $0$};
			}
			\quad
			&
			\quad
			\tikz{0.6}{
				\draw[lgray,line width=1.5pt,->] (-1,0) -- (1,0);
				\draw[lgray,line width=1.5pt,->] (0,-1) -- (0,1);
				\node[left] at (-1,0) {\tiny $1$};\node[right] at (1,0) {\tiny $0$};
				\node[below] at (0,-1) {\tiny $0$};\node[above] at (0,1) {\tiny $1$};
			}
			\quad
			\\[1.3cm]
			\quad
			$1$
			\quad
			& 
			\quad
			$\dfrac{1-y/x}{1-ty/x}$
			\quad
			&
			\quad
			$\dfrac{(1-t)y/x}{1-ty/x}$
			\quad 
			\\[0.7cm]
			\hline
		\end{tabular}
	\end{align}

	\noindent Observe that these weights are stochastic, that is,
	\begin{flalign*}
		\displaystyle\sum_{i_2, j_2} \mathcal{R}_{y / x} (i_1, j_1; i_2, j_2) = 1.
	\end{flalign*}
	
	\subsection{Partition Function $Z(x_1,\dots,x_N;y_1,\dots,y_N;k)$}
	
	Fix two alphabets $(x_1,\dots,x_N) \in \mathbb{C}^N$ and $(y_1,\dots,y_N) \in \mathbb{C}^N$, and an integer $k \in \mathbb{Z}_{\geq 0}$. Define the following partition function in the quadrant:
	\begin{align}
		\label{hybrid-pf}
		Z\left(x_1,\dots,x_N;y_1,\dots,y_N;k\right)
		=
		\tikz{0.75}{
			\foreach\y in {1,...,5}{
				\draw[lgray,line width=1.5pt,->] (0,\y) -- (7,\y);
			}
			\foreach\x in {1,...,5}{
				\draw[lgray,line width=1.5pt,->] (\x,0) -- (\x,7);
			}
			\draw[lgray,line width=4pt,->] (6,0) -- (6,6) -- (0,6);
			\node[above] at (5,7) {$1$};
			\node[above] at (4,7) {$1$};
			\node[above] at (3,7) {$1$};
			\node[above] at (2,7) {$1$};
			\node[above] at (1,7) {$1$};
			\node[below] at (6,0) {$k$};
			\node[below] at (5,0) {$0$};
			\node[below] at (4,0) {$0$};
			\node[below] at (3,0) {$0$};
			\node[below] at (2,0) {$0$};
			\node[below] at (1,0) {$0$};
			\node[below,circle,scale=0.75,draw] at (5,-0.75) {$y_N$};
			\node[below] at (3.5,-0.75) {$\cdots$};
			\node[below] at (2.5,-0.75) {$\cdots$};
			\node[below,circle,scale=0.75,draw] at (1,-0.75) {$y_1$};
			\node[left] at (0,1) {$1$};
			\node[left] at (0,2) {$1$};
			\node[left] at (0,3) {$1$};
			\node[left] at (0,4) {$1$};
			\node[left] at (0,5) {$1$};
			\node[left] at (0,6) {$k$};
			\node[left,circle,scale=0.75,draw] at (-0.75,1) {$\overline{x}_1$};
			\node[left] at (-0.75,2.5) {$\vdots$};
			\node[left] at (-0.75,3.5) {$\vdots$};
			\node[left,circle,scale=0.75,draw] at (-0.75,5) {$\overline{x}_N$};
			\node[right] at (7,1) {$0$};
			\node[right] at (7,2) {$0$};
			\node[right] at (7,3) {$0$};
			\node[right] at (7,4) {$0$};
			\node[right] at (7,5) {$0$};
		}
	\end{align}
	Vertices in the topmost row and the rightmost column are of the form \eqref{L-vert}; all other vertices are given by \eqref{R-vert}. The variables associated to horizontal lines are reciprocated; namely, we set $\overline{x}_a = 1/x_a$, for all $1 \leq a \leq N$. Partition functions of the form \eqref{hybrid-pf} were originally studied in \cite{WZJ}, where they appeared in connection with Cauchy identities and as a one-parameter generalization of domain wall partition functions \cite{Izergin,Korepin}. 
	
	We now show that \eqref{hybrid-pf} admits a nice probabilistic interpretation. Fix parameters $t$, $(x_1,\dots,x_N)$ and $(y_1,\dots,y_N)$ such that each vertex \eqref{R-vert} within \eqref{hybrid-pf} has real weight in the interval $[0,1]$. We may then associate to any configuration $\mathfrak{C}$ in the (finite) quadrant $[1,N] \times [1,N]$ (that does not include the thick arrow along the northeast boundary in \eqref{hybrid-pf}) a probability weight $\mathbb{P}_{6{\rm v}}(\mathfrak{C})$, defined as the product of weights of all vertices within $\mathfrak{C}$.
	\begin{prop}
		Let $\mathfrak{h}(N,N)$ denote the height function assigned to the vertex $(N,N)$ within the stochastic six-vertex model. We have that
		\begin{align}
			\label{6v-E}
			Z(x_1,\dots,x_N;y_1,\dots,y_N;k)
			=
			\prod_{a=1}^N
			y_a
			\cdot
			\mathbb{E}_{6{\rm v}}
			\left[
			\frac{(t^{k+1};t)_{\infty}}{(t^{k+1+\mathfrak{h}(N,N)};t)_{\infty}}
			\right],
		\end{align}
		where the expectation is taken with respect to the measure $\mathbb{P}_{6{\rm v}}$ defined above.
	\end{prop}
	
	\begin{proof}
		We begin by decomposing the partition function \eqref{hybrid-pf} along the edges where vertices of the types \eqref{L-vert} and \eqref{R-vert} meet. This produces the equation
		\begin{multline}
			\label{Z-expand}
			Z(x_1,\dots,x_N;y_1,\dots,y_N;k)
			\\
			= \sum_{\{i_1,\dots,i_N\} \in \{0,1\}^N}
			\sum_{\{j_1,\dots,j_N\} \in \{0,1\}^N}
			\mathbb{P}_{6{\rm v}}(i_1,\dots,i_N;j_1,\dots,j_N)
			H_k(j_1,\dots,j_N;i_1,\dots,i_N),
		\end{multline}
		where we have defined two new partition functions. The first is given by
		\begin{align}
			\label{square}
			\mathbb{P}_{6{\rm v}}(i_1,\dots,i_N;j_1,\dots,j_N)
			=
			\tikz{0.8}{
				\foreach\y in {1,...,5}{
					\draw[lgray,line width=1.5pt,->] (0,\y) -- (6,\y);
				}
				\foreach\x in {1,...,5}{
					\draw[lgray,line width=1.5pt,->] (\x,0) -- (\x,6);
				}
				\node[above] at (5,6) {$i_N$};
				\node[above] at (4,6) {$\cdots$};
				\node[above] at (3,6) {$\cdots$};
				\node[above] at (2,6) {$\cdots$};
				\node[above] at (1,6) {$i_1$};
				\node[below] at (5,0) {$0$};
				\node[below] at (4,0) {$0$};
				\node[below] at (3,0) {$0$};
				\node[below] at (2,0) {$0$};
				\node[below] at (1,0) {$0$};
				\node[below,circle,scale=0.75,draw] at (5,-0.75) {$y_N$};
				\node[below] at (3.5,-0.75) {$\cdots$};
				\node[below] at (2.5,-0.75) {$\cdots$};
				\node[below,circle,scale=0.75,draw] at (1,-0.75) {$y_1$};
				\node[left] at (0,1) {$1$};
				\node[left] at (0,2) {$1$};
				\node[left] at (0,3) {$1$};
				\node[left] at (0,4) {$1$};
				\node[left] at (0,5) {$1$};
				\node[left,circle,scale=0.75,draw] at (-0.75,1) {$\overline{x}_1$};
				\node[left] at (-0.75,2.5) {$\vdots$};
				\node[left] at (-0.75,3.5) {$\vdots$};
				\node[left,circle,scale=0.75,draw] at (-0.75,5) {$\overline{x}_N$};
				\node[right] at (6,1) {$j_1$};
				\node[right] at (6,2) {$\vdots$};
				\node[right] at (6,3) {$\vdots$};
				\node[right] at (6,4) {$\vdots$};
				\node[right] at (6,5) {$j_N$};
			}
		\end{align}
		where all vertices are of the type \eqref{R-vert}; this quantity is the probability that a random configuration $\mathfrak{C}$ in the quadrant $[1,N] \times [1,N]$ has state $i_a \in \{0,1\}$ exiting vertically from vertex $(a,N)$ and state $j_a \in \{0,1\}$ exiting horizontally from vertex $(N,a)$, for all $1 \leq a \leq N$. The second is a tower of vertices of the type \eqref{L-vert} (which one may also view as a straightened version of the thick arrow along the northeast boundary in \eqref{hybrid-pf}):
		\begin{align}
			\label{tower}
			H_k(j_1,\dots,j_N;i_1,\dots,i_N)
			=
			\tikz{0.85}{
				\foreach\y in {1,...,6}{
					\draw[lgray,line width=1.5pt,->] (0,\y) -- (2,\y);
				}
				\draw[lgray,line width=4pt,->] (1,0) -- (1,7);
				\node[above] at (1,7) {$k$};
				\node[below] at (1,0) {$k$};
				\node[left] at (0,1) {$j_1$};
				\node[left] at (0,2) {$\vdots$};
				\node[left] at (0,3) {$j_N$};
				\node[left] at (0,4) {$i_N$};
				\node[left] at (0,5) {$\vdots$};
				\node[left] at (0,6) {$i_1$};
				\node[left,circle,scale=0.75,draw] at (-0.75,1) {$\overline{x}_1$};
				\node[left] at (-0.75,2) {$\vdots$};
				\node[left,circle,scale=0.75,draw] at (-0.75,3) {$\overline{x}_N$};
				\node[left,circle,scale=0.75,draw] at (-0.75,4) {$y_N$};
				\node[left] at (-0.75,5) {$\vdots$};
				\node[left,circle,scale=0.75,draw] at (-0.75,6) {$y_1$};
				\node[right] at (2,1) {$0$};
				\node[right] at (2,2) {$0$};
				\node[right] at (2,3) {$0$};
				\node[right] at (2,4) {$1$};
				\node[right] at (2,5) {$1$};
				\node[right] at (2,6) {$1$};
			}
		\end{align}
		In view of the arrow conservation property of the vertices \eqref{L-vert}, it is easy to see that each internal vertical edge within \eqref{tower} admits a unique state such that the tower has non-vanishing weight. This allows us to compute $H_k(j_1,\dots,j_N;i_1,\dots,i_N)$ explicitly:
		\begin{align}
			\label{H-prod}
			H_k(j_1,\dots,j_N;i_1,\dots,i_N)
			=
			\prod_{a=1}^{N}
			\mathcal{L}_{\overline{x}_a}(k+J_{a-1},j_a;k+J_a,0)
			\mathcal{L}_{y_a}(k+I_a,i_a;k+I_{a-1},1),
		\end{align}
		where we have defined the partial sums $J_0 = I_0 = 0$, $J_a = \sum_{b=1}^{a} j_b$, $I_a = a-\sum_{b=1}^{a} i_b$ for $1 \leq a \leq N$, and where we note that $J_N = I_N$. From the table \eqref{rank1-weights}, the weights 
		$\mathcal{L}_{\overline{x}_a}(k+J_{a-1},j_a;k+J_a,0)$ are all equal to $1$; the remaining terms in the product \eqref{H-prod} are given by
		\begin{align*}
			\mathcal{L}_{y_a}(k+I_a,i_a;k+I_{a-1},1)
			=
			y_a \cdot
			\left\{
			\begin{array}{ll}
				1-t^{k+I_a}, & \qquad i_a=0,
				\\
				\\
				1, & \qquad i_a =1.
			\end{array}
			\right.
		\end{align*}
		The product \eqref{H-prod} then simplifies to
		\begin{align*}
			H_k(j_1,\dots,j_N;i_1,\dots,i_N)
			=
			\prod_{a=1}^{N} y_a
			\cdot
			\prod_{b=1}^{I_N} (1-t^{k+b})
			=
			\prod_{a=1}^{N} y_a
			\cdot
			\prod_{b=1}^{J_N} (1-t^{k+b}).
		\end{align*}
		In particular, $H_k(j_1,\dots,j_N;i_1,\dots,i_N)$ depends on $(i_1,\dots,i_N)$ and $(j_1,\dots,j_N)$ only via $I_N = J_N = \sum_{b=1}^{N} j_b$. In fact, returning to the quadrant \eqref{square}, we see that 
		$J_N = \mathfrak{h}(N,N)$; accordingly, one has
		\begin{align*}
			H_k(j_1,\dots,j_N;i_1,\dots,i_N)
			&=
			\prod_{a=1}^{N}
			y_a
			\cdot
			\prod_{b=1}^{\mathfrak{h}(N,N)}
			(1-t^{k+b})
			\\
			&=
			\prod_{a=1}^{N}
			y_a
			\cdot
			\prod_{b=1}^{\infty}
			\frac{1-t^{k+b}}{1-t^{k+b+\mathfrak{h}(N,N)}}
			=
			\prod_{a=1}^{N}
			y_a
			\cdot
			\frac{(t^{k+1};t)_{\infty}}{(t^{k+1+\mathfrak{h}(N,N)};t)_{\infty}}.
		\end{align*}
		Coming back to \eqref{Z-expand}, we then find that 
		\begin{multline*}
			Z(x_1,\dots,x_N;y_1,\dots,y_N;k)
			\\
			=
			\prod_{a=1}^{N}
			y_a
			\cdot
			\sum_{\{i_1,\dots,i_N\} \in \{0,1\}^N}
			\sum_{\{j_1,\dots,j_N\} \in \{0,1\}^N}
			\mathbb{P}_{6{\rm v}}(i_1,\dots,i_N;j_1,\dots,j_N)
			\frac{(t^{k+1};t)_{\infty}}{(t^{k+1+\mathfrak{h}(N,N)};t)_{\infty}}.
		\end{multline*}
		Conditioning on the possible values of $\mathfrak{h}(N,N)$, this may be written as
		\begin{align*}
			Z(x_1,\dots,x_N;y_1,\dots,y_N;k)
			&=
			\prod_{a=1}^{N}
			y_a
			\cdot
			\sum_{m=0}^{N}
			\frac{(t^{k+1};t)_{\infty}}{(t^{k+1+m};t)_{\infty}}
			\mathbb{P}_{6{\rm v}}(\mathfrak{h}(N,N)=m),
		\end{align*}
		which proves the claim \eqref{6v-E}.
	\end{proof}
	
	\subsection{Determinant Evaluation}
	
	Following Section 4.2 and Appendix B of \cite{WZJ}, the partition function \eqref{hybrid-pf} may be computed in closed form:
	\begin{prop}
		For any $N \geq 1$ and $k \geq 0$, one has
		\begin{align}
			\label{Z-det}
			Z(x_1,\dots,x_N;y_1,\dots,y_N;k)
			=
			\frac{\prod_{i=1}^{N}y_i \prod_{1 \leq i,j \leq N} (1-x_i y_j)}
			{\prod_{1 \leq i<j \leq N} (x_i-x_j)(y_i-y_j)}
			\det_{1 \leq i,j \leq N}
			\left[
			\frac{1-t^{k+1}-t(1-t^k) x_i y_j}{(1-x_i y_j)(1-t x_i y_j)}
			\right].
		\end{align}
	\end{prop}
	
	\begin{proof}[Proof (Outline)]
		
		The proof relies on finding a list of properties of $Z(x_1,\dots,x_N;y_1,\dots,y_N;k)$ that determine it uniquely; one then shows that the right hand side of \eqref{Z-det} obeys the same properties. We list these properties below without derivation; for more information, we refer the reader to \cite[Lemma 5]{WZJ}.
		\begin{enumerate}
			\item $Z(x_1,\dots,x_N;y_1,\dots,y_N;k)$ is symmetric in the alphabet $(x_1,\dots,x_N)$ and separately in the alphabet $(y_1,\dots,y_N)$;
			\item $\prod_{1 \leq i,j \leq N} (1-t x_i y_j) \cdot Z(x_1,\dots,x_N;y_1,\dots,y_N;k)$ is a polynomial in $x_N$ of degree $N$;
			\item Setting $x_N = 1/y_N$ we have
			\begin{align*}
				Z(x_1,\dots,x_N;y_1,\dots,y_N;k)
				\Big|_{x_N = 1/y_N}
				=
				y_N
				\cdot
				Z(x_1,\dots,x_{N-1};y_1,\dots,y_{N-1};k);
			\end{align*}
			\item Setting $x_i = 0$ for all $1 \leq i \leq N$, we have
			\begin{align*}
				Z(0,\dots,0;y_1,\dots,y_N;k)
				=
				\prod_{b=1}^{N}
				(1-t^{k+b});
			\end{align*}
			\item For $N=1$, there holds
			\begin{align*}
				Z(x_1;y_1;k)
				=
				y_1 \cdot \left( \frac{1-t^{k+1}-t(1-t^k)x_1 y_1}{1-t x_1 y_1} \right).
			\end{align*}
		\end{enumerate}
	\end{proof}
	
	\subsection{Schur Expectation}
	
	Having computed \eqref{hybrid-pf} in determinant form, it is now easy to pass to its Schur expansion.
	
	\begin{prop}
		Define the following Schur measure with respect to two alphabets $(x_1,\dots,x_N)$ and $(y_1,\dots,y_N)$:
		\begin{align}
			\label{schur-meas}
			\mathbb{P}_{\SM}(\lambda)
			=
			\prod_{1 \leq i,j \leq N}
			(1-x_i y_j)
			\cdot
			s_{\lambda}(x_1,\dots,x_N)
			s_{\lambda}(y_1,\dots,y_N).
		\end{align}
		For fixed $N \geq 1$ and $k \geq 0$ we then have the identity
		\begin{align}
			\label{Z-schur}
			Z(x_1,\dots,x_N;y_1,\dots,y_N;k)
			=
			\prod_{i=1}^{N} y_i
			\cdot
			\mathbb{E}_{\SM}
			\left[
			\prod_{i=1}^{N}
			(1-t^{k+1+\lambda_i-i+N})
			\right],
		\end{align}
		where the expectation is taken with respect to the measure \eqref{schur-meas}.
	\end{prop}
	
	\begin{proof}
		We begin by manipulating the determinant present in \eqref{Z-det}. One has
		\begin{align*}
			\det_{1 \leq i,j \leq N}
			\left[
			\frac{1-t^{k+1}-t(1-t^k) x_i y_j}{(1-x_i y_j)(1-t x_i y_j)}
			\right]
			=
			\det_{1 \leq i,j \leq N}
			\left[
			\frac{1}{1-x_i y_j}
			-
			\frac{t^{k+1}}{1-t x_i y_j}.
			\right],
		\end{align*}
		Replacing the two fractions on the right hand side by their corresponding geometric series, we obtain the identity
		\begin{align*}
			\det_{1 \leq i,j \leq N}
			\left[
			\frac{1-t^{k+1}-t(1-t^k) x_i y_j}{(1-x_i y_j)(1-t x_i y_j)}
			\right]
			=
			\det_{1 \leq i,j \leq N}
			\left[
			\sum_{a=0}^{\infty}
			(1-t^{k+1+a})
			(x_i y_j)^a
			\right].
		\end{align*}
		To the latter we apply the Cauchy--Binet identity, which yields
		\begin{align*}
			\det_{1 \leq i,j \leq N}
			\left[
			\frac{1-t^{k+1}-t(1-t^k) x_i y_j}{(1-x_i y_j)(1-t x_i y_j)}
			\right]
			=
			\sum_{0 \leq a_1 < \cdots < a_N}
			\prod_{i=1}^{N}
			(1-t^{k+1+a_i})
			\det_{1 \leq i,j \leq N}
			\left[ x_i^{a_j} \right]
			\det_{1 \leq i,j \leq N}
			\left[ y_j^{a_i} \right];
		\end{align*}
		including Vandermonde factors on both sides of the equation, and making the change of summation indices $a_i = \lambda_{N-i+1}+i-1$, we then find that
		\begin{multline*}
			\prod_{1 \leq i<j \leq N}
			\frac{1}{(x_i-x_j)(y_i-y_j)}
			\cdot
			\det_{1 \leq i,j \leq N}
			\left[
			\frac{1-t^{k+1}-t(1-t^k) x_i y_j}{(1-x_i y_j)(1-t x_i y_j)}
			\right]
			\\
			=
			\sum_{\lambda_1 \geq \cdots \geq \lambda_N \geq 0}
			\
			\prod_{i=1}^{N}
			(1-t^{k+1+\lambda_i-i+N})
			s_{\lambda}(x_1,\dots,x_N)
			s_{\lambda}(y_1,\dots,y_N).
		\end{multline*}
		This recovers the claim \eqref{Z-schur}, after multiplying through by $\prod_{i=1}^{N} y_i \cdot \prod_{1 \leq i,j \leq N} (1-x_i y_j)$.
	\end{proof}
	
	\subsection{Final Match}
	\label{ssec:final}
	
	Comparing \eqref{6v-E} and \eqref{Z-schur}, we have proved that
	\begin{align}
		\label{final-match}
		\mathbb{E}_{6{\rm v}}
		\left[
		\frac{(t^{k+1};t)_{\infty}}{(t^{k+1+\mathfrak{h}(N,N)};t)_{\infty}}
		\right]
		=
		\mathbb{E}_{\SM}
		\left[
		\prod_{i=1}^{N}
		(1-t^{k+1+\lambda_i-i+N})
		\right].
	\end{align}
	Both sides of \eqref{final-match} are polynomial in $t^{k}$; since \eqref{final-match} holds for all integer values $k \in \mathbb{Z}_{\geq 0}$ and the set $\{t^k\}_{k \geq 0}$ has a point of accumulation for $|t|<1$, we may extend the equality to all complex values by the analytic continuation $t^{k+1} = -\zeta \in \mathbb{C}$. This yields
	\begin{align*}
		\mathbb{E}_{6{\rm v}}
		\left[
		\frac{(-\zeta;t)_{\infty}}{(-\zeta t^{\mathfrak{h}(N,N)};t)_{\infty}}
		\right]
		=
		\mathbb{E}_{\SM}
		\left[
		\prod_{i=1}^{N}
		(1+\zeta t^{\lambda_i-i+N})
		\right],
	\end{align*}
	which is precisely the result \eqref{fvexpectation} in the case $q=t$, $r_i=s_i=t^{-1/2}$ and $M=N$.
	
	Extending this result to $M<N$ may be achieved by taking $y_{M+1} = \cdots = y_N = 0$ in \eqref{final-match}. On the side of the six-vertex model, this choice trivializes the contribution of the final $N-M$ thin vertical lines in the picture \eqref{hybrid-pf} (as no arrows can enter them through the left), leading to a rectangular domain; on the side of the Schur expectation, this choice does not damage the measure \eqref{schur-meas} in view of the stability property of Schur polynomials. A similar reduction is possible in the case $M>N$.

\end{document}